\documentclass[11pt]{article}
\usepackage{csquotes}
\usepackage[numbers]{natbib}
\usepackage{fix-cm}

\usepackage{amsmath}    
\usepackage{graphicx}   
\usepackage{hyperref}   
\usepackage{fancyhdr}   

\usepackage[a4paper, total={6in, 8in}]{geometry}
\usepackage{physics}
\usepackage{amssymb,latexsym,amsthm,amsfonts,mathrsfs}
\usepackage{caption,color,graphicx}
\usepackage{xspace}
\usepackage[american]{babel}
\usepackage{bbm}
\usepackage{enumerate}
\usepackage{verbatim}
\usepackage{mathtools}

\usepackage{stackengine}
\usepackage{xcolor}
\usepackage{braket}
\usepackage{multicol}
\usepackage{float}
\usepackage{enumitem}

\theoremstyle{plain}
\newtheorem{theorem}{Theorem}[section]
\newtheorem{lemma}[theorem]{Lemma}
\newtheorem{proposition}[theorem]{Proposition}
\newtheorem{corollary}[theorem]{Corollary}

\theoremstyle{definition}
\newtheorem{definition}[theorem]{Definition}
\newtheorem{remark}[theorem]{Remark}

\newcommand{\ul}[1]{\underline{#1}}

\newcommand{\N}{\mathbb{N}}
\newcommand{\Z}{\mathbb{Z}}

\newcommand{\R}{\mathbb{R}}
\newcommand{\C}{\mathbb{C}}
\newcommand{\E}{\mathbb{E}}

\newcommand{\calO}{\mathcal{O}}
\newcommand{\calP}{\mathcal{P}}

\newcommand{\calR}{\mathcal{R}}

\newcommand{\bfS}{\textbf{S}\xspace}
\newcommand{\bfU}{\textbf{U}\xspace}
\newcommand{\bfD}{\textbf{D}\xspace}

\newcommand{\id}{\mathrm{id}}
\newcommand{\LP}{\mathrm{LP}}
\newcommand{\RP}{\mathrm{RP}}
\newcommand{\len}{\mathrm{len}}
\newcommand{\RtoL}{\mathrm{R\to L}}
\newcommand{\rmR}{\mathrm{R}}
\newcommand{\rmL}{\mathrm{L}}
\newcommand{\iso}{\mathrm{iso}}
\newcommand{\sym}{\mathrm{sym}}
\newcommand{\alt}{\mathrm{alt}}
\newcommand{\Sp}{\mathrm{Sp}}
\newcommand{\up}{\mathrm{up}}
\newcommand{\down}{\mathrm{down}}
\newcommand{\im}{\mathrm{im}}
\newcommand{\aux}{\mathrm{aux}}

\newcommand{\ie}{i.e.,\ }

\newcommand{\scrF}{\mathscr{F}}

\usepackage{amssymb,latexsym,amsthm,amsmath,amsfonts,mathrsfs}
\usepackage{caption,color,graphicx}
\usepackage{hyperref}
\usepackage{bbm}
\usepackage{verbatim}
\usepackage{stackengine}
\usepackage{xcolor}
\usepackage{braket}
\usepackage[normalem]{ulem}

\newcommand\xof[3][1ex]{\ensurestackMath{%
  \setbox0=\hbox{$#2$}%
  \setbox2=\hbox{$#3$}%
  \kern\wd0\kern-\dimexpr#1\relax%
  \stackinset{r}{#1}{b}{\dimexpr#1-1ex}{\mathrlap{#3}}{%
    \stackinset{l}{#1}{t}{\dimexpr#1-1ex}{\mathllap{#2}}{\bigg/}%
  }%
  \kern-\dimexpr#1\relax\kern\wd2%
}}
\allowdisplaybreaks
\title{\texorpdfstring{\textsc{Root-to-Leaf Path Random Walks, Normalized Hodge\\ Laplacians, and Cheeger Inequalities on\\ Simplicial Complexes}}{Root-to-Leaf Path Random Walks, Normalized Hodge Laplacians, and Cheeger Inequalities on Simplicial Complexes}}
\author{}
\date{}

\makeatletter
\renewcommand{\maketitle}{%
  \begin{center}
    {\fontsize{14.4}{16.8}\selectfont \parbox{0.98\textwidth}{\centering \@title\par}}
    
    \vspace{1.6em}
    
    {\normalsize
    Francesco Vigan\`o\footnotemark[1] \quad
    Tolga Birdal\footnotemark[2] \quad
    Michael T. Schaub\footnotemark[3] \quad
    Mauricio Barahona\footnotemark[4]\par}
  \end{center}
  \footnotetext[1]{Imperial College London, \texttt{vigano.francesco@outlook.com}}
  \footnotetext[2]{Imperial College London, \texttt{t.birdal@imperial.ac.uk}}
  \footnotetext[3]{RWTH Aachen University, \texttt{schaub@netsci.rwth-aachen.de}}
  \footnotetext[4]{Imperial College London, \texttt{m.barahona@imperial.ac.uk}}
  \vspace{1.4em}
  
  \begin{center}
    \textbf{Abstract}
  \end{center}
  \small
  \noindent We introduce root-to-leaf path random walks on double covers of graded signed graphs and analyze their behavior in a general setting. Viewing simplicial complexes within this framework, we show that these walks induce the natural normalization of the coboundary operator and of the Hodge Laplacians while preserving the basic structural features of combinatorial Hodge theory. We then derive Cheeger inequalities for the upper side of the normalized Hodge spectrum, identify the coherent structures governing these bounds, and combine the up- and down-cases into sharper estimates.
\par
  \normalsize
  
  \vspace{1.1em}
}
\makeatother

\begin{document}

\newgeometry{top=1.07in,bottom=1.06in,left=1.1in,right=1.1in}
\linespread{1.}\selectfont
\sloppy

\maketitle

\section*{Introduction}
This paper develops a general theory of root-to-leaf path random walks on double covers of graded signed graphs. Such a structure consists of a graph with a notion of “above” and “below” across nodes (the grading), in which each node has a “flipped” companion (the doubleness), and edges carry a positive or negative sign (the signature). The walker moves either up along positive edges or down along negative ones, with transition probabilities determined by root-to-leaf paths, namely paths connecting down-most nodes (roots) to up-most nodes (leaves). The spectral behavior of the dynamics decomposes into two half-dimensional operators: one corresponding to the unsigned quotient and the other capturing the signed structure of the object. \\

Simplicial complexes provide a natural instance of this framework through a construction closely related to the Hasse diagram of their face poset. Faces act as nodes of the abstract graph, graded by dimension, with orientation determining both the “flipped” counterpart of a face and the signature between a face and its subfaces. In this setting, the induced random walk across dimensions yields a natural normalization of the coboundary operator and the Hodge Laplacians, while preserving key combinatorial properties, including the Hodge decomposition and the spectral correspondence between adjacent up- and down-operators. For the resulting normalized Laplacians, we establish Cheeger inequalities on the upper side of the spectrum, combining the up- and down-cases into sharper bounds. These take the form
\begin{align*}
\frac{\max \left(
\frac{\big(h_{k-1}^\up \big)^2}{k},
\frac{\big(h_k^\down \big)^2}{d_k^\down}
\right)}{2 \cdot (k + 1)}
&\le 1 - \lambda_{\max} \big( \Delta_{k-1}^\up \big) \\
&= 1 - \lambda_{\max} \big( \Delta_k^\down \big) \le
\frac{ 2 \cdot \min \left( h_{k-1}^\up, h_k^\down \right)}{k + 1},
\end{align*}
where $\Delta_{k-1}^\up$ and $\Delta_k^\down$ are the normalized up- and down-Laplacians in dimensions $k-1$ and $k$, $h_{k-1}^\up$ and $h_k^\down$ are the corresponding Cheeger constants, and $d_k^\down$ is the relevant combinatorial down-degree term; the middle quantity is their shared spectral gap from the upper bound $1$. \\

This line of research begins in the simplest setting of a graph with no additional structure, its adjacency matrix, and its combinatorial (standard) Laplacian $L$. The smallest eigenvalue $\lambda_{\min}(L)$ is always zero, and the second smallest eigenvalue $\lambda_{\min + 1}(L)$ is zero if and only if the graph is disconnected. More generally, the multiplicity of $0$ as an eigenvalue of $L$ corresponds to the number of connected components of the graph. Intuitively, if the graph is connected, but nearly disconnected with respect to some metric, then $\lambda_{\min + 1}(L)$ will be close to zero. This metric is the Cheeger constant of a graph, and the relationship between this topological quantity and $\lambda_{\min + 1}(L)$ can be quantified using a Cheeger inequality. \\

Cheeger inequalities were first explored in 1970, in the framework of differential manifolds \cite{cheeger1970}, relating the spectral gap of the Laplace-Beltrami operator to the isoperimetric constant of a manifold. Thereafter, several works in the 80's extended Cheeger's result to the discrete setting of graphs \cite{dodziuk1984,alon1985,alon1986eigenvalues,mohar1989isoperimetric}, relating the spectra of the combinatorial Laplacian $L$ and normalized Laplacian $\Delta$ to the Cheeger cut problem, that is, partitioning the nodes of a graph into two components that have a small proportion of inter-class edges (with respect to the same metric that defines the Cheeger constant). Cheeger inequalities in this setting are now considered fundamental results in spectral graph theory. \\

More specifically, a two-sided Cheeger inequality quantifies the lower spectral gap of $L$ (\ie how far $\lambda_{\min + 1}(L)$ is from zero) in terms of the Cheeger constant of a graph. Furthermore, the proof of this inequality provides a sub-optimal solution to the Cheeger cut problem. Finding the optimal solution to the Cheeger cut problem is an NP-hard problem, and this algorithm of polynomial complexity provides an approximation obtained from the eigenfunction associated with $\lambda_{\min + 1}(L)$. \\

In addition to graph clustering, spectral graph theory prolifically serves other applications, including graph embeddings \cite{ng2001spectral,belkin2001laplacian}, random walks on graphs \cite{lovasz1993random}, and learning models \cite{digiovanni2023understanding,kipf2016semi,defferrard2016convolutional,huang2026hogdiff}. For graph embeddings, the eigenfunctions associated with low eigenvalues provide an optimal embedding in a Euclidean space with respect to the Dirichlet energy of a graph. In graph random walks, Cheeger inequalities characterize how the graph topology (in particular, how close to disconnected or bipartite a graph is) obstructs mixing over the nodes. Finally, high and low frequencies of the graph (\ie eigenfunctions associated with high and low eigenvalues of the normalized Laplacian) are relevant in the message passage dynamics of graph-based learning models. \\

It is important to note that the Laplacian considered in applications is often the normalized Laplacian $\Delta$ and a spectral theory and Cheeger inequalities for this operator are more appropriate in this context. Additionally, normalizing the Laplacian operator enables detecting topological properties that are independent of node degrees and regularity conditions of the graph. For instance, the spectrum of $\Delta$ is bounded between $0$ and $2$, and the value $2$ is an eigenvalue of $\Delta$ if and only if the graph admits a bipartite connected component. In a similar way, there is a Cheeger inequality relating the spectral gap between $\lambda_{\max}(\Delta)$ and the value $2$ to a metric that quantifies how far the graph is from admitting a bipartite component. An elegant extension is provided by \cite{atay2020} in the context of signed graphs, where the relevant topological structures are balanced and antibalanced components (and such properties are dual to each other). \\

In recent years, there has been an increasing effort to investigate the aforementioned applications in the context of simplicial complexes, a natural higher-order analog of graphs~\cite{bick2023higher}. This is motivated by several reasons, such as overcoming the limitations of pair-wise interaction, or handling geometric datasets that are inherently higher-dimensional. In particular, in \cite{schaub2020} the authors analyze the behavior of random walks on the edges of a simplicial complex, incorporating additional aspects of the topology of these objects, rather than graphs. Using spectral properties of the Hodge Laplacian, clustering algorithms are proposed in \cite{krishnagopal2021spectral}, which extract and predict higher-order interactions, revealing simplicial communities across various datasets. Simplicial models have also been studied from a geometric perspective, in the context of hyperbolic network geometry \cite{bianconi2017emergent}. Weighted simplicial complexes have further been introduced \cite{baccini2022weighted}, overcoming the limitations of unweighted models by preserving information while still capturing higher-order topology. In \cite{bodnar2021weisfeiler}, the authors propose a framework that extends the Weisfeiler-Lehman graph isomorphism algorithm to simplicial complexes, enabling message passing on higher-order topological structures for improved representation learning. We refer the reader to \cite{hajij2022topological,papamarkou2024position} for a detailed review of higher-order interaction learning models, including architectures that are specifically built on simplicial complexes. \\

It is then natural to investigate a spectral theory of simplicial complex operators, seeking a generalization of the classic results on graphs. Hodge Laplacians $L_k$ can be defined on a simplicial complex through its boundary operators, where the index $k$ denotes the face dimension. These operators encode meaningful properties of the simplicial complex -- such as its (co-)homology, through the Hodge decomposition of $L_k$. The Hodge Laplacian splits into its up- and down-components as $L_k = L_k^\up + L_k^\down$. On the nodes of a graph, when $k=0$, the down-component is reduced to the zero operator. One can restrict oneself to the spectral analysis of up-Laplacians, since the non-zero spectrum of $L_k^\down$ coincides with the non-zero spectrum of $L_{k-1}^\up$. The zero eigenvalue represents an exception, and its study is strictly linked to the (co-)homology theory of the simplicial complex. \\

Motivated by the applications of this spectral theory, it is of interest to consider a proper normalization of Hodge Laplacians. Generalizing the graph framework to simplicial complexes, a good normalization procedure should be obtained based on some notion of face degree, it should bound the spectrum of such normalized operators uniformly over all simplicial complexes, and, importantly, it should ensure that spectral gaps are saturated if and only if the simplicial complex of interest possesses certain topological properties (which are independent of face degree and regularity conditions). \\

In order to achieve this, we develop the theory of root-to-leaf path random walks on double covers of graded signed graphs in Section \ref{chapter.root_to_leaf}. We show how the transition matrix of this random walk can be understood through two half-dimensional operators: one corresponding to a random walk on the quotient graph that discards the signature; the other being a signed operator that encodes more geometric information related to the orientability of the discrete structure. Crucially, simplicial complexes, together with the inclusion structure of their faces and the orientation defined by the boundary operator, constitute a natural example of such double covers. Delightfully, we find that the second operator is related to the normalized Laplacian $\Delta_k = \Delta_k^\up + \Delta_k^\down$. More generally, root-to-leaf path random walks on simplicial complexes reveal the elegant normalization procedure for the coboundary operator introduced in \cite{horak2013spectra}. We explore the specialization of the root-to-leaf path random walk to simplicial complexes and their normalized Laplacians in the first part of Section \ref{chapter.norm_lapl_cheeger}. \\

As for graphs, the spectra of the normalized up- and down-Laplacians are uniformly bounded. Specifically, the operators $\Delta_k^\up$ and $\Delta_k^\down$ are uniformly bounded from above by the value $1$, since they are derived from conditional random up- and down-walks (variants of root-to-leaf path random walks). Thus, it is natural to ask what combinatorial/topological properties of the simplicial complex correspond to lower and upper bounds of these spectra. Moreover, one could seek to obtain Cheeger inequalities that quantify, in terms of the spectra of these operators, how far the simplicial complex is from admitting such structures. This is explored in the last part of Section \ref{chapter.norm_lapl_cheeger}. In contrast, the lower side of the spectrum is more challenging to analyze: the lower bound is $0$, Cheeger constants are more involved, and the relevant topological structures are related to the (co-)homology of the simplicial complex. In \cite{jost2023}, the authors present a possible definition of Cheeger constants for the lower bound of an alternative normalization of the up-Laplacian. In this work, we limit ourselves to the analysis of the upper side of the spectrum. We prove Cheeger inequalities for the upper side of the spectrum by leveraging the theory of signed graphs, following an idea of \cite{jost2023}, and the topological structures of interest are referred to in our work as coherent-up- and -down-components (generalizing the notion of bipartiteness). By exploiting the relation between the spectra of normalized up- and down-Laplacians in different dimensions, we are able to merge two Cheeger inequalities into more effective and tighter bounds. We also provide a tractable example to show the effectiveness of such combined Cheeger inequalities. \\

Section \ref{chapter.root_to_leaf} introduces the general theory of root-to-leaf path random walks on double covers of graded signed graphs, together with the quotient and signed operators that govern their dynamics. Section \ref{chapter.norm_lapl_cheeger} then specializes this framework to simplicial complexes, derives the induced normalization of the coboundary operator and of the Hodge Laplacians, and establishes the Cheeger inequalities for the upper side of the spectrum, including the combined bounds and an explicit example.

\section{Root-to-leaf path random walks}\label{chapter.root_to_leaf}
\begingroup
\let\section\subsection
\let\subsection\subsubsection
\let\subsubsection\paragraph
In this section, we develop the theory of root-to-leaf path random walks in a general setting. We first outline a simpler case to provide some intuition (Section \ref{section.signed_graphs_classic_random_walk}). We then introduce the objects the random walker walks on: double covers of graded signed graphs (Section \ref{section.double_covers_graded_signed_graphs}). We describe leaves and roots (Section \ref{section.leaves_roots}), and quotient and cover components (Section \ref{section.components}) of such objects. We then move on to describing the rules of the root-to-leaf path random walk (Section \ref{section.rules_random_walk}), determining its stationary distribution (Section \ref{section.stationary_distribution}), and studying the operators associated with this random process (Section \ref{section.associated_operators}). Finally, we introduce two important variants, the conditional random up- and down-walks, and study their convergence as random processes (Section \ref{section.conditional_random_walks}). \\

For an introduction to spectral graph theory, we refer the reader to the comprehensive book \cite{chung1997spectral}. For random walks on graphs, we recommend the definitive survey \cite{lovasz1993random}, along with \cite{burioni2005random} and \cite{sarkar2011random}. Signed graphs were first introduced in \cite{harary1953notion}, and we also refer to \cite{zaslavsky1982signed}. Random walks on signed graphs were studied in \cite{jung2016personalized}.

\section{A simpler case: random walks on signed graphs}\label{section.signed_graphs_classic_random_walk}

\begin{figure}[H]
    \centering
    \includegraphics[scale=0.8]{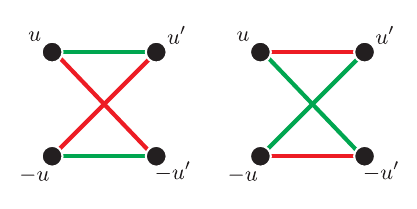}
    \caption{Given an edge $(u, u')$ in $\Gamma^0$, the four corresponding edges in the double cover are drawn. On the left, the original edge $(u, u')$ has positive signature (colored in green); on the right, it has negative signature (colored in red).}
    \label{figure.edge_construction}
\end{figure}

Let $\Gamma^0 = (X^0, E^0, s^0)$ be an undirected signed graph. This means that $X^0$ is the finite set of nodes of $\Gamma^0$, $E^0$ is the set of undirected edges on $X^0$, and $s^0 \colon E^0 \to \{\pm 1\}$ is the edge signature function, satisfying $s^0(u, u') = s^0(u', u)$. From this, we can construct a duplicate $-X^0$ of the set of nodes $X^0$, and consider the disjoint union $X = X^0 \amalg (-X^0)$. We think of $-u$ as the upside-down version of $u$, or the node $u$ with reversed orientation. We extend $E^0$ to $E$ by saying that, given an edge $(u,u') \in E^0$, then $(u,u'), (-u,u'), (u,-u'), (-u,-u')$ are all (undirected) edges in $E$. Finally, we extend the signature $s^0$ to $s \colon E \to \{\pm 1\}$, in the unique way satisfying $s(-u, u') = s(u, -u') = - s(u, u')$. We refer to the signed graph $\Gamma = (X, E, s)$ as the double cover of $\Gamma^0$. The double cover $\Gamma$ contains twice as many nodes, and four times as many edges, compared to $\Gamma^0$. We can also consider the quotient $\ul{\Gamma} = (\ul{X}, \ul{E})$, that is obtained by identifying $u$ and $u'$, collapsing them into a single quotient node $\ul{u} = \ul{-u}$. The quotient $\ul{\Gamma}$ contains the same number of nodes and edges of $\Gamma^0$. Notice that $\ul{\Gamma}$ differs from $\Gamma^0$, in the sense that $\ul{\Gamma}$ is not a signed graph. In Figure \ref{figure.edge_construction} and Figure \ref{figure.double_cover_not_graded} we provide a visualization of the double cover construction.

\begin{figure}[H]
    \centering
    \includegraphics[scale=0.8]{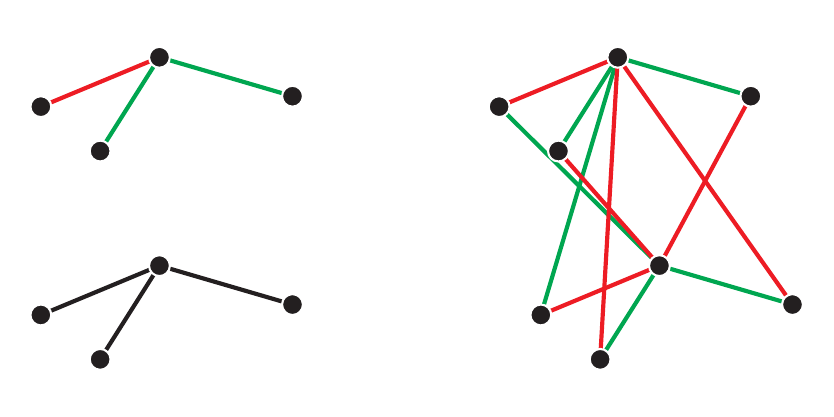}
    \caption{Example of construction of a double from a signed graph. On the top-left corner, the original signed graph $\Gamma^0$ (edges of positive signature are in green, edges of negative signature are in red). On the right, the constructed double cover $\Gamma$. On the bottom-left corner, the quotient graph $\ul{\Gamma}$ (unsigned edges are in black).}
    \label{figure.double_cover_not_graded}
\end{figure}

It turns out that the natural extension of the classic random walk from unsigned to signed graphs, which takes into account the signature function of $\Gamma^0$, has $X$ as state space, rather than $X^0$. The one-step rule for this random walk is as follows (assuming that $\Gamma^0$ has no isolated nodes). Suppose that a walker is at a node $u \in X$. Then, they choose uniformly at random one node among those neighbors $u' \in X$ of $u$ that satisfy $s(u, u') = 1$. In Figure \ref{figure.one_step_random_walk_not_graded}, we sketch an example of such one-step update. For this random walk, the spectrum of the transition matrix $P^\Gamma$ can be described in terms of two half-dimensional operators. The former is $A^{\ul{\Gamma}}$, the symmetric normalized adjacency matrix of the unsigned quotient graph $\ul{\Gamma}$, or equivalently $P^{\ul{\Gamma}}$, the transition matrix of the induced random walk on $\ul{\Gamma}$. The latter is $A^0$, the symmetric normalized adjacency matrix of the signed graph $\Gamma^0$. \\

\begin{figure}[H]
    \centering
    \includegraphics[scale=0.62]{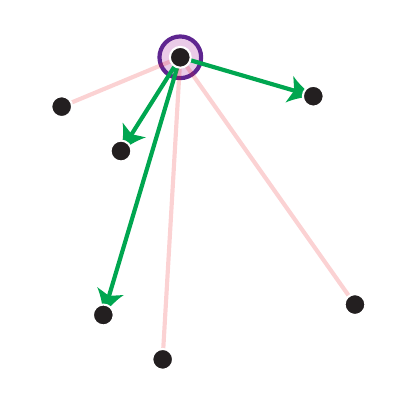}
    \caption{A one-step update of the random walk on the double cover $\Gamma$. The walker is currently at the node circled in purple. They are allowed to move only along edges of positive signature (colored in green). Forbidden edges (of negative signature, colored in red) are opaque.}
    \label{figure.one_step_random_walk_not_graded}
\end{figure}

From this description, we can observe two phenomena. Figure \ref{figure.four_relevant_cases_not_graded} provides an intuitive visualization of the interaction between the random walk behavior and the relevant geometric properties.

 \begin{figure}[H]
    \centering
    \includegraphics[scale=0.55]{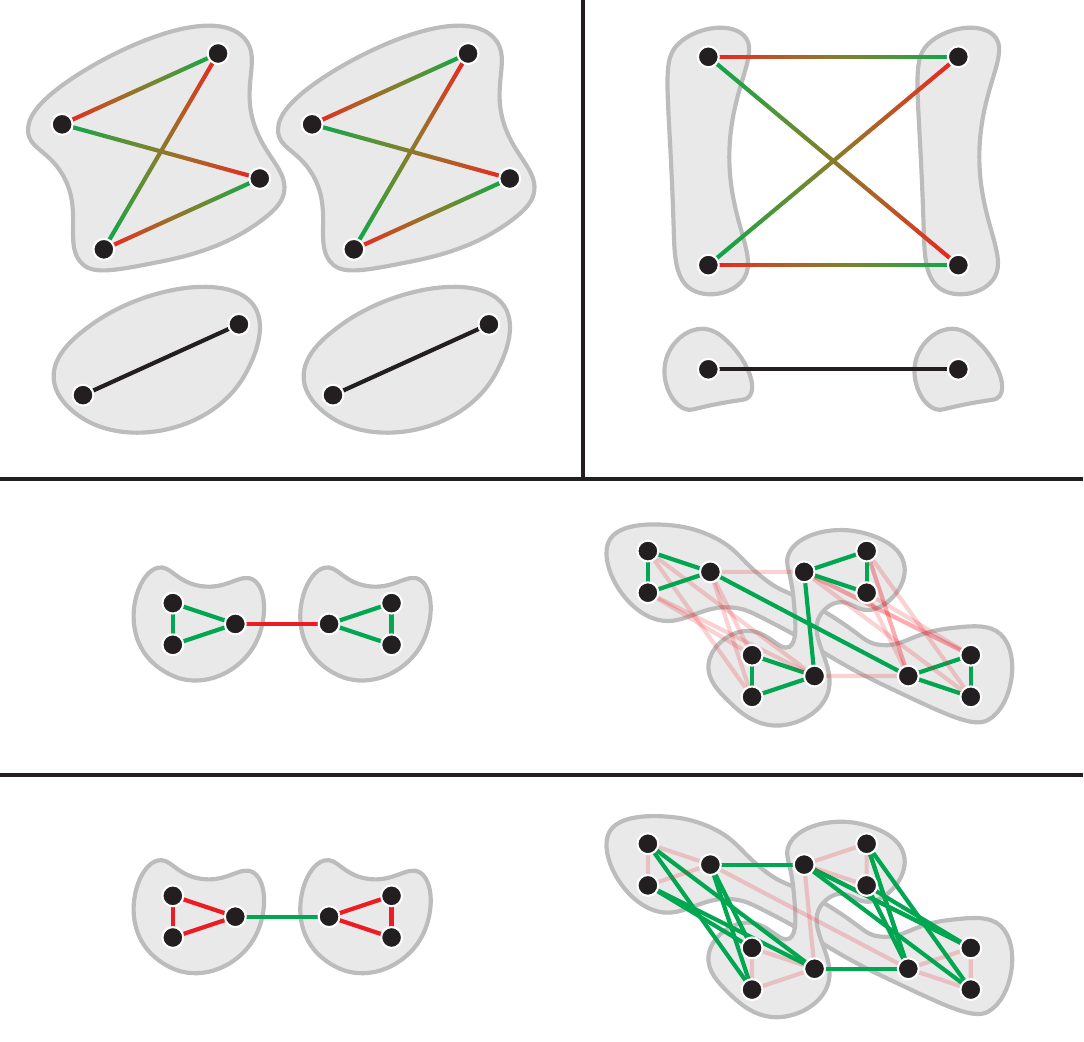}
    \caption{The four topological/combinatorial structures governing root-to-leaf random walks on double covers and the extremal eigenvalues of the associated operators: disconnected quotient, bipartite quotient, balanced signed graph, and antibalanced signed graph. The first and third cases yield an extra eigenvalue $1$, whereas the second and fourth yield an extra eigenvalue $-1$.}
    \label{figure.four_relevant_cases_not_graded}
\end{figure}

First, a random process on the double cover $\Gamma$ is spectrally described by operators associated with the half-dimensional geometric spaces $\ul{\Gamma}$ and $\Gamma^0$. Second, as it is known, the extremal eigenvalues of the operators $A^{\ul{\Gamma}}$ and $A^0$ are related to geometric properties of the spaces $\ul{\Gamma}$ and $\Gamma^0$, respectively. In particular, the behavior of the considered random walk on $\Gamma$ is influenced by the number of connected components and bipartite connected components of $\ul{\Gamma}$, as well as the number of balanced and antibalanced components of $\Gamma^0$.

We hope that the reader bears in mind this analogy throughout the rest of the chapter, which follows the same philosophy in its investigation.

\section{Double covers of graded signed graphs}\label{section.double_covers_graded_signed_graphs}

\begin{definition}\label{definition.double_covers_graded_signed_graphs}
A \textbf{double cover $\Gamma$ of a graded signed graph} is the datum of $\Gamma = (X, E, [ \ : \ ], \dim, -)$, where:
\begin{itemize}
    \item $X$ is a finite non-empty set of \textbf{vertices}, or \textbf{nodes}.
    \item $E$ is a set of directed \textbf{edges} on $X$. If a directed edge joins $u$ to $v$, we write $u \subset v$, or equivalently $v \supset u$.
    \item $[ \ : \ ] \colon E \to \{\pm 1\}$ is a \textbf{signature} function defined on $E$. Given a directed edge $u \subset v$, we denote by $[v : u]$ its signature. We say that $u$ and $v$ are \textbf{coherently oriented} if $[v : u] = 1$, and \textbf{not coherently oriented} if $[v : u] = -1$.
    \item $\dim \colon X \to \Z$ is a \textbf{grading} on $X$. We refer to $\dim(u)$ as the \textbf{dimension} of a vertex $u$. A grading defines a partition $X = \amalg_{k \in \Z} X_k$, where $X_k$ is the subset of $X$ containing the vertices of dimension $k$ (possibly empty).
    \item $- \colon X \to X$ is an \textbf{involution} of $X$ with no fixed points. For a given $u \in X$, we say that $-u$ has \textbf{opposite} or \textbf{reversed orientation} compared to $u$. We say that $u$ and $-u$ constitute an \textbf{involutory pair} of nodes.
\end{itemize}
Moreover, these properties must be satisfied:
\begin{itemize}
    \item Compatibility between $E$ and $\dim$: if $u \subset v$, then $\dim(u) < \dim(v)$.
    \item Compatibility between $E$ and $-$: if $u \subset v$, then also $-u \subset v$ and $u \subset -v$ (and, therefore, $-u \subset -v$).
    \item Compatibility between $[ \ : \ ]$ and $-$: if $u \subset v$, then $[-v : u] = [v : -u] = - [v : u]$ (and, therefore, $[-v : -u] = [v : u]$).
    \item Compatibility between $\dim$ and $-$: $\dim(-u) = \dim(u)$ for any $u \in X$.
\end{itemize}
\end{definition}

\begin{definition}\label{definition.strongly_graded}
    We say that $\Gamma$ is a \textbf{double cover of a strongly graded signed graph}, or that the grading $\dim$ is \textbf{strong}, if, for any edge $u \subset v$, $\dim(v) = \dim(u) + 1$.
\end{definition}

\begin{definition}\label{definition.involutory_quotient}
    The \textbf{involutory quotient} $\ul{\Gamma} = \Gamma / \pm$ of $\Gamma = (X, E, [ \ : \ ], \dim, -)$ is the graded directed unsigned graph given by $\ul{\Gamma} = (\ul{X}, \ul{E}, \dim)$, where:
    \begin{itemize}
        \item $\ul{X}$ is the quotient $X/\pm$. We denote by $\ul{u}$ the equivalence class of $u$ and $-u$.
        \item $\ul{E}$ is the set of directed edges on $\ul{X}$ induced by $E$: a directed edge joins $\ul{u}$ to $\ul{v}$, and we write $\ul{u} \subset \ul{v}$, if and only if $u \subset v$.
        \item $\dim \colon \ul{X} \to \Z$ is induced by the grading $\dim$ on $\Gamma$: $\dim(\ul{u}) = \dim(u)$.
    \end{itemize}
    We denote by $\ul{ \ \cdot \ } \colon \Gamma \to \ul{\Gamma}$ the projection onto the quotient, that is invariant under pre-composition with the involution of $\Gamma$: $\ul{ \ \cdot \ } \circ - = \ul{ \ \cdot \ }$. We also denote by $N$ the number of nodes in $\ul{X}$, so that $2N$ is the number of nodes in $X$. Similarly, we denote by $N_k$ the number of nodes in $\ul{X}_k$, so that $2N_k$ is the number of nodes in $X_k$.
\end{definition}

\begin{definition}\label{definition.orientation}
    An \textbf{orientation} on $\Gamma$ is a graph homomorphism $\calO \colon \ul{\Gamma} \to \Gamma$ that is a section of the projection $\ul{ \ \cdot \ }$, that is, $\ul{ \ \cdot \ } \circ \calO = \id_{\ul{\Gamma}}$. In particular, the image $\calO(\ul{\Gamma})$ is a graded directed signed subgraph of $\Gamma$. For each $\ul{u} \in \ul{X}$, $\calO(\ul{u})$ is either $u$ or $-u$, and the unsigned edge $\ul{u} \subset \ul{v}$ is mapped by $\calO$ to the signed edge $\calO(\ul{u}) \subset \calO(\ul{v})$.
\end{definition}

\begin{remark}
    There are exactly $2^N$ possible orientations on $\Gamma$.
\end{remark}

\begin{remark}\label{remark.definition_from_signed_graph}
    All the signed graphs obtained through orientations $\calO \colon \ul{\Gamma} \to \Gamma$ are switching-equivalent to each other, according to \cite[Definition 2.1]{atay2020}. Technically, $\Gamma$ is a double cover of $\ul{\Gamma}$ in an ``unsigned'' way, and not of $\calO(\ul{\Gamma})$ -- as there is no natural map from $\Gamma$ to $\calO(\ul{\Gamma})$ that well behaves on the signature. More precisely, $\Gamma$ is rather the space of all switching-equivalent realizations of a given signed graph $\calO(\ul{\Gamma})$, obtained by embedding $\ul{\Gamma}$ into $\Gamma$ via sections of the projection. Despite this, we took the liberty to refer to $\Gamma$ as a double cover of a signed graph. All the signed graphs $\calO(\ul{\Gamma})$, as $\calO$ varies, are spectrally equivalent to each other, as clarified in Remark \ref{remark.flip_function_calO}.
\end{remark}

\section{Leaves and roots}\label{section.leaves_roots}

\begin{definition}\label{definition.leaf_root}
    A \textbf{leaf} of $\ul{\Gamma}$ is a vertex $\ul{u}$ such that there is no other vertex $\ul{v}$ satisfying $\ul{v} \supset \ul{u}$. We denote by $\rmL(\ul{\Gamma})$ the subset of $\ul{X}$ containing the leaves of $\ul{\Gamma}$. Dually, a \textbf{root} of $\ul{\Gamma}$ is a vertex $\ul{u}$ such that there is no other vertex $\ul{t}$ satisfying $\ul{t} \subset \ul{u}$. We denote by $\rmR(\ul{\Gamma})$ the subset of $\ul{X}$ containing the roots of $\ul{\Gamma}$.
\end{definition}

We also denote by $\rmL(\Gamma)$ and $\rmR(\Gamma)$ the sets of nodes of $\Gamma$ whose images via the quotient map $\ul{ \ \cdot \ }$ are leaves and roots, respectively. However, we reserve the terms ``leaf'' and ``root'' for nodes of $\ul{\Gamma}$.

\begin{definition}\label{definition.leaf_root_path_function}
    We define the \textbf{leaf-path function} $\LP \colon \ul{X} \to \N$ of $\ul{\Gamma}$ recursively from above:
    \[
    \LP(\ul{u}) =
    \begin{cases}
        \sum_{\ul{v} \supset \ul{u}} \LP(\ul{v}) & \text{if $\ul{u}$ is not a leaf}, \\
        1 & \text{if $\ul{u}$ is a leaf}.
    \end{cases}
    \]
    Dually, we define the \textbf{root-path function} $\RP \colon \ul{X} \to \N$ of $\ul{\Gamma}$ recursively from below:
    \[
    \RP(\ul{u}) =
    \begin{cases}
        \sum_{\ul{t} \subset \ul{u}} \RP(\ul{t}) & \text{if $\ul{u}$ is not a root}, \\
        1 & \text{if $\ul{u}$ is a root}.
    \end{cases}
    \] 
\end{definition}
\begin{definition}\label{definition.ascending_descending_path}
    An \textbf{ascending path} $\eta$ from $\ul{w}$ to $\ul{w'}$ is a sequence of vertices $\ul{w} = \ul{w}_0, \ul{w}_1, \dots, \ul{w}_n = \ul{w'}$ of $\ul{\Gamma}$ such that $\ul{w}_{i-1} \subset \ul{w}_i$ for all $i = 1, \dots, n$. The value $n$ is the \textbf{length} of $\eta$, denoted by $\len(\eta)$, that is, the number of vertices contained in $\eta$ minus one. We denote by $\calP^\uparrow(\ul{w}, \ul{w'})$ the set of ascending paths from $\ul{w}$ to $\ul{w'}$, and by $\calP^\uparrow(\ul{w})$ the set of ascending paths from $\ul{w}$ to some leaf $\ul{\ell}$ of $\ul{\Gamma}$. In other words,
        \[
        \calP^\uparrow(\ul{w}) = \bigcup_{\ul{\ell} \in \rmL(\ul{\Gamma})} \calP^\uparrow(\ul{w}, \ul{\ell}).
    \]
    Dually, a \textbf{descending path} $\eta$ from $\ul{w}$ to $\ul{w'}$ is a sequence of vertices $\ul{w} = \ul{w}_0, \ul{w}_1, \dots, \ul{w}_n = \ul{w'}$ of $\ul{\Gamma}$ such that $\ul{w}_{i-1} \supset \ul{w}_i$ for all $i = 1, \dots, n$. We denote by $\calP^\downarrow(\ul{w}, \ul{w'})$ the set of descending paths from $\ul{w}$ to $\ul{w'}$, and by $\calP^\downarrow(\ul{w})$ the set of descending paths from $\ul{w}$ to some root $\ul{r}$ of $\ul{\Gamma}$. In other words,
    \[
    \calP^\downarrow(\ul{w}) = \bigcup_{\ul{r} \in \rmR(\ul{\Gamma})} \calP^\downarrow(\ul{w}, \ul{r}).
    \]
    We also denote by $\calP_{\RtoL}$ the set of \textbf{root-to-leaf paths}, \ie ascending paths $\eta$ from $\ul{r}$ to $\ul{\ell}$, for some root $\ul{r}$ and leaf $\ul{\ell}$ of $\ul{\Gamma}$:
    \[
    \calP_{\RtoL} = \bigcup_{\substack{\ul{r} \in \rmR(\ul{\Gamma}) \\ \ul{\ell} \in \rmL(\ul{\Gamma})}} \calP^\uparrow(\ul{r}, \ul{\ell}).
    \]
\end{definition}

\begin{remark}\label{remark.paths_of_length_zero_allowed}
    Ascending or descending paths of length zero (that is, consisting of only one vertex) are allowed.
\end{remark}

\begin{remark}\label{remark.ascending_paths_descending_paths}
    There is a natural bijection between $\calP^\uparrow(\ul{w}, \ul{w'})$ and $\calP^\downarrow(\ul{w'}, \ul{w})$, and root-to-leaf paths can also be seen as descending paths from leaves to roots of $\ul{\Gamma}$.
\end{remark}
\begin{lemma}\label{lemma.number_of_paths_alternative_computation_leaf_root_path_function}
    $\LP(\ul{u})$ is the number of ascending paths from $\ul{u}$ to some leaf $\ul{\ell}$ of $\ul{\Gamma}$. Dually, $\RP(\ul{u})$ is the number of descending paths from $\ul{u}$ to some root $\ul{r}$ of $\ul{\Gamma}$.
\end{lemma}

\begin{figure}[H]
    \centering
    \includegraphics[scale=0.8]{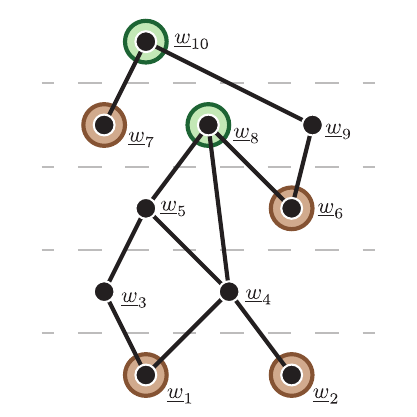}
    \caption{Example of quotient $\ul{\Gamma}$ of a double cover. Edges are unsigned (colored in black). Nodes are separated by dashed gray lines based on their dimension. The grading is not strong, since $\ul{w}_4 \subset \ul{w}_8$ and their dimensions are not consecutive integers. Roots are circled in brown, and leaves are circled in green.}
    \label{figure.LP_and_RP}
\end{figure}

\begin{table}[H]
    \centering
    \begin{tabular}{c|c|c||c|c|c}
        {node $\ul{w}$} & {$\LP(\ul{w})$} & {$\RP(\ul{w})$} & {node $\ul{w}$} & {$\LP(\ul{w})$} & {$\RP(\ul{w})$} \\
        \hline
        $\ul{w}_1$ & $ 3$ & $ 1$ & $ \ul{w}_6$ & $ 2$ & $ 1$ \\
        $\ul{w}_2$ & $ 2$ & $ 1$ & $ \ul{w}_7$ & $ 1$ & $ 1$ \\
        $\ul{w}_3$ & $ 1$ & $ 1$ & $ \ul{w}_8$ & $ 1$ & $ 6$ \\
        $\ul{w}_4$ & $ 2$ & $ 2$ & $ \ul{w}_9$ & $ 1$ & $ 1$ \\
        $\ul{w}_5$ & $ 1$ & $ 3$ & $ \ul{w}_{10}$ & $ 1$ & $ 2$ \\
    \end{tabular}
    \caption{Computation of the values $\LP(\ul{w})$ and $\RP(\ul{w})$ for $\ul{\Gamma}$ in Figure \ref{figure.LP_and_RP}. The values can be obtained via the definition of $\LP$ or Lemma \ref{lemma.number_of_paths_alternative_computation_leaf_root_path_function}. \\}
    \label{table.LP_and_RP}
\end{table}

\begin{proof}
    If $\ul{u}$ has maximal dimension, then it is a leaf, and therefore $\LP(\ul{u}) = 1 = |\calP^\uparrow(\ul{u})|$. By descending strong induction on $\dim(\ul{u})$, we then have that
    \[
    \LP(\ul{u}) = \sum_{\ul{v} \supset \ul{u}} \LP(\ul{v}) = \sum_{\ul{v} \supset \ul{u}} | \calP^\uparrow(\ul{v}) | = | \calP^\uparrow(\ul{u}) |,
    \]
    since the recursive property is satisfied for ascending paths by taking their first step. The proof is dual for $\RP(\ul{u})$.
\end{proof}

\begin{corollary}\label{corollary.product_LP_RP_is_root_to_leaf_paths}
    The product $\LP(\ul{u}) \cdot \RP(\ul{u})$ coincides with the number of root-to-leaf paths passing through $\ul{u}$.
\end{corollary}
\begin{proof}
    This follows from Lemma \ref{lemma.number_of_paths_alternative_computation_leaf_root_path_function} and the fact that the set of root-to-leaf paths passing through $\ul{u}$ is naturally in bijection with the product
    \[
    \calP^\uparrow(\ul{u}) \times \calP^\downarrow(\ul{u}).
    \qedhere
    \]
\end{proof}
In Figure \ref{figure.LP_and_RP}, we provide an example of the structure of the quotient $\ul{\Gamma}$, and in Table \ref{table.LP_and_RP} we summarize the values of the leaf-path and root-path functions of $\Gamma$.

\section{Quotient and cover components}\label{section.components}

\begin{definition}\label{definition.isolated}
    We say that a node $\ul{u} \in \ul{X}$ is \textbf{isolated} if it is both a leaf and a root, and we extend the definition to a node $u \in X$ when its projection $\ul{u}$ is isolated. We denote by $\ul{X}^\iso$ and $X^\iso$ the sets of isolated nodes of $\ul{\Gamma}$ and $\Gamma$, respectively. We say that a pair $\{u, u'\}$ of isolated nodes of $\Gamma$ is an \textbf{involutory pair of isolated nodes} if $u'=-u$.
\end{definition}

\begin{definition}\label{definition.quotient_component}
    A \textbf{quotient-component} $\ul{C}$ is a non-empty subset of $\ul{X}$ that is a connected component of $\ul{\Gamma}$. This notion includes isolated nodes of $\ul{\Gamma}$.
\end{definition}

\begin{remark}\label{remark.components_O}
    The quotient-components of $\ul{\Gamma}$ are in bijection with the connected components of the signed graph $\calO(\ul{\Gamma})$, for any orientation $\calO \colon \ul{\Gamma} \to \Gamma$, through the orientation map $\calO$ itself.
\end{remark}

\begin{definition}\label{definition.cover_component}
    A \textbf{cover-component} $C$ is either a non-empty subset of $X \setminus X^\iso$ that is a connected component of $\Gamma$, or an involutory pair of isolated nodes in $X^\iso$.
\end{definition}

In the following definitions of this section, we assume that the grading $\dim$ is strong.

\begin{definition}\label{definition.up-adjacent}
    Two nodes $\ul{u}, \ul{u'} \in \ul{X}_k$ with $\ul{u} \neq \ul{u'}$ of the same dimension are said to be \textbf{up-adjacent} if there is a $\ul{v} \in \ul{X}_{k+1}$ such that $\ul{u}, \ul{u'} \subset \ul{v}$. In this case, we write $\ul{u} \sim^\up \ul{u'}$. If we drop the condition $\ul{u} \neq \ul{u'}$, we write instead $\ul{u} \sim^\up_= \ul{u'}$, that is: $\ul{u}$ and $\ul{u'}$ are either up-adjacent, or $\ul{u} = \ul{u'}$ is not a leaf. The definitions of up-adjacent extends analogously to the double cover $\Gamma$.
\end{definition}

\begin{definition}\label{definition.down-adjacent}
    Two nodes $\ul{u}, \ul{u'} \in \ul{X}_k$ with $\ul{u} \neq \ul{u'}$ of the same dimension are said to be \textbf{down-adjacent} if there is a $\ul{t} \in \ul{X}_{k-1}$ such that $\ul{u}, \ul{u'} \supset \ul{t}$. In this case, we write $\ul{u} \sim^\down \ul{u'}$. If we drop the condition $\ul{u} \neq \ul{u'}$, we write instead $\ul{u} \sim^\down_= \ul{u'}$, that is: $\ul{u}$ and $\ul{u'}$ are either down-adjacent, or $\ul{u} = \ul{u'}$ is not a root. The definition of down-adjacent extends analogously to the double cover $\Gamma$.
\end{definition}

\begin{definition}\label{up-connected}
    A non-empty subset $\ul{C}_k$ of $\ul{X}_k$ is said to be \textbf{up-connected} (in dimension $k$) if, for any $\ul{u}, \ul{u'} \in \ul{C}_k$ with $\ul{u} \neq \ul{u'}$, there exists a sequence of vertices $\ul{u} = \ul{u}_0, \ul{u}_1, \dots, \ul{u}_n = \ul{u'}$ all belonging to $\ul{C}_k$ such that $\ul{u}_{i-1} \sim^\up \ul{u}_i$ for all $i=1, \dots, n$. Notice that singletons (whether they are leaves or not) constitute up-connected subsets of $\ul{X}_k$. The definition of up-connected extends analogously to the double cover $\Gamma$.
\end{definition}

\begin{definition}\label{definition.down-connected}
    A non-empty subset $\ul{C}_k$ of $\ul{X}_k$ is said to be \textbf{down-connected} (in dimension $k$) if, for any $\ul{u}, \ul{u'} \in \ul{C}_k$ with $\ul{u} \neq \ul{u'}$, there exists a sequence of vertices $\ul{u} = \ul{u}_0, \ul{u}_1, \dots, \ul{u}_n = \ul{u'}$ all belonging to $\ul{C}_k$ such that $\ul{u}_{i-1}, \ul{u}_i$ are down-adjacent for all $i=1, \dots, n$. Notice that singletons (whether they are roots or not) constitute down-connected subsets of $\ul{X}_k$. The definition of down-connected extends analogously to the double cover $\Gamma$.
\end{definition}

\begin{definition}\label{definition.k-quotient_up_component}
    A \textbf{quotient-up-component} $\ul{C}_k$ (in dimension $k$) is a non-empty subset of $\ul{X}_k$ that is maximally up-connected, that is: it is up-connected, and not properly contained in any other up-connected subset of $\ul{X}_k$. This notion includes the leaves of $\ul{\Gamma}$.
\end{definition}

\begin{definition}\label{definition.k-cover_up_component}
    A \textbf{cover-up-component} $C_k$ (in dimension $k$) is either a non-empty subset of $X_k \setminus \rmL(\Gamma)$ that is maximally up-connected, or an involutory pair in $X_k \cap \rmL(\Gamma)$.
\end{definition}

\begin{definition}\label{definition.k-coherent_up_component}
    We say that a quotient-up-component $\ul{C}_k$ is a \textbf{coherent-up-component} (in dimension $k$) if it is not a leaf, and there exists an orientation $\calO \colon \ul{\Gamma} \to \Gamma$ with the property that, for any $u, u' \in \calO(\ul{C}_k)$, and for any $v \in X_{k+1}$ such that $u, u' \subset v$, it holds that $[v : u] = [v : u']$.
\end{definition}

\begin{definition}\label{definition.k-quotient_down_component}
    A \textbf{quotient-down-component} $\ul{C_k}$ (in dimension $k$) is a non-empty subset of $\ul{X}_k$ that is maximally down-connected, that is: it is down-connected, and not properly contained in any other down-connected subset of $\ul{X}_k$. This notion includes the roots of $\ul{\Gamma}$.
\end{definition}

\begin{definition}\label{definition.k-cover_down_component}
    A \textbf{cover-down-component} $C_k$ (in dimension $k$) is either a non-empty subset of $X_k \setminus \rmR(\Gamma)$ that is maximally down-connected, or an involutory pair in $X_k \cap \rmR(\Gamma)$.
\end{definition}

\begin{definition}\label{definition.k-coherent_down_component}
    We say that a quotient-down-component $\ul{C}_k$ in dimension $k$ is a \textbf{coherent-down-component} (in dimension $k$) if it is not a root, and there exists an orientation $\calO \colon \ul{\Gamma} \to \Gamma$ with the property that, for any $u, u' \in \calO(\ul{C}_k)$, and for any $t \in X_{k-1}$ such that $u, u' \supset t$, it holds that $[u : t] = [u' : t]$.
\end{definition}

\begin{figure}[H]
    \centering
    \includegraphics[scale=0.75]{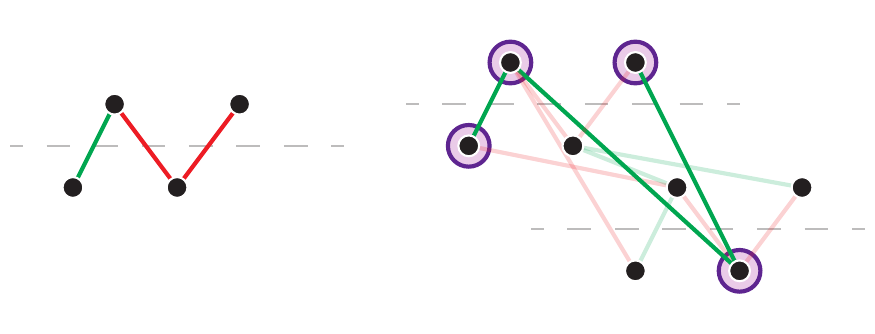}
    \caption{Example of a coherent-up-component in dimension $k-1$ and a coherent-down-component in dimension $k$. On the left, the original graded signed graph. On the right, its associated double cover. Positively signed edges are in green, and negatively signed edges are in red. Nodes below a gray dashed line have dimension $k-1$, nodes above such a line have dimension $k$. The choice of an orientation $\calO \colon \ul{\Gamma} \to \Gamma$ as described in Remark \ref{remark.adjustment_orientation} is highlighted by circling the nodes belonging to $\calO(\ul{\Gamma})$ in purple. Opaque edges do not belong to $\calO(\ul{\Gamma})$.}
    \label{figure.coherent}
\end{figure}

\begin{remark}\label{remark.on_defintion_coherent}
    Notice that, in Definitions \ref{definition.k-coherent_up_component} and \ref{definition.k-coherent_down_component}, even though $\ul{C}_k$ is a subset of the set of nodes $\ul{X}_k$ of the quotient $\ul{\Gamma}$, the definitions depend in fact on the orientation $\calO$, and, ultimately, on the signed structure of $\Gamma$.
\end{remark}

\begin{remark}\label{remark.adjustment_orientation}
    It is worth noticing that, in the definitions of coherent-up- and -down-components, it is only relevant how $\calO$ is defined on $\ul{X}_k$, and not on nodes of other dimensions. Moreover, given a coherent-up-component $\calO(\ul{C}_k)$, it is possible to adjust $\calO$ on $\ul{X}_{k+1}$ so that for any $u \in \calO(\ul{C}_k)$, and for any $v \in \calO(\ul{X}_{k+1})$ such that $u \subset v$, then $[v : u] = 1$, that is, $u$ and $v$ are coherently oriented. This justifies the choice of terminology. A similar adjustment of $\calO$ is possible on $\ul{X}_{k-1}$ in the case of a coherent-down-component.
\end{remark}
\begin{remark}\label{remark.automatic_coherent}
    A quotient-up- or -down-component containing only one node is automatically coherent.
\end{remark}
We provide a visual example of coherent components in Figure \ref{figure.coherent}. The following result relates these different notions of components of $\Gamma$ and $\ul{\Gamma}$:
\begin{lemma}\label{lemma.components_of_Gamma}
    \begin{enumerate}
        \item The quotient-components of $\ul{\Gamma}$ are in bijection with the cover-components of $\Gamma$ via the projection map $\ul{ \ \cdot \ }$.
        \vspace{0.1cm}
    \end{enumerate}
    Suppose also that the grading $\dim$ is strong. Then:
    \vspace{-0.1cm}
    \begin{enumerate}
        \item[2.] For any $k$, the quotient-up-components in dimension $k$ of $\ul{\Gamma}$ are in bijection with the cover-up-components in dimension $k$ of $\Gamma$ via the projection map $\ul{ \ \cdot \ }$. The statement also holds when replacing ``up'' with ``down''.
        \item[3.] If $\ul{C}_k$ is a quotient-down-component in dimension $k$ that is not a root, then the set
        \[
        \ul{C}_{k-1} = \Set{ \ul{t} | \ul{t} \subset \ul{u} \text{ for some } \ul{u} \in \ul{C}_k }
        \]
        is a quotient-up-component in dimension $k-1$ that is not a leaf. This set association, together with the inverse association
        \[
        \ul{C}_k = \Set{ \ul{u} | \ul{u} \supset \ul{t} \text{ for some } \ul{t} \in \ul{C}_{k-1}},
        \]
        is a bijection between quotient-down-components in dimension $k$ that are not roots and quotient-up-components in dimension $k-1$ that are not leaves. Similarly, cover-down-components in dimension $k$ that are not involutory pairs of roots are in bijection with cover-up-components in dimension $k-1$ that are not involutory pairs of leaves.
        \item[4.] The bijection in 3 restricts to a bijection between coherent-down-components in dimension $k$ and coherent-up-components in dimension $k-1$.
    \end{enumerate}
\end{lemma}

\begin{proof}
    The projection $\ul{ \ \cdot \ }$ maps connected components that are not involutory pairs of isolated nodes of $\Gamma$ to connected components of $\ul{\Gamma}$ that are not isolated nodes. Moreover, it maps involutory pairs of isolated nodes to isolated nodes. This proves 1. The same argument extends to up- and down-components, and this proves 2. For 3, it is enough to notice that the prescribed associations are one the inverse of the other. Finally, 4 follows from Remark \ref{remark.adjustment_orientation}.
\end{proof}

\section{Rules of the root-to-leaf path random walk}\label{section.rules_random_walk}

We now describe the one-step rule of the \textbf{root-to-leaf path random walk} on the double cover of a graded signed graph. Suppose a walker is at node $u \in X$. Then:
\begin{itemize}
    \item If $\ul{u}$ is both a leaf and a root, then the walker performs action \bfS with probability $1$.
    \item If $\ul{u}$ is a leaf but not a root, then the walker performs either action \bfS or action \bfD, each with probability $1/2$.
    \item If $\ul{u}$ is not a leaf but it is a root, then the walker performs either action \bfU or action \bfS, each with probability $1/2$.
    \item If $\ul{u}$ is neither a leaf nor a root, then the walker performs action \bfU or action \bfD, each with probability $1/2$.
\end{itemize}
We now describe actions \bfS, \bfU, \bfD:
\begin{itemize}
    \item \bfS (same): the walker either stays at node $u$, or moves to $-u$. Each choice has a probability of $1/2$.
    \item \bfU (up): the walker chooses a node $\ul{v} \in \ul{X}$ with $\ul{v} \supset \ul{u}$, and moves to the corresponding unique $v \in X$ satisfying $[v : u] = 1$. Nodes $\ul{v}$ are chosen accordingly to odds prescribed by the values of the leaf-path function $\LP(\ul{v})$. Equivalently, node $\ul{v}$ is chosen with probability $\LP(\ul{v}) / \LP(\ul{u})$.
    \item \bfD (down): the walker chooses a node $\ul{t} \in \ul{X}$ with $\ul{t} \subset \ul{u}$, and moves to the corresponding unique $t \in X$ satisfying $[u : t] = -1$. Nodes $\ul{t}$ are chosen accordingly to odds prescribed by the values of the root-path function $\RP(\ul{t})$. Equivalently, node $\ul{t}$ is chosen with probability $\RP(\ul{t}) / \RP(\ul{u})$.
\end{itemize}
The following result justifies the terminology and gives an alternative description of the random walk:
\begin{proposition}\label{proposition.equivalent_description_random_walk}
    The root-to-leaf path random walk on $\Gamma$ can be equivalently described as follows. Suppose that the walker is at node $u \in X$. The walker chooses uniformly at random a root-to-leaf path $\eta$ passing through $\ul{u}$. Then, the walker chooses, as described previously, one of the actions 
    \bfS, \bfU', \bfD' to perform, according to whether or not $\ul{u}$ is a leaf or a root. Here, actions \bfU' and \bfD' replace actions \bfU and \bfD, respectively, and are defined as:
    \begin{itemize}
    \item \bfU': let $\ul{v} \supset \ul{u}$ be obtained by walking one step up along $\eta$ from $\ul{u}$. Then, the walker moves to the corresponding unique $v \in X$ satisfying $[v : u] = 1$.
    \item \bfD': let $\ul{t} \subset \ul{u}$ be obtained by walking one step down along $\eta$ from $\ul{u}$. Then, the walker moves to the corresponding unique $t \in X$ satisfying $[u : t] = -1$.
    \end{itemize}
\end{proposition}

\begin{remark}\label{remark.root_to_leaf_path_can_be_sampled_up_down_independently}
    Instead of choosing uniformly at random a root-to-leaf path $\eta$ through $\ul{u}$, the walker might first decide to walk either up or down (when possible), and then only uniformly sample an ascending path $\eta \in \calP^\uparrow(\ul{u})$ when moving up, or a descending path $\eta \in \calP^\downarrow(\ul{u})$ when moving down. Indeed, selecting uniformly at random a root-to-leaf path $\eta$ through $\ul{u}$ is equivalent to sampling uniformly and independently an ascending path in $\calP^\uparrow(\ul{u})$ and a descending path in $\calP^\downarrow(\ul{u})$.
\end{remark}
\begin{proof}
    Suppose that $\ul{u}$ is not a leaf, and let us uniformly sample an ascending path $\eta \in \calP^\uparrow(\ul{u})$. There is a total of $\LP(\ul{u})$ of such paths, by Lemma \ref{lemma.number_of_paths_alternative_computation_leaf_root_path_function}. Of these, exactly $\LP(\ul{v})$ have $\ul{v} \supset \ul{u}$ as node obtained by walking one step up along $\eta$ from $\ul{u}$. Thus, the probability of selecting $\ul{v}$ in this alternative description is $\LP(\ul{v})/\LP(\ul{u})$. Replicating the argument when walking down, it follows that the two descriptions of the root-to-leaf path random walk on $\Gamma$ are equivalent.
\end{proof}

To conclude this section, we provide an explicit expression for the $2N \times 2N$ transition matrix $P^\Gamma$ of the root-to-leaf path random walk on $\Gamma$. Given $u, u' \in X$,
\[
P^\Gamma_{uu'} =
\begin{cases}
    1/2 \cdot \LP(\ul{v})/\LP(\ul{u}) & \text{if $u \subset u' = v$ and $[v : u] = 1$,} \\
    1/2 \cdot \RP(\ul{t})/\RP(\ul{u}) & \text{if $u \supset u' = t$ and $[u : t] = -1$,} \\
    1/4 & \text{if $u = \pm u'$ and $\ul{u}=\ul{u'}$ is a leaf xor a root,} \\
    1/2 & \text{if $u = \pm u'$ and $\ul{u}=\ul{u'}$ is a leaf and a root,} \\
    0 & \text{otherwise.}
\end{cases}
\]
By ``xor'' we mean ``exclusive or''. The root-to-leaf path random walk on the double cover $\Gamma$ induces, via the quotient map $\ul{ \ \cdot \ }$, a random walk on the involutory quotient $\ul{\Gamma}$. The $N \times N$ transition matrix $P^{\ul{\Gamma}}$ is given by
\[
P^{\ul{\Gamma}}_{\ul{u}\ul{u'}} =
\begin{cases}
    1/2 \cdot \LP(\ul{v})/\LP(\ul{u}) & \text{if $\ul{u} \subset \ul{u'} = \ul{v}$,} \\
    1/2 \cdot \RP(\ul{t})/\RP(\ul{u}) & \text{if $\ul{u} \supset \ul{u'} = \ul{t}$,} \\
    1/2 & \text{if $\ul{u} = \ul{u'}$ is a leaf xor a root,} \\
    1 & \text{if $\ul{u} = \ul{u'}$ is a leaf and a root,} \\
    0 & \text{otherwise.}
\end{cases}
\]

\section{The stationary distribution}\label{section.stationary_distribution}

We first explore irreducibility and aperiodicity of the root-to-leaf path random walk, on both the double cover $\Gamma$ and the quotient $\ul{\Gamma}$.
\begin{proposition}\label{proposition.existence_uniqueness_stationary_distributions}
    The root-to-leaf path random walk splits, as a random process, into distinct random walks on the cover-components. On each of these cover-components, the random walk is irreducible and aperiodic. As a consequence, each cover-component induces a unique stationary distribution of the random walk supported on it, and each initial probability distribution supported on it converges to such stationary distribution. The same statements hold when replacing ``cover'' by ``quotient''.
\end{proposition}
\begin{proof}
    The random walk clearly splits into distinct random walks on connected components and involutory pairs of isolated nodes of $\Gamma$. As the values of $\LP(\ul{u})$ and $\RP(\ul{u})$ are strictly positive for any $u \in X$, and since the probability of flipping orientation from $u$ to $-u$ when $u$ is a leaf or a root of a component is strictly positive, the probability of reaching in finite time a node $w'$ from a node $w$ in the same cover-component is positive. Therefore, the restriction of the random walk to each cover-component is irreducible. Moreover, the walker has a positive probability of being lazy when they find themselves at a node in $\rmL(\Gamma)$ or $\rmR(\Gamma)$, that is, they have positive probability of not moving from $u$ in one step. This implies that the random walk is aperiodic. The last statement follows from the fundamental theorem of Markov chains \cite[Chapter 5.3]{beichelt2002stochastic}. The same argument applies to the quotient $\ul{\Gamma}$.
\end{proof}
\begin{proposition}\label{proposition.stationary_distribution_is_induced}
    The stationary probability distribution $\pi_C$ associated with a cover-component $C$ of $\Gamma$ is induced by the stationary probability distribution $\pi_{\ul{C}}$ associated with its projection $\ul{C}$ in $\ul{\Gamma}$, via the relation
    \[
    \pi_C(u) = \frac{\pi_{\ul{C}}(\ul{u})}{2}.
    \]
    Moreover, the same is true for a generic stationary distribution $\pi$ on $\Gamma$.
\end{proposition}
\begin{proof}
    We denote by $R$ the $2N \times 2N$ matrix, index by the nodes of $\Gamma$, representing the involution that reverses orientations:
    \[
    R_{uu'} =
    \begin{cases}
        1 & \text{if $u = -u'$,} \\
        0 & \text{otherwise.}
    \end{cases}
    \]
    We then notice that the transition matrix $P^\Gamma$ satisfies the commuting relation
    \[
    P^\Gamma R = R P^\Gamma.
    \]
    Notice also that $R^T = R$. Now, let $\pi_C$ be the unique stationary distribution associated with $C$. This means that $\pi_C^T P^\Gamma = \pi_C^T$. Then, $R\pi_C$ is also a probability distribution supported on $C$ and
    \[
    (R\pi_C)^T P^\Gamma = \pi_C^T R^T P^\Gamma = \pi_C^T R P^\Gamma = \pi_C^T P^\Gamma R = \pi_C^T R = \pi_C^T R^T = (R\pi_C)^T.
    \]
    By uniqueness, $R\pi_C = \pi_C$ or, in other words, $\pi_C$ is invariant under the involution of $\Gamma$. This implies that $\pi_C$ is induced by a stationary probability distribution $\pi_{\ul{C}}$ on the projection $\ul{C}$ in $\ul{\Gamma}$ of $C$. The factor $2$ at the denominator provides the correct normalization. Finally, the result holds true for a generic stationary distribution $\pi$, as this is a convex combination of stationary distributions supported on the cover-components of $\Gamma$.
\end{proof}
Given a quotient-component $\ul{C}$ of $\ul{\Gamma}$, we denote by $\E_{\ul{C}} [\len(\eta)]$ the expected value of the length of a uniformly sampled root-to-leaf path in $\ul{C}$, and by $\calP_{\rmR \to \rmL}(\ul{C})$ the set of root-to-leaf paths in $\ul{C}$, that is:
\[
    \calP_{\RtoL}(\ul{C}) = \bigcup_{\substack{\ul{r} \in \rmR(\ul{\Gamma}) \cap \ul{C} \\ \ul{\ell} \in \rmL(\ul{\Gamma}) \cap \ul{C}}} \calP^\uparrow(\ul{r}, \ul{\ell}).
\]
\begin{theorem}\label{theorem.explicit_description_stationary_distribution_quotient}
    For each quotient-component $\ul{C}$ of $\ul{\Gamma}$, the probability distribution $\pi_{\ul{C}}$ given by
    \[
    \pi_{\ul{C}}(\ul{u}) = \frac{\LP(\ul{u}) \cdot \RP(\ul{u})}{K_{\ul{C}}},
    \]
    for $\ul{u} \in \ul{C}$, satisfies the Markov chain reversibility condition, where $K_{\ul{C}}$ is the normalizing constant
    \[
    K_{\ul{C}} = |\calP_{\rmR \to \rmL}(\ul{C})| \cdot \big(\E_{\ul{C}} [\len(\eta)] + 1 \big).
    \]
    As a consequence, $\pi_{\ul{C}}$ is the stationary distribution associated with $\ul{C}$.
\end{theorem}
\begin{remark}
    In other words, when starting the random walk on a given quotient-component $\ul{C}$ of $\ul{\Gamma}$, the odds of being at node $\ul{u}$ at time infinity are the number of root-to-leaf paths through $\ul{u}$ (from Corollary \ref{corollary.product_LP_RP_is_root_to_leaf_paths}).
\end{remark}
\begin{proof}
    First, using the explicit description of the transition operator $P^{\ul{\Gamma}}$, we notice that, for any $\ul{u}, \ul{u'} \in \ul{C}$,
    \[
    \LP(\ul{u}) \cdot \RP(\ul{u}) \cdot P^{\ul{\Gamma}}_{\ul{u}\ul{u'}} = \LP(\ul{u'}) \cdot \RP(\ul{u'}) \cdot P^{\ul{\Gamma}}_{\ul{u'}\ul{u}}.
    \]
    It follows that the odds $\LP(\ul{u}) \cdot \RP(\ul{u})$ on $\ul{C}$ determine a steady-state probability distribution (\ie verifying the reversibility condition). We now determine the normalization factor. By double counting,
    \begin{align*}
        K_{\ul{C}} &= \sum_{\ul{u} \in \ul{C}} \LP(\ul{u}) \cdot \RP(\ul{u}) \\
        &= \sum_{\ul{u} \in \ul{C}} |\{ \text{$\eta \in \calP_{\rmR \to \rmL}(\ul{C})$  with $\ul{u} \in \eta$} \}| \\
        &= |\{ (\eta, \ul{u}) \text{ with $\eta \in \calP_{\rmR \to \rmL}(\ul{C})$ and $\ul{u} \in \eta$} \}| \\
        &= \sum_{\eta \in \calP_{\rmR \to \rmL}(\ul{C})} |\{ \ul{u} \in \eta \}| \\
        &= \sum_{\eta \in \calP_{\rmR \to \rmL}(\ul{C})} ( \len(\eta) + 1) \\
        &= | \calP_{\rmR \to \rmL}(\ul{C}) | \cdot \big(\E_{\ul{C}} [\len(\eta)] + 1 \big).
    \end{align*}
    The last statement of the theorem follows from the fact that a steady-state probability distribution is necessarily stationary.
\end{proof}

\begin{corollary}\label{corollary.explicit_description_stationary_distribution}
    For each cover-component $C$, the probability distribution $\pi_C$ given by
    \[
    \pi_C(u) = \frac{\pi_{\ul{C}} (\ul{u})}{2} = \frac{\LP(\ul{u}) \cdot \RP(\ul{u})}{2 \cdot K_{\ul{C}}}
    \]
    for $u \in C$ is the stationary distribution associated with $C$.
\end{corollary}
\begin{proof}
    This follows from Proposition \ref{proposition.stationary_distribution_is_induced} and Theorem \ref{theorem.explicit_description_stationary_distribution_quotient}.
\end{proof}

\begin{remark}\label{remark.no_steady_state_on_cover}
    It is critical to observe that, while the stationary distribution $\pi_{\ul{C}}$ satisfies the reversibility condition, $\pi_C$ does not, due to the opposite signature condition when moving up or down on the double cover.
\end{remark}

\section{Associated operators}\label{section.associated_operators}
Denote by $D^{\ul{\Gamma}}$ the $N \times N$ diagonal matrix given by $D^{\ul{\Gamma}}_{\ul{u}\ul{u}} = \LP(\ul{u}) \cdot \RP(\ul{u})$, the number of root-to-leaf paths passing through $\ul{u}$. The diagonal of $D^{\ul{\Gamma}}$ is, up to a rescaling factor, a nowhere zero convex combination of the steady-state distributions $\pi_{\ul{C}}$, when $\ul{C}$ varies among the quotient-components. In particular, the diagonal of $D^{\ul{\Gamma}}$ satisfies the Markov chain reversibility condition. As a crucial consequence, the $N \times N$ matrix
\[
A^{\ul{\Gamma}} = \big(D^{\ul{\Gamma}}\big)^{-1/2} \big(P^{\ul{\Gamma}}\big)^T \big(D^{\ul{\Gamma}}\big)^{1/2} = \big(D^{\ul{\Gamma}}\big)^{1/2} P^{\ul{\Gamma}} \big(D^{\ul{\Gamma}}\big)^{-1/2}
\]
is symmetric. More explicitly, if we denote by $H(\ul{u})$ the ratio
\[
H(\ul{u}) = \frac{\LP(\ul{u})}{\RP(\ul{u})},
\]
$A^{\ul{\Gamma}}$ is given by
\[
A^{\ul{\Gamma}}_{\ul{u}\ul{u'}} =
\begin{cases}
    1/2 \cdot H(\ul{v})^{1/2} \cdot H(\ul{u'})^{-1/2} & \text{if $\ul{v} = \ul{u} \supset \ul{u'}$,} \\
    1/2 \cdot H(\ul{t})^{-1/2} \cdot H(\ul{u'})^{1/2} & \text{if $\ul{t} = \ul{u} \subset \ul{u'}$,} \\
    1/2 & \text{if $\ul{u} = \ul{u'}$ is a leaf xor a root,} \\
    1 & \text{if $\ul{u} = \ul{u'}$ is a leaf and a root,} \\
    0 & \text{otherwise.}
\end{cases}
\]
However, if we denote by $D^\Gamma$ the $2N \times 2N$ diagonal matrix given by $D^\Gamma_{uu} = \LP(\ul{u}) \cdot \RP(\ul{u})$, then the $2N \times 2N$ matrix $A^\Gamma$ similarly defined as
\[
A^\Gamma = \big(D^\Gamma\big)^{-1/2} \big(P^\Gamma\big)^T \big(D^\Gamma\big)^{1/2}
\]
is, in general, not symmetric (see Remark \ref{remark.no_steady_state_on_cover}). More explicitly, we have that
\[
A^\Gamma_{uu'} =
\begin{cases}
    1/2 \cdot H(\ul{v})^{1/2} \cdot H(\ul{u'})^{-1/2}  & \text{if $v = u \supset u'$ and $[v : u'] = 1$,} \\
    1/2 \cdot H(\ul{t})^{-1/2} \cdot H(\ul{u'})^{1/2} & \text{if $t = u \subset u'$ and $[u' : t] = -1$,} \\
    1/4 & \text{if $\ul{u}=\ul{u'}$ is a leaf xor a root,} \\
    1/2 & \text{if $\ul{u}=\ul{u'}$ is a leaf and a root,} \\
    0 & \text{otherwise.}
\end{cases}
\]
\begin{remark}
    We argue that it is more natural to define $A^\Gamma$ as prescribed above, rather than through its transpose. Indeed, as we will see, we seek a suitable operator $A^\Gamma$ acting on a function space. The matrix $P^\Gamma$ describes node transition probabilities, and therefore it is its transpose $(P^\Gamma)^T$ that acts on probability distributions and, more generally, on functions, when representing them in vectorial form as column vectors.
\end{remark}
In order to analyze the spectrum of $A^\Gamma$, we split this operator as $A^\Gamma = A^{\Gamma, \sym} + A^{\Gamma, \alt}$, where
\[
A^{\Gamma, \sym} = \frac{A^\Gamma + \big(A^\Gamma\big)^T}{2} = \frac{A^\Gamma + A^\Gamma R}{2} = \frac{A^\Gamma + R A^\Gamma}{2}
\]
is symmetric and
\[
A^{\Gamma, \alt} = \frac{A^\Gamma - \big(A^\Gamma\big)^T}{2} = \frac{A^\Gamma - A^\Gamma R}{2} = \frac{A^\Gamma - R A^\Gamma}{2}
\]
is antisymmetric (or alternating). Recall that $R$ is the $2N \times 2N$ matrix, indexed by nodes of $\Gamma$, representing the involution that reverses orientations. More explicitly, these matrices are defined, for $u, u' \in X$, as
\[
A^{\Gamma, \sym}_{uu'} =
\begin{cases}
    1/4 \cdot H(\ul{v})^{1/2} \cdot H(\ul{u'})^{-1/2} & \text{if $v = u \supset u'$,} \\
    1/4 \cdot H(\ul{t})^{-1/2} \cdot H(\ul{u'})^{1/2}  & \text{if $t = u \subset u'$,} \\
    1/4 & \text{if $\ul{u}=\ul{u'}$ is a leaf xor a root,} \\
    1/2 & \text{if $\ul{u}=\ul{u'}$ is a leaf and a root,} \\
    0 & \text{otherwise,}
\end{cases}
\]
and
\[
A^{\Gamma, \alt}_{uu'} =
\begin{cases}
    [v : u']/4 \cdot H(\ul{v})^{1/2} \cdot H(\ul{u'})^{-1/2} & \text{if $v = u \supset u'$,} \\
    - [u' : t]/4 \cdot H(\ul{t})^{-1/2} \cdot H(\ul{u'})^{1/2} & \text{if $t = u \subset u'$,} \\
    0 & \text{otherwise.}
\end{cases}
\]
Denote by $Q^\sym$ the $N \times 2N$ matrix representing the quotient map $\ul{ \ \cdot \ }$, that is:
\[
Q^\sym_{\ul{u} u'} = \begin{cases}
    1 & \text{if $\ul{u} = \ul{u'}$,} \\
    0 & \text{otherwise.}
\end{cases}
\]
For a given orientation $\calO$, we denote by $O$ the $2N \times N$ matrix representing said orientation, that is:
\[
O_{u \ul{u'}} = \begin{cases}
    1 & \text{if $u = \calO(\ul{u'}),$} \\
    0 & \text{otherwise.}
\end{cases}
\]
Then, an alternative description of $Q^\sym$ is
\[
Q^\sym = (O + RO)^T.
\]
It is then natural to define the $N \times 2N$ matrix $Q^\alt$ as
\[
Q^\alt = (O - RO)^T,
\]
or, more explicitly,
\[
Q^\alt_{\ul{u} u'} = \begin{cases}
    1 & \text{if $\calO(\ul{u}) = u'$,} \\
    -1 & \text{if $\calO(\ul{u}) = -u'$,} \\
    0 & \text{otherwise.}
\end{cases}
\]
The symmetric quotient operator $A^{\ul{\Gamma}}$ previously considered can be rewritten in terms of $Q^\sym$ as
\[
A^{\ul{\Gamma}} = \frac{Q^\sym A^\Gamma (Q^\sym)^T}{2} = \frac{Q^\sym A^{\Gamma, \sym} (Q^\sym)^T}{2}.
\]
We define the antisymmetric signed $N \times N$ matrix $A^{\calO(\ul{\Gamma})}$ as
\[
A^{\calO(\ul{\Gamma})} = \frac{Q^\alt A^\Gamma (Q^\alt)^T}{2} = \frac{Q^\alt A^{\Gamma, \alt} (Q^\alt)^T}{2}.
\]
More explicitly, if $u, u' \in \calO(\ul{X})$, then
\[
A^{\calO(\ul{\Gamma})}_{uu'} = \begin{cases}
    [v : u'] /2 \cdot H(\ul{v})^{1/2} \cdot H(\ul{u'})^{-1/2} & \text{if $\ul{v} = \ul{u} \supset \ul{u'}$,} \\
    - [u' : t] /2 \cdot H(\ul{t})^{-1/2} \cdot H(\ul{u'})^{1/2} & \text{if $\ul{t} = \ul{u} \subset \ul{u'}$,} \\
    0 & \text{otherwise.}
\end{cases}
\]
We can think of $A^\Gamma, A^{\Gamma, \sym}$ and $A^{\Gamma, \alt}$ as operators acting on the space of functions
\[
\scrF^\Gamma = \Set{ f \colon X \to \R }.
\]
By representing a function $f$ as a column vector, the matrices above act on $f$ by matrix multiplication on the left. On the other hand, $A^{\ul{\Gamma}}$ acts on the space of functions
\[
\scrF^{\ul{\Gamma}} = \Set{ f \colon \ul{X} \to \R },
\]
and $A^{\calO(\ul{\Gamma})}$ acts on the space of functions
\[
\scrF^{\calO(\ul{\Gamma})} = \Set{ f \colon \calO(\ul{X}) \to \R }.
\]
We consider the standard scalar product on $\scrF^\Gamma$, defined as
\[
\langle f , g \rangle = \sum_{v \in X} f(v)g(v).
\]
on functions $f, g \in \scrF^\Gamma$, and similarly for $\scrF^{\ul{\Gamma}}$ and $\scrF^{\calO(\ul{\Gamma})}$. Then, $\scrF^\Gamma$ admits an orthogonal decomposition into even and odd functions with respect to the involution, that is:
\[
    \scrF^\Gamma = \scrF^{\Gamma, \sym} \oplus^\perp \scrF^{\Gamma, \alt},
\]
where
\[
    \scrF^{\Gamma, \sym} = \Set{ f \colon X \to \R | f(-u) = f(u) \text{ for any $u \in X$} },
\]
and
\[
    \scrF^{\Gamma, \alt} = \Set{ f \colon X \to \R | f(-u) = -f(u) \text{ for any $u \in X$} }.
\]
We then have the identifications
\[
    \scrF^{\ul{\Gamma}} \simeq \scrF^{\Gamma, \sym}
\]
by pulling-back via the quotient map $\ul{ \ \cdot \ }$, and
\[
    \scrF^{\calO(\ul{\Gamma})} \simeq \scrF^{\Gamma, \alt}
\]
by oddly extending a function defined on $\calO(\ul{X})$ to $X$.
\begin{remark}\label{remark.flip_function_calO}
    Given two different orientations $\calO, \calO'$, we have the identification
    \[
    \scrF^{\calO(\ul{\Gamma})} \simeq \scrF^{\Gamma, \alt} \simeq \scrF^{\calO'(\ul{\Gamma})}.
    \]
    In this identification, a function $f$ defined on nodes in $\calO(\ul{X})$ is mapped to a function $f'$ defined on $\calO'(\ul{X})$, satisfying, for $\ul{u} \in \ul{X}$,
    \[
    f'(\calO'(\ul{u})) = \begin{cases}
        f(\calO(\ul{u})) & \text{if $\calO'(\ul{u}) = \calO(\ul{u})$,} \\
        - f(\calO(\ul{u})) & \text{if $\calO'(\ul{u}) = - \calO(\ul{u})$.} \\
    \end{cases}
    \]
    Indeed, the functions $f$ and $f'$ coincide on $X$ when oddly extended from $\calO(\ul{X})$ and $\calO'(\ul{X})$, respectively. In this work, all the operators with superscript $\calO(\ul{\Gamma})$ technically depend on the chosen orientation $\calO$. They are, however, completely spectrally equivalent: the identification above specifies an orthogonal change of basis to pass from one such operator to another.
\end{remark}

We can then state a result in line with the point of view outlined in Section \ref{section.signed_graphs_classic_random_walk}: the spectrum of the transition operator $P^\Gamma$ of the root-to-leaf path random walk on the double cover $\Gamma$ can be described in terms of the spectra of the half-dimensional operators $A^{\ul{\Gamma}}$ and $A^{\calO(\ul{\Gamma})}$, on different geometric spaces -- the quotient $\ul{\Gamma}$ and the signed graph $\calO(\ul{\Gamma})$.
\begin{theorem}\label{theorem.spectrum_double_cover_quotient_signed_graph}
\begin{enumerate}
    \item We have that
    \begin{alignat*}{2}
    A^{\Gamma, \sym}(\scrF^{\Gamma, \sym}) &\subseteq \scrF^{\Gamma, \sym}, & \qquad 
    A^{\Gamma, \alt}(\scrF^{\Gamma, \sym}) &= 0, \\
    A^{\Gamma, \sym}(\scrF^{\Gamma, \alt}) &= 0, & \qquad 
    A^{\Gamma, \alt}(\scrF^{\Gamma, \alt}) &\subseteq \scrF^{\Gamma, \alt}.
    \end{alignat*}
    \item As a consequence, the spectrum of $A^\Gamma$ splits as
    \[
    \Sp\big(A^\Gamma\big) = \Sp\big(A^{\Gamma, \sym}|_{\scrF^{\Gamma, \sym}}\big) \amalg \Sp\big(A^{\Gamma, \alt}|_{\scrF^{\Gamma, \alt}}\big),
    \]
    where the disjoint union also counts multiplicities. Each eigenfunction of $A^{\Gamma, \sym}|_{\scrF^{\Gamma, \sym}}$ or of $A^{\Gamma, \alt}|_{\scrF^{\Gamma, \alt}}$ is also an eigenfunction of $A^\Gamma$, relative to the same eigenvalue.
    \item The operator $A^{\Gamma, \sym}|_{\scrF^{\Gamma, \sym}}$ is symmetric and therefore diagonalizable, and $\Sp(A^{\Gamma, \sym}|_{\scrF^{\Gamma, \sym}})$ is real. On the other hand, the operator $A^{\Gamma, \alt}|_{\scrF^{\Gamma, \alt}}$ is antisymmetric and therefore diagonalizable over $\C$, and $\Sp(A^{\Gamma, \alt}|_{\scrF^{\Gamma, \alt}})$ is purely imaginary.
    \item We have that, counting multiplicities,
    \[
    \Sp\big(A^{\Gamma, \sym}|_{\scrF^{\Gamma, \sym}}\big) = \Sp\big(A^{\ul{\Gamma}}\big),
    \]
    and pulling-back via the quotient map $\ul{ \ \cdot \ }$ maps eigenfunctions of $A^{\ul{\Gamma}}$ to eigenfunctions of $A^{\Gamma, \sym}|_{\scrF^{\Gamma, \sym}}$ relative to the same eigenvalue. Similarly, counting multiplicities,
    \[
    \Sp\big(A^{\Gamma, \alt}|_{\scrF^{\Gamma, \alt}}\big) = \Sp\big(A^{\calO(\ul{\Gamma})}\big),
    \]
    and oddly extending a function defined on $\calO(\ul{X})$ to $X$ maps eigenfunctions of $A^{\calO(\ul{\Gamma})}$ to eigenfunctions of $A^{\Gamma, \alt}|_{\scrF^{\Gamma, \alt}}$ relative to the same eigenvalue.
\end{enumerate}
\end{theorem}
\begin{proof}
    If we identify a function $f \in \scrF^{\Gamma}$ with a vector indexed by $X$, then $f \in \scrF^{\Gamma, \sym}$ if and only if $R f = f$, and $f \in \scrF^{\Gamma, \alt}$ if and only if $R f = - f$. Then, the statement in 1 follows from the fact that $A^{\Gamma, \sym} R = R A^{\Gamma, \sym}$, and $A^{\Gamma, \alt} R = - R A^{\Gamma, \alt}$. The claim in 2 follows from the decomposition $\scrF^\Gamma = \scrF^{\Gamma, \sym} \oplus^\perp \scrF^{\Gamma, \alt}$. The statement in 3 follows from basic properties of symmetric and antisymmetric matrices. For 4, notice that, given $f \in \scrF^{\ul{\Gamma}}$, precomposing with $\ul{ \ \cdot \ }$ corresponds to considering $(Q^\sym)^T f$, when interpreting $f$ as a vector indexed by $\ul{X}$. Furthermore, notice that
    \[
    (Q^\sym)^T Q^\sym = I_{2N} + R,
    \]
    and
    \[
    (Q^\alt)^T Q^\alt = I_{2N} - R.
    \]
    Then, if $f$ is an eigenfunction of $A^{\ul{\Gamma}}$ relative to $\lambda$, we have that
    \begin{align*}
        A^{\Gamma, \sym} (Q^\sym)^T f &= \frac{(I_{2N} + R) A^{\Gamma, \sym}}{2}(Q^\sym)^T f \\
        &= (Q^\sym)^T \frac{Q^\sym A^{\Gamma, \sym}(Q^\sym)^T}{2} f \\
        & = (Q^\sym)^T A^{\ul{\Gamma}} f \\
        &= \lambda (Q^\sym)^T f,
    \end{align*}
    so that $(Q^\sym)^T f$ is an eigenfuction of $A^{\Gamma, \sym}$ relative to the same eigenvalue $\lambda$. The argument is analogous for $A^{\calO(\ul{\Gamma})}$ and $A^{\Gamma, \alt}$.
\end{proof}
\begin{remark}\label{remark.restriction_theorem_functions}
    The decomposition of $X = \amalg_C C$ into cover-components of $\Gamma$, and that of $\ul{X} = \amalg_{\ul{C}} \ul{C}$ into quotient-components of $\ul{\Gamma}$, induce orthogonal decompositions of the spaces $\scrF^\Gamma, \scrF^{\Gamma, \sym}, \scrF^{\Gamma, \alt}, \scrF^{\ul{\Gamma}}$, and $\scrF^{\calO(\ul{\Gamma})}$, into their subspaces of functions supported on $C, \ul{C}$, and $\calO(\ul{C})$. Theorem \ref{theorem.spectrum_double_cover_quotient_signed_graph} remains valid when restricting all operators and function spaces to a cover-component $C$ of $\Gamma$, its corresponding quotient-component $\ul{C}$ of $\ul{\Gamma}$, and its image $\calO(\ul{C})$ in $\calO(\ul{\Gamma})$.
\end{remark}

We now introduce additional operators. In terms of the root-to-leaf path random walk, $\delta^\Gamma$ is the operator associated with action $\bfU$, and its transpose $(\delta^\Gamma)^T$ with action $\bfD$. In the same way, the operators $\Theta_\rmL^\Gamma$ and $\Theta_\rmR^\Gamma$ are associated with action $\bfS$. Let $\delta^\Gamma$ be the $2N \times 2N$ matrix defined, for $u, u' \in X$, as
\[
\delta^{\Gamma}_{uu'} = \begin{cases}
    H(\ul{v})^{1/2} \cdot H(\ul{u'})^{-1/2} & \text{if $v = u \supset u'$ and $[v : u'] = 1$,} \\
    0 & \text{otherwise.}
\end{cases}
\]
The operators $\delta^{\Gamma, \sym}$ and $\delta^{\Gamma, \alt}$ are then defined as
\[
\delta^{\Gamma, \sym} = \frac{\delta^{\Gamma} + \delta^{\Gamma}R}{2} = \frac{\delta^{\Gamma} + R\delta^{\Gamma}}{2}, \quad \delta^{\Gamma, \alt} = \frac{\delta^{\Gamma} - \delta^{\Gamma}R}{2} = \frac{\delta^{\Gamma} - R\delta^{\Gamma}}{2},
\]
so that, analogously to the decomposition of  $A^\Gamma$ into its symmetric and alternating components,
\[
\delta^{\Gamma} = \delta^{\Gamma, \sym} + \delta^{\Gamma, \alt}.
\]
More explicitly,
\[
\delta^{\Gamma, \sym}_{uu'} = \begin{cases}
    1/2 \cdot H(\ul{v})^{1/2} \cdot H(\ul{u'})^{-1/2} & \text{if $v = u \supset u'$,} \\
    0 & \text{otherwise,}
\end{cases}
\]
and
\[
\delta^{\Gamma, \alt}_{uu'} = \begin{cases}
    [v : u']/2 \cdot H(\ul{v})^{1/2} \cdot H(\ul{u'})^{-1/2} & \text{if $v = u \supset u'$,} \\
    0 & \text{otherwise.}
\end{cases}
\]
Despite what the notation might suggest, $\delta^{\Gamma, \sym}$ is not, in general, symmetric, and $\delta^{\Gamma, \alt}$ is not, in general, antisymmetric. Let us also consider the operators $\Pi_\rmL^\Gamma$ and $\Pi_\rmR^\Gamma$, the projectors onto the spaces of functions supported on $\rmL(\Gamma)$ and $\rmR(\Gamma)$, respectively. Define by $\Theta_\rmL^\Gamma$ and $\Theta_\rmR^\Gamma$ the operators
\[
\Theta_\rmL^\Gamma = \Pi_\rmL^\Gamma \frac{I_{2N} + R}{2} = \frac{I_{2N} + R}{2} \Pi_\rmL^\Gamma, \quad \Theta_\rmR^\Gamma = \Pi_\rmR^\Gamma \frac{I_{2N} + R}{2} = \frac{I_{2N} + R}{2} \Pi_\rmR^\Gamma.
\]
More explicitly,
\[
\big(\Theta_\rmL^\Gamma\big)_{uu'} = \begin{cases}
    1/2 & \text{if $u = \pm u'$ and $\ul{u} = \ul{u'}$ is a leaf,} \\
    0 & \text{otherwise,}
\end{cases}
\]
and
\[
\big(\Theta_\rmR^\Gamma\big)_{uu'} = \begin{cases}
    1/2 & \text{if $u = \pm u'$ and $\ul{u} = \ul{u'}$ is a root,} \\
    0 & \text{otherwise.}
\end{cases}
\]
Then,
\[
A^\Gamma = \frac{1}{2} \delta^{\Gamma} + \frac{1}{2} \big(\delta^{\Gamma}\big)^T + \frac{1}{2} \Theta_\rmL^\Gamma + \frac{1}{2} \Theta_\rmR^\Gamma,
\]
as well as
\[
A^{\Gamma, \sym} = \frac{1}{2} \delta^{\Gamma, \sym} + \frac{1}{2} \big(\delta^{\Gamma, \sym}\big)^T + \frac{1}{2} \Theta_\rmL^\Gamma + \frac{1}{2} \Theta_\rmR^\Gamma,
\]
and
\[
A^{\Gamma, \alt} = \frac{1}{2} \delta^{\Gamma, \alt} - \frac{1}{2} \big(\delta^{\Gamma, \alt}\big)^T.
\]
We can replicate a similar construction for the quotient $\ul{\Gamma}$ and a signed graph $\calO(\ul{\Gamma})$. We define $\delta^{\ul{\Gamma}}$, for $\ul{u}, \ul{u'} \in \ul{X}$, as
\begin{equation}\label{equation.delta_ul_Gamma}
    \delta^{\ul{\Gamma}}_{\ul{u}\ul{u'}} = \begin{cases}
        H(\ul{v})^{1/2} \cdot H(\ul{u'})^{-1/2} & \text{if $\ul{v} = \ul{u} \supset \ul{u'}$,} \\
        0 & \text{otherwise.}
    \end{cases}
\end{equation}
We also define $\delta^{\calO(\ul{\Gamma})}$, for $u, u' \in \calO(\ul{X})$, as
\[
\delta^{\calO(\ul{\Gamma})}_{uu'} = \begin{cases}
    [v : u'] \cdot H(\ul{v})^{1/2} \cdot H(\ul{u'})^{-1/2} & \text{if $\ul{v} = \ul{u} \supset \ul{u'}$,} \\
    0 & \text{otherwise.}
\end{cases}
\]
We also denote by $\Pi^{\ul{\Gamma}}_\rmL$ and $\Pi^{\ul{\Gamma}}_\rmR$ the projectors onto the spaces of functions supported on leaves and roots of $\ul{\Gamma}$, respectively. Then,
\[
    A^{\ul{\Gamma}} = \frac{1}{2} \delta^{\ul{\Gamma}} + \frac{1}{2} \big(\delta^{\ul{\Gamma}}\big)^T + \frac{1}{2} \Pi_\rmL^{\ul{\Gamma}} + \frac{1}{2} \Pi_\rmR^{\ul{\Gamma}},
\]
and
\[
    A^{\calO(\ul{\Gamma})} = \frac{1}{2} \delta^{\calO(\ul{\Gamma})} - \frac{1}{2} \big(\delta^{\calO(\ul{\Gamma})})^T.
\]
To conclude this section, we point out that, since the random walk is aperiodic on each of the cover-components, $-1$ is not an eigenvalue of $P^{\ul{\Gamma}}$. However, in contrast with the conditional random walks that we will shortly explore in Section \ref{section.conditional_random_walks}, the transition matrix $P^{\ul{\Gamma}}$ of the root-to-leaf path random walk on the quotient $\ul{\Gamma}$ might admit negative eigenvalues:
\begin{proposition}\label{proposition.minimal_eigenvalue_expected_length}
    Let $\lambda_{\min}(A^{\ul{\Gamma}})$ be the minimal eigenvalue of $A^{\ul{\Gamma}}$, or equivalently of $P^{\ul{\Gamma}}$. Then,
    \[
    \lambda_{\min}\big(A^{\ul{\Gamma}}\big) - (-1) \le \min_{\ul{C}} \frac{2}{\E_{\ul{C}}[\len(\eta)] + 1},
    \]
    where the minimum is taken over the quotient-components $\ul{C}$.
\end{proposition}
\begin{proof}
    If $\ul{C}$ is an isolated node, then $\frac{2}{\E_{\ul{C}}[\len(\eta)] + 1} = 2$, and the inequality is satisfied. Otherwise, consider the non-zero function $q_{\ul{C}}$ on $\ul{\Gamma}$ defined, for $\ul{u} \in \ul{X}$, as $q_{\ul{C}}(\ul{u}) = (-1)^{\dim(\ul{u})} \pi_{\ul{C}}(\ul{u})^{1/2}$.
    Then,
    \[
    \big(A^{\ul{\Gamma}}q_{\ul{C}}\big)(\ul{u}) =
    \begin{cases}
        (-1)^{\dim(\ul{u}) + 1} \pi_{\ul{C}}(\ul{u})^{1/2} & \text{if $\ul{u}$ is neither a leaf nor a root,} \\
        0 & \text{otherwise.}
    \end{cases}
    \]
    It follows that the Rayleigh quotient of the function $q_{\ul{C}}$ with respect to the operator $A^{\ul{\Gamma}}$ satisfies
    \begin{align*}
        \lambda_{\min}\big(A^{\ul{\Gamma}}\big) \le \calR_{A^{\ul{\Gamma}}}(q_{\ul{C}}) &= \frac{(q_{\ul{C}})^T A^{\ul{\Gamma}} q_{\ul{C}}}{\|q_{\ul{C}}\|^2} \\
        &= - \frac{\sum_{\ul{u} \in \ul{C} \setminus (\rmL(\ul{\Gamma}) \cup \rmR(\ul{\Gamma}))} \LP(\ul{u}) \cdot \RP(\ul{u})}{\sum_{\ul{u} \in \ul{C}} \LP(\ul{u}) \cdot \RP(\ul{u})} \\
        &= - \frac{\sum_{\eta \in \calP_{\rmR \to \rmL}(\ul{C})} (\len(\eta) - 1)}{\sum_{\eta \in \calP_{\rmR \to \rmL}(\ul{C})} (\len(\eta) + 1)} \\
        &= - \frac{\E_{\ul{C}}[\len(\eta)] - 1}{\E_{\ul{C}}[\len(\eta)] + 1} \\
        &= - 1 + \frac{2}{\E_{\ul{C}}[\len(\eta)] + 1},
    \end{align*}
    and the result follows.
\end{proof}

\section{Conditional random up- and down-walks}\label{section.conditional_random_walks}

From now on, we assume that the grading $\dim$ is strong, according to Definition \ref{definition.strongly_graded}. This means that, if $u \subset v$ is an edge of $\Gamma$, then $\dim(v) = \dim(u) + 1$. The \textbf{conditional random up-walk} on $\Gamma$, or simply \textbf{up-walk}, is described as follows. We put emphasis on the dimension $k$ of the starting node $u$, even though the description might be given simultaneously across all dimensions. A single step of the up-walk is made of two steps following the rules of the root-to-leaf path random walk, where we condition on:
\begin{itemize}
    \item If $\ul{u}$ is not a leaf: the walker being at a node $v \supset u$, necessarily of dimension $k + 1$, after the first step, and then at a node $u' \subset v$, necessarily of dimension $k$, after the second step.
    \item If $\ul{u}$ is a leaf: the walker walking to a node of the same dimension $k$, or equivalently walking uniformly to $\pm u$.
\end{itemize}
We underline that the up-walk splits across different dimensions, in the sense that a walker starting at a node of dimension $k$ will find themselves at dimension $k$ at all times. We denote by $P^{\Gamma, \up}_k$ the $2N_k \times 2N_k$ transition matrix of the up-walk in dimension $k$, that is, for $u, u' \in X_k$:
\[
\big(P^{\Gamma, \up}_k\big)_{uu'} = \begin{cases}
    \sum_{\substack{v \supset u, u' \\ [v : u] = 1, [v : u'] = -1}} \LP(\ul{v}) / \LP(\ul{u}) \cdot \RP(\ul{u'}) / \RP(\ul{v}) & \text{if $\ul{u} \sim^\up_= \ul{u'}$,} \\
    1/2 & \text{if $\ul{u} = \ul{u'}$ is a leaf,} \\
    0 & \text{otherwise.}
\end{cases}
\]
The induced up-walk in dimension $k$ on the quotient has, as $N_k \times N_k$ transition matrix,
\[
\big(P^{\ul{\Gamma}, \up}_k\big)_{\ul{u}\ul{u'}} = \begin{cases}
    \sum_{\ul{v} \supset \ul{u}, \ul{u'}} \LP(\ul{v}) / \LP(\ul{u}) \cdot \RP(\ul{u'}) / \RP(\ul{v}) & \text{if $\ul{u} \sim^\up_= \ul{u'}$,} \\
    1 & \text{if $\ul{u} = \ul{u'}$ is a leaf,} \\
    0 & \text{otherwise.}
\end{cases}
\]
Dually, we define the \textbf{conditional random down-walk} on $\Gamma$, or simply \textbf{down-walk}. A single step of the down-walk is made of two steps following the rules of the root-to-leaf path random walk, where we condition on:
\begin{itemize}
    \item If $\ul{u}$ is not a root: the walker being at a node $t \subset u$, necessarily of dimension $k - 1$, after the first step, and then at a node $u' \supset t$, necessarily of dimension $k$, after the second step.
    \item If $\ul{u}$ is a root: the walker walking to a node of the same dimension $k$, or equivalently walking uniformly to $\pm u$.
\end{itemize}
As for the up-case, the down-walk splits across different dimensions, in the sense that a walker starting at a node of dimension $k$ will find themselves at dimension $k$ at all times. We denote by $P^{\Gamma, \down}_k$ the $2N_k \times 2N_k$ transition matrix of the down-walk in dimension $k$, that is, for $u, u' \in X_k$:
\[
\big(P^{\Gamma, \down}_k\big)_{uu'} = \begin{cases}
    \sum_{\substack{t \subset u, u' \\ [u : t] = -1, [u' : t] = 1}} \RP(\ul{t}) / \RP(\ul{u}) \cdot \LP(\ul{u'}) / \LP(\ul{t}) & \text{if $\ul{u} \sim^\down_= \ul{u'}$,} \\
    1/2 & \text{if $\ul{u} = \ul{u'}$ is a root,} \\
    0 & \text{otherwise.}
\end{cases}
\]
The induced down-walk in dimension $k$ on the quotient has, as $N_k \times N_k$ transition matrix,
\[
\big(P^{\ul{\Gamma}, \down}_k\big)_{\ul{u}\ul{u'}} = \begin{cases}
    \sum_{\ul{t} \subset \ul{u}, \ul{u'}} \RP(\ul{t}) / \RP(\ul{u}) \cdot \LP(\ul{u'}) / \LP(\ul{t}) & \text{if $\ul{u} \sim^\down_= \ul{u'}$,} \\
    1 & \text{if $\ul{u} = \ul{u'}$ is a root,} \\
    0 & \text{otherwise.}
\end{cases}
\]
\begin{proposition}\label{proposition.existence_uniqueness_stationary_distributions_up_down}
    The conditional random up-walk in dimension $k$ on $\Gamma$ splits, as a random process, into distinct random walks on the cover-up-components in dimension $k$. On each of these cover-up-components, the random walk is irreducible. As a consequence, each cover-up-component induces a unique stationary distribution of the random walk supported on it. These statements similarly extend to the quotient $\ul{\Gamma}$. In addition, the up-walk in dimension $k$ induced on the quotient $\ul{\Gamma}$ is aperiodic on each of the quotient-up-components. Therefore, on a quotient-up-component, each initial probability distribution converges to the corresponding stationary distribution. All of the above also holds when replacing ``up'' with ``down''.
\end{proposition}
\begin{remark}\label{remark.up_down_walk_not_aperiodic}
    The up-walk on a cover-up-component is not necessarily aperiodic (see Theorem \ref{theorem.-1_eigenvalue}). The same holds for the down-walk.
\end{remark}
\begin{proof}
    The proof is analogous to the proof of Proposition \ref{proposition.existence_uniqueness_stationary_distributions}. Notice that, on the quotient, there is a positive probability of returning at the same node in one step, and therefore the up- and down-walks are aperiodic on each quotient-up- and -down-component.
\end{proof}
\begin{proposition}\label{proposition.stationary_distribution_is_induced_up_down}
    The stationary probability distribution $\pi_{C_k}$ for the up-walk in dimension $k$ associated with a cover-up-component $C_k$ of $\Gamma$ is induced by the stationary probability distribution $\pi_{\ul{C}_k}$ associated with its projection $\ul{C}_k$ in $\ul{\Gamma}$, via the relation
    \[
    \pi_{C_k}(u) = \frac{\pi_{\ul{C}_k}(\ul{u})}{2}.
    \]
    Moreover, the same is true for a generic stationary distribution $\pi$ on $\Gamma$. The same result holds when replacing ``up'' with ``down''.
\end{proposition}
\begin{proof}
    The proof is analogous to the proof of Proposition \ref{proposition.stationary_distribution_is_induced}.
\end{proof}
The following result shows that the restrictions of stationary distributions of the root-to-leaf path random walk to the set of nodes of dimension $k$ induce, in fact, stationary distributions of up- and down-walks:
\begin{theorem}\label{theorem.explicit_description_stationary_distribution_quotient_up_down}
    For each quotient-up-component $\ul{C}_k$ in dimension $k$ of $\ul{\Gamma}$, the probability distribution $\pi_{\ul{C}_k}$ given by
    \[
    \pi_{\ul{C}_k}(\ul{u}) = \frac{\LP(\ul{u}) \cdot \RP(\ul{u})}{K_{\ul{C}_k}},
    \]
    for $\ul{u} \in \ul{C}_k$, satisfies the Markov chain reversibility condition for the up-walk on $\ul{\Gamma}$, where $K_{\ul{C}_k}$ is the normalizing constant
    \[
    K_{\ul{C}_k} = | \Set{ \eta \in \calP_{\rmR \to \rmL} | \eta \ \text{contains a node in $\ul{C}_k$}} |.
    \]
    As a consequence, $\pi_{\ul{C}_k}$ is the stationary distribution associated with $\ul{C}_k$. Moreover, the induced probability distribution $\pi_{C_k}$, described as
    \[
    \pi_{C_k}(u) = \frac{\pi_{\ul{C}_k}(\ul{u})}{2} = \frac{\LP(\ul{u}) \cdot \RP(\ul{u})}{2 \cdot K_{\ul{C}_k}},
    \]
    also satisfies the Markov chain reversibility condition for the up-walk on $\Gamma$. All of the above also holds when replacing ``up'' with ``down''.
\end{theorem}
\begin{proof}
    The proof is analogous to the proof of Theorem \ref{theorem.explicit_description_stationary_distribution_quotient_up_down}. For up- and down-walks, the reversibility condition is verified not only on the quotient, but also on the cover.
\end{proof}
As a consequence, the four operators
\begin{alignat*}{3}
A_k^{\Gamma, \up}    & \quad = \quad & \big(D_k^\Gamma\big)^{-1/2} \big(P_k^{\Gamma, \up}\big)^T \big(D_k^\Gamma\big)^{1/2} 
                      & \quad = \quad & \big(D_k^\Gamma\big)^{1/2} P_k^{\Gamma, \up} \big(D_k^\Gamma\big)^{-1/2}, \\
A_k^{\Gamma, \down}  & \quad = \quad & \big(D_k^\Gamma\big)^{-1/2} \big(P_k^{\Gamma, \down}\big)^T \big(D_k^\Gamma\big)^{1/2} 
                      & \quad = \quad & \big(D_k^\Gamma\big)^{1/2} P_k^{\Gamma, \down} \big(D_k^\Gamma\big)^{-1/2}, \\
A_k^{\ul{\Gamma}, \up} & \quad = \quad & \big(D_k^{\ul{\Gamma}}\big)^{-1/2} \big(P_k^{\ul{\Gamma}, \up}\big)^T \big(D_k^{\ul{\Gamma}}\big)^{1/2} 
                      & \quad = \quad & \big(D_k^{\ul{\Gamma}}\big)^{1/2} P_k^{\ul{\Gamma}, \up} \big(D_k^{\ul{\Gamma}}\big)^{-1/2}, \\
A_k^{\ul{\Gamma}, \down} & \quad = \quad & \big(D_k^{\ul{\Gamma}}\big)^{-1/2} \big(P_k^{\ul{\Gamma}, \down}\big)^T \big(D_k^{\ul{\Gamma}}\big)^{1/2} 
                      & \quad = \quad & \big(D_k^{\ul{\Gamma}}\big)^{1/2} P_k^{\ul{\Gamma}, \down} \big(D_k^{\ul{\Gamma}}\big)^{-1/2}.
\end{alignat*}
are all symmetric. Here $D_k^\Gamma$ and $D_k^{\ul{\Gamma}}$ are the $2N_k \times 2N_k$ and $N_k \times N_k$ restrictions of $D^\Gamma$ and $D^{\ul{\Gamma}}$ to $X_k$ and $\ul{X}_k$, respectively. More explicitly, these operators can be written as
\[
\big(A^{\Gamma, \up}_k\big)_{uu'} = \begin{cases}
    H(\ul{u})^{-1/2} \cdot H(\ul{u'})^{-1/2} \cdot \sum_{\substack{v \supset u, u' \\ [v : u] = 1, [v : u'] = -1}} H(\ul{v}) & \text{if $\ul{u} \sim^\up_= \ul{u'}$,} \\
    1/2 & \text{if $\ul{u} = \ul{u'}$ is a leaf,} \\
    0 & \text{otherwise,}
\end{cases}
\]
\[
\big(A^{\Gamma, \down}_k\big)_{uu'} = \begin{cases}
    H(\ul{u})^{1/2} \cdot H(\ul{u'})^{1/2} \cdot \sum_{\substack{t \subset u, u' \\ [u : t] = -1, [u' : t] = 1}} H(\ul{t})^{-1} & \text{if $\ul{u} \sim^\down_= \ul{u'}$,} \\
    1/2 & \text{if $\ul{u} = \ul{u'}$ is a root,} \\
    0 & \text{otherwise,}
\end{cases}
\]
\[
\big(A^{\ul{\Gamma}, \up}_k\big)_{\ul{u}\ul{u'}} = \begin{cases}
    H(\ul{u})^{-1/2} \cdot H(\ul{u'})^{-1/2} \cdot \sum_{\ul{v} \supset \ul{u}, \ul{u'}} H(\ul{v}) & \text{if $\ul{u} \sim^\up_= \ul{u'}$,} \\
    1 & \text{if $\ul{u} = \ul{u'}$ is a leaf,} \\
    0 & \text{otherwise,}
\end{cases}
\]
and
\[
\big(A^{\ul{\Gamma}, \down}_k\big)_{\ul{u}\ul{u'}} = \begin{cases}
    H(\ul{u})^{1/2} \cdot H(\ul{u'})^{1/2} \cdot \sum_{\ul{t} \subset \ul{u}, \ul{u'}} H(\ul{t})^{-1} & \text{if $\ul{u} \sim^\down_= \ul{u'}$,} \\
    1 & \text{if $\ul{u} = \ul{u'}$ is a root,} \\
    0 & \text{otherwise.}
\end{cases}
\]
Following the analogy with the operators of the non-conditional root-to-leaf path random walk, it is natural to also introduce, given an orientation $\calO$ and $u, u' \in \calO(\ul{X}_k)$,
\[
\big(A^{\calO(\ul{\Gamma}), \up}_k\big)_{uu'} = \begin{cases}
    - H(\ul{u})^{-1/2} \cdot H(\ul{u'})^{-1/2} \\
    \qquad \qquad \cdot \sum_{\ul{v} \supset \ul{u}, \ul{u'}} [v:u] \cdot [v:u'] \cdot H(\ul{v}) & \text{if $\ul{u} \sim^\up_= \ul{u'}$,} \\
    0 & \text{otherwise,}
\end{cases}
\]
and
\[
\big(A^{\calO(\ul{\Gamma}), \down}_k\big)_{uu'} = \begin{cases}
    - H(\ul{u})^{1/2} \cdot H(\ul{u'})^{1/2} \\
    \qquad \qquad \cdot \sum_{\ul{t} \subset \ul{u}, \ul{u'}} [u:t] \cdot [u':t] \cdot H(\ul{t})^{-1} & \text{if $\ul{u} \sim^\down_= \ul{u'}$,} \\
    0 & \text{otherwise.}
\end{cases}
\]
Notice that the minus signs in these definitions agree with the fact that, in the rules of up- and down-walks, a coherently oriented face is chosen when moving from dimension $k$ to $k+1$, or from $k-1$ to $k$, and a non-coherently oriented face when moving from dimension $k+1$ to $k$, or from $k$ to $k-1$. Combining two choices always results in a one-step update presenting a minus sign. \\

We now consider the operators $\delta^\Gamma, \delta^{\Gamma, \sym}, \delta^{\Gamma, \alt}, \delta^{\ul{\Gamma}}$, and $\delta^{\calO(\ul{\Gamma})}$, and their decompositions
\[
\delta^\Gamma = \oplus_k \delta_k^\Gamma, \quad \delta^{\Gamma, \sym} = \oplus_k \delta_k^{\Gamma, \sym}, \quad \delta^{\Gamma, \alt} = \oplus_k \delta_k^{\Gamma, \alt},
\]
\[
\delta^{\ul{\Gamma}} = \oplus_k \delta_k^{\ul{\Gamma}}, \quad \delta^{\calO(\ul{\Gamma})} = \oplus_k \delta_k^{\calO(\ul{\Gamma})},
\]
where indexing by $k$ means restricting an operator to $X_k$, $\ul{X}_k$, or $\calO(\ul{X}_k)$. This means that the column index of each of these $k$-indexed matrix runs through nodes of dimension $k$ of their corresponding geometric space. Since the grading is strong, these operators respect the grading, in the sense that they are grading-shifting:
\[
\delta_k^\Gamma, \delta_k^{\Gamma, \sym}, \delta_k^{\Gamma, \alt} \colon \scrF_k^\Gamma \to \scrF_{k+1}^\Gamma, \quad \delta_k^{\ul{\Gamma}} \colon \scrF_k^{\ul{\Gamma}} \to \scrF_{k+1}^{\ul{\Gamma}}, \quad \delta_k^{\calO(\ul{\Gamma})} \colon \scrF_k^{\calO(\ul{\Gamma})} \to \scrF_{k+1}^{\calO(\ul{\Gamma})},
\]
where indexing by $k$ on function spaces means considering functions supported on $X_k$, $\ul{X}_k$, or $\calO(\ul{X}_k)$. In other words, the row index of each of these $k$-indexed matrices runs through nodes of dimension $k+1$ of their corresponding geometric space. A similar decomposition holds for the operators $\Theta_\rmL^\Gamma, \Theta_\rmR^\Gamma, \Pi_\rmL^{\ul{\Gamma}}$, and $\Pi_\rmR^{\ul{\Gamma}}$:
\[
\Theta_\rmL^\Gamma = \oplus_k \Theta_{\rmL, k}^\Gamma, \quad \Theta_\rmR^\Gamma = \oplus_k \Theta_{\rmR, k}^\Gamma, \quad \Pi_\rmL^{\ul{\Gamma}} = \oplus_k \Pi_{\rmL, k}^{\ul{\Gamma}}, \quad \Pi_\rmR^{\ul{\Gamma}} = \oplus_k \Pi_{\rmR, k}^{\ul{\Gamma}}.
\]
Again, indexing by $k$ means restricting an operator to $X_k$, $\ul{X}_k$, or $\calO(\ul{X}_k)$. In contrast with the grading-shifting operator previously considered, these four operators are grading-preserving, that is:
\[
\Theta_{\rmL, k}^\Gamma, \Theta_{\rmR, k}^\Gamma \colon \scrF_k^\Gamma \to \scrF_k^\Gamma, \quad \Pi_{\rmL, k}^{\ul{\Gamma}}, \Pi_{\rmR, k}^{\ul{\Gamma}} \colon \scrF_k^{\ul{\Gamma}} \to \scrF_k^{\ul{\Gamma}}.
\]
In particular, both row and column indices of each of these $k$-indexed operators run through nodes of the same dimension $k$ of their corresponding geometric space. We can then rewrite the following six symmetric operators, by decomposing them into their positive semi-definitive and negative semi-definitive components:
\begin{alignat*}{3}
A_k^{\Gamma, \up}    & \quad = \quad & \big(\delta_k^\Gamma\big)^T \delta_k^\Gamma + \Theta_{\rmL, k}^\Gamma 
                      & \quad = \quad & \underbrace{\big(\delta_k^{\Gamma, \sym}\big)^T \delta_k^{\Gamma, \sym} + \Theta_{\rmL, k}^\Gamma}_{A_k^{\Gamma, \up, +}} 
                      & \underbrace{- \big(\delta_k^{\Gamma, \alt}\big)^T \delta_k^{\Gamma, \alt}}_{A_k^{\Gamma, \up, -}}, \\
A_k^{\Gamma, \down}  & \quad = \quad & \delta_{k-1}^\Gamma \big(\delta_{k-1}^\Gamma\big)^T + \Theta_{\rmR, k}^\Gamma 
                      & \quad = \quad & \underbrace{\delta_{k-1}^{\Gamma, \sym} \big(\delta_{k-1}^{\Gamma, \sym}\big)^T + \Theta_{\rmR, k}^\Gamma}_{A_k^{\Gamma, \down, +}} 
                      & \underbrace{- \delta_{k-1}^{\Gamma, \alt} \big(\delta_{k-1}^{\Gamma, \alt}\big)^T}_{A_k^{\Gamma, \down, -}}.
\end{alignat*}
\vspace{-0.3cm}
\begin{alignat*}{3}
A_k^{\ul{\Gamma}, \up} & = & \big(\delta_k^{\ul{\Gamma}}\big)^T \delta_k^{\ul{\Gamma}} + \Pi_{\rmL, k}^{\ul\Gamma}, & \qquad &
A_k^{\ul{\Gamma}, \down} & = \delta_{k-1}^{\ul{\Gamma}} \big(\delta_{k-1}^{\ul{\Gamma}}\big)^T + \Pi_{\rmR, k}^{\ul\Gamma}, \\
A_k^{\calO(\ul{\Gamma}), \up} & = & - \big(\delta_k^{\calO(\ul{\Gamma})}\big)^T \delta_k^{\calO(\ul{\Gamma})}, & \qquad &
A_k^{\calO(\ul{\Gamma}), \down} & = - \delta_{k-1}^{\calO(\ul{\Gamma})} \big(\delta_{k-1}^{\calO(\ul{\Gamma})}\big)^T.
\end{alignat*}
Notice that the matrices $A_k^{\Gamma, \up, +}, A_k^{\Gamma, \up, -}, A_k^{\Gamma, \down, +}$ and $A_k^{\Gamma, \down, -}$ can also be written as
\begin{alignat*}{3}
A_k^{\Gamma, \up, +}    & \quad = \quad & \frac{A_k^{\Gamma, \up} + A_k^{\Gamma, \up} R_k}{2} 
                      & \quad = \quad & \frac{A_k^{\Gamma, \up} + R_k A_k^{\Gamma, \up}}{2}, \\
A_k^{\Gamma, \up, -}    & \quad = \quad & \frac{A_k^{\Gamma, \up} - A_k^{\Gamma, \up} R_k}{2} 
                      & \quad = \quad & \frac{A_k^{\Gamma, \up} - R_k A_k^{\Gamma, \up}}{2}, \\
A_k^{\Gamma, \down, +}  & \quad = \quad & \frac{A_k^{\Gamma, \down} + A_k^{\Gamma, \down} R_k}{2} 
                      & \quad = \quad & \frac{A_k^{\Gamma, \down} + R_k A_k^{\Gamma, \down}}{2}, \\
A_k^{\Gamma, \down, -}  & \quad = \quad & \frac{A_k^{\Gamma, \down} - A_k^{\Gamma, \down} R_k}{2} 
                      & \quad = \quad & \frac{A_k^{\Gamma, \down} - R_k A_k^{\Gamma, \down}}{2},
\end{alignat*}
where $R_k$ is the restriction of $R$ to nodes of dimension $k$. \\

Items 1 to 4 of the following theorem constitute the analogous result of Theorem \ref{theorem.spectrum_double_cover_quotient_signed_graph}, in the case of up- and down-walks. Additionally, we also prove items 5 and 6, which relate up- and down-walks to each other.
\begin{theorem}\label{theorem.spectrum_up_down_double_cover_quotient_signed_graph}
\begin{enumerate}
    \item We have that
    \begin{alignat*}{2}
    A_k^{\Gamma, \up, +}(\scrF_k^{\Gamma, \sym}) &\subseteq \scrF_k^{\Gamma, \sym}, & \qquad 
    A_k^{\Gamma, \up, +}(\scrF_k^{\Gamma, \alt}) &= 0, \\
    A_k^{\Gamma, \up, -}(\scrF_k^{\Gamma, \sym}) &= 0, & \qquad 
    A_k^{\Gamma, \up, -}(\scrF_k^{\Gamma, \alt}) &\subseteq \scrF_k^{\Gamma, \alt}.
    \end{alignat*}
    \item As a consequence, the spectrum of $A_k^{\Gamma, \up}$ splits as
    \[
    \Sp\big(A_k^{\Gamma, \up}\big) = \Sp\big(A_k^{\Gamma, \up, +}|_{\scrF_k^{\Gamma, \sym}}\big) \amalg \Sp\big(A_k^{\Gamma, \up, -}|_{\scrF_k^{\Gamma, \alt}}\big),
    \]
    where the disjoint union also counts multiplicities. Each eigenfunction of $A_k^{\Gamma, \up, +}|_{\scrF_k^{\Gamma, \sym}}$ or of $A_k^{\Gamma, \up, -}|_{\scrF_k^{\Gamma, \alt}}$ is also an eigenfunction of $A_k^{\Gamma, \up}$, relative to the same eigenvalue.
    \item The operator $A_k^{\Gamma, \up, +}|_{\scrF_k^{\Gamma, \sym}}$ is symmetric (and therefore diagonalizable) and positive semi-definite, so that $\Sp(A_k^{\Gamma, \up, +}|_{\scrF_k^{\Gamma, \sym}}) \subseteq [0, 1]$. On the other hand, the operator $A_k^{\Gamma, \up, -}|_{\scrF_k^{\Gamma, \alt}}$ is also symmetric (and therefore diagonalizable) and negative semi-definite, so that $\Sp(A_k^{\Gamma, \up, -}|_{\scrF_k^{\Gamma, \alt}}) \subseteq [-1, 0]$.
    \item We have that, counting multiplicities,
    \[
    \Sp\big(A_k^{\Gamma, \up, +}|_{\scrF_k^{\Gamma, \sym}}\big) = \Sp\big(A_k^{\ul{\Gamma}, \up}\big),
    \]
    and pulling-back via the quotient map $\ul{ \ \cdot \ }$ maps eigenfunctions of $A_k^{\ul{\Gamma}, \up}$ to eigenfunctions of $A_k^{\Gamma, \up, +}|_{\scrF_k^{\Gamma, \sym}}$ relative to the same eigenvalue. Similarly, counting multiplicities,
    \[
    \Sp\big(A_k^{\Gamma, \up, -}|_{\scrF_k^{\Gamma, \alt}}\big) = \Sp\big(A_k^{\calO(\ul{\Gamma}), \up}\big),
    \]
    and oddly extending a function defined on $\calO(\ul{X}_k)$ to $X_k$ maps eigenfunctions of $A_k^{\calO(\ul{\Gamma}), \up}$ to eigenfunctions of $A_k^{\Gamma, \up, -}|_{\scrF_k^{\Gamma, \alt}}$ relative to the same eigenvalue.
    \vspace{0.1cm}
    \end{enumerate}
Statements 1 to 4 also holds when replacing ``up'' with ``down''. Additionally:
\vspace{-0.1cm}
\begin{enumerate}
    \item[5.] The operator $\delta_{k-1}^{\Gamma, +}$ maps eigenfunctions of $A_{k-1}^{\Gamma, \up, +}$ relative to a non-zero eigenvalue and supported on $X_{k-1} \setminus \rmL(\Gamma)$ to eigenfunctions of $A_k^{\Gamma, \down, +}$ relative to the same non-zero eigenvalue and supported on $X_k \setminus \rmR(\Gamma)$. The same is true in reversed order, through the operator $(\delta_{k-1}^{\Gamma, +})^T$. An analogous statement holds for $\delta_{k-1}^{\ul{\Gamma}}$, $A_{k-1}^{\ul{\Gamma}, \up}$, and $A_k^{\ul{\Gamma}, \down}$.
    \item[6.] The operator $\delta_{k-1}^{\Gamma, -}$ maps eigenfunctions of $A_{k-1}^{\Gamma, \up, -}$ relative to a non-zero eigenvalue to eigenfunctions of $A_k^{\Gamma, \down, -}$ relative to the same non-zero eigenvalue. The same is true in reversed order, through the operator $(\delta_{k-1}^{\Gamma, -})^T$. An analogous statement holds for $\delta_{k-1}^{\calO(\ul{\Gamma})}$, $A_{k-1}^{\calO(\ul{\Gamma}), \up}$, and $A_k^{\calO(\ul{\Gamma}), \down}$.
\end{enumerate}
\end{theorem}
\begin{remark}\label{remark.no_condition_on_leaves_roots_already_zero}
    If we compare items 5 and 6, we notice that the condition for a function to be supported on $X_{k-1} \setminus \rmL(\Gamma)$ or $X_k \setminus \rmR(\Gamma)$ does not appear in 6. This is because eigenfunctions relative to non-zero eigenvalues of $A_{k-1}^{\Gamma, \up, -}$ or $A_k^{\Gamma, \down, -}$ already satisfy such conditions.
\end{remark}
\begin{remark}\label{remark.restriction_theorem_functions_up_down}
    The decomposition of $X_k = \amalg_{C_k} C_k$ into cover-up-components in dimension $k$ of $\Gamma$, and that of $\ul{X}_k = \amalg_{\ul{C}_k} \ul{C}_k$ into quotient-up-components in dimension $k$ of $\ul{\Gamma}$, induce orthogonal decompositions of the spaces $\scrF_k^\Gamma, \scrF_k^{\Gamma, \sym}, \scrF_k^{\Gamma, \alt}, \scrF_k^{\ul{\Gamma}}$, and $\scrF_k^{\calO(\ul{\Gamma})}$, into their subspaces of functions supported on $C_k, \ul{C}_k$, and $\calO(\ul{C}_k)$. Items 1 to 4 of Theorem \ref{theorem.spectrum_up_down_double_cover_quotient_signed_graph} remain valid when restricting all operators and function spaces to a cover-up-component $C_k$ of $\Gamma$, its corresponding quotient-up-component $\ul{C}_k$ of $\ul{\Gamma}$, and its image $\calO(\ul{C}_k)$ in $\calO(\ul{\Gamma})$. Items 5 and 6 hold true when restricting the operators to a pair of components: in the case of the cover $\Gamma$, a cover-up-component $C_{k-1}$ that is not an involutory pair of leaves and a cover-down-component $C_k$ that is not an involutory pair of roots, associated with each other through the correspondence prescribed by Lemma \ref{lemma.components_of_Gamma}; in the case of the quotient $\ul{\Gamma}$, their corresponding quotient-up-component $\ul{C}_{k-1}$ and quotient-down-component $\ul{C}_k$; in the case of the signed graph $\calO(\ul{\Gamma})$, the images of $\calO(\ul{C}_{k-1})$ and $\calO(\ul{C}_k)$.
\end{remark}
\begin{proof}
    Items 1 to 4 are proven similarly to Theorem \ref{theorem.spectrum_double_cover_quotient_signed_graph}. Items 5 and 6 follow from this general linear algebra fact. Given a (not necessarily square) matrix $M$, consider the symmetric matrices $M^T M$ and $M M^T$. Then, the non-zero spectrum of $M^T M$ coincides with that of $M M^T$. Moreover, $M$ maps eigenvectors of $M^T M$ relative to a non-zero eigenvalue to eigenvectors of $M M^T$ relative to the same eigenvalue (and vice versa, through the matrix $M^T$).
\end{proof}

While root-to-leaf path random walks across different dimensions are aperiodic on each cover-component, this might not be the case when considering conditional random walks. Let us consider, for instance, the up-walk. From Theorem \ref{theorem.spectrum_up_down_double_cover_quotient_signed_graph}, the negative contribution of the spectrum of the up-walk comes from $A_k^{\Gamma, \up, -}|_{\scrF_k^{\Gamma, \alt}}$, or equivalently $A_k^{\calO(\ul{\Gamma}), \up}$. Given a quotient-up-component $\ul{C}_k$ in dimension $k$, we denote by $\scrF_k^{\ul{C}_k}$ and $\scrF_k^{\calO(\ul{C}_k)}$ the subspaces of $\scrF_k^{\ul{\Gamma}}$ and $\scrF_k^{\calO(\ul{\Gamma})}$ of functions supported on $\ul{C}_k$ and $\calO(\ul{C}_k)$, respectively. We also denote by $A_k^{\ul{C}_k, \up}$ the restriction of $A_k^{\ul{\Gamma}, \up}$ to $\scrF_k^{\ul{C}_k}$, and by $A_k^{\calO(\ul{C}_k), \up}$ the restriction of $A_k^{\calO(\ul{\Gamma}), \up}$ to $\scrF_k^{\calO(\ul{C}_k)}$. We also make use of a similar notation in the down-case.
\begin{theorem}\label{theorem.-1_eigenvalue}
    Assume that $\ul{C}_{k-1}$ is a quotient-up-component in dimension $k-1$ that is not a leaf, and let $\ul{C}_k$ be the corresponding quotient-down-component in dimension $k$ prescribed by Lemma \ref{lemma.components_of_Gamma}, which is not a root. Then, the following statements are equivalent:
    \begin{enumerate}
        \item The up-walk is not aperiodic on $C_{k-1}$,
        \item The down-walk is not aperiodic on $C_k$,
        \item The minimal value $-1$ is an eigenvalue of $A_{k-1}^{\calO(\ul{C}_{k-1}), \up}$,
        \item The minimal value $-1$ is an eigenvalue of $A_k^{\calO(\ul{C}_k), \down}$,
        \item $\ul{C}_{k-1}$ is a coherent-up-component,
        \item $\ul{C}_k$ is a coherent-down-component.
    \end{enumerate}
    In this case, the multiplicity of $-1$ is one for both operators. Furthermore, if $\calO$ is chosen according to Remark \ref{remark.adjustment_orientation}, then the corresponding eigenfunction $g$ of $A_{k-1}^{\calO(\ul{C}_{k-1}), \up}$ is given, up to a constant, by
    \[
    g(t) = (\LP(\ul{t}) \cdot \RP(\ul{t}))^{1/2},
    \]
    and the corresponding eigenfunction $f$ of $A_k^{\calO(\ul{C}_k), \down}$, up to a constant, has the same expression
    \[
    f(u) = (\LP(\ul{u}) \cdot \RP(\ul{u}))^{1/2}.
    \]
\end{theorem}
\begin{remark}\label{remark.ping_pong}
    It is intuitive that, if $\ul{C}_{k-1}$ is a coherent-up-component, then the up-walk on $C_{k-1}$ is not aperiodic. Indeed, once a suitable orientation $\calO$ is chosen as in Definition \ref{definition.k-coherent_up_component}, and supposing that the walker is at a node $u \in \calO(\ul{C}_{k-1})$, the rules of the up-walk imply that, after one step, the walker finds themselves at a node in $C_{k-1} \setminus \calO(\ul{C}_{k-1})$, and only after two steps back at a node in $\calO(\ul{C}_{k-1})$.
\end{remark}
\begin{proof}
    First, notice that 3 is equivalent to 4 thanks to Theorem \ref{theorem.spectrum_up_down_double_cover_quotient_signed_graph} (and 5 is equivalent to 6 thanks to Lemma \ref{lemma.components_of_Gamma}). Therefore, we focus on the up-statements, and prove their equivalence. Once we fix $g \in \scrF_{k-1}^{\calO(\ul{C}_{k-1})}$, thanks to Remark \ref{remark.flip_function_calO}, we can always adjust $\calO$ so that $g(t) \ge 0$ for any $t \in \calO(\ul{C}_{k-1})$.
    In the following, when considering a summand prescribed by a choice of $\ul{u}$ and $\ul{v}$ in $\ul{C}_k$, we denote by $t = \calO(\ul{t})$ and $u = \calO(\ul{u})$ the corresponding nodes in $\calO(\ul{C}_{k-1})$ and $\calO(\ul{C}_k)$, respectively. We make use of the Cauchy-Schwartz inequality in order to obtain the following chain of inequalities for the Rayleigh quotient of $g$ with respect to the operator $A_{k-1}^{\calO(\ul{C}_{k-1}), \up}$, or equivalently $A_{k-1}^{\calO(\ul{\Gamma}), \up}$:
    \begin{align*}
        \calR_{A_{k-1}^{\calO(\ul{\Gamma}), \up}}(g) &= \frac{g^T A_{k-1}^{\calO(\ul{\Gamma}), \up} g}{\|g\|^2} \\
        &= - \frac{g^T \big(\delta_{k-1}^{\calO(\ul{\Gamma})}\big)^T \delta_{k-1}^{\calO(\ul{\Gamma})} g}{\|g\|^2} \\
        &= - \frac{\| \delta_{k-1}^{\calO(\ul{\Gamma})} g \|^2}{\|g\|^2} \\
        &= - \frac{\sum_{\ul{u}} H(\ul{u}) \cdot \big(\sum_{\ul{t} \subset \ul{u}} [u : t] \cdot H(\ul{t})^{-1/2} \cdot g(t)\big)^2}{\sum_{\ul{t}} g(t)^2} \\
        &= - \frac{\sum_{\ul{u}} H(\ul{u}) \cdot \big(\sum_{\ul{t} \subset \ul{u}} [u : t] \cdot \LP(\ul{t})^{-1/2} \cdot \RP(\ul{t})^{1/2} \cdot g(t)\big)^2}{\sum_{\ul{t}} g(t)^2}\\
        &\ge - \frac{\sum_{\ul{u}} H(\ul{u}) \cdot \big(\sum_{\ul{t} \subset \ul{u}} [u : t]^2 \cdot \RP(\ul{t})\big) \cdot \big(\sum_{\ul{t} \subset \ul{u}} \LP(\ul{t})^{-1} \cdot g(t)^2\big)}{\sum_{\ul{t}} g(t)^2} \\
        &= - \frac{\sum_{\ul{u}} H(\ul{u}) \cdot \big(\sum_{\ul{t} \subset \ul{u}} \RP(\ul{t})\big) \cdot \big(\sum_{\ul{t} \subset \ul{u}} \LP(\ul{t})^{-1} \cdot g(t)^2\big)}{\sum_{\ul{u}} g(t)^2} \\
        &=- \frac{\sum_{\ul{u}} H(\ul{u}) \cdot \RP(\ul{u}) \cdot \big(\sum_{\ul{t} \subset \ul{u}} \LP(\ul{t})^{-1} \cdot g(t)^2\big)}{\sum_{\ul{t}} g(t)^2} \\
        &=- \frac{\sum_{\ul{u}} \LP(\ul{u}) \cdot \big(\sum_{\ul{t} \subset \ul{u}} \LP(\ul{t})^{-1} \cdot g(t)^2\big)}{\sum_{\ul{t}} g(t)^2} \\
        &=- \frac{\sum_{\ul{u}} \sum_{\ul{t} \subset \ul{u}} \LP(\ul{u}) \cdot \LP(\ul{t})^{-1} \cdot g(t)^2}{\sum_{\ul{t}} g(t)^2} \\
        &=- \frac{\sum_{\ul{t}} \big( \sum_{\ul{u} \supset \ul{t}} \LP(\ul{u})\big) \cdot \LP(\ul{t})^{-1} \cdot g(t)^2}{\sum_{\ul{t}} g(t)^2} \\
        &=- \frac{\sum_{\ul{t}} \LP(\ul{t}) \cdot \LP(\ul{t})^{-1} \cdot g(t)^2}{\sum_{\ul{t}} g(t)^2} \\
        &= -1.
        \vspace{-0.1cm}
    \end{align*}
    The function $g$ is an eigenfunction relative to the eigenvalue $-1$ if and only if the Cauchy-Schwartz inequality above is saturated. This happens if and only if, for any $u \in \calO(\ul{C}_k)$, the vector $([u : t] \cdot \RP(\ul{t})^{1/2})_{t \subset u}$ is proportional to the vector $(\LP(\ul{t})^{-1/2} \cdot g(t))_{t \subset u}$. If this is the case, by connectivity, since $g$ is not everywhere zero, then $g$ is nowhere zero, and therefore $g(t) > 0$ on $\calO(\ul{C}_{k-1})$, and in particular $g(t) = (\LP(\ul{t}) \cdot \RP(\ul{t}))^{1/2}$ up to a positive constant. This also implies that the signature $[u : t]$, as $t \in \calO(\ul{C}_{k-1})$ varies under $u$, does not depend on $t$. In other words, $\ul{C}_{k-1}$ is a coherent-up-component. Vice versa, if $\ul{C}_{k-1}$ is a coherent-up-component, then $g(t) = (\LP(\ul{t}) \cdot \RP(\ul{t}))^{1/2}$ saturates the inequality above. This also implies that the multiplicity of $-1$ is one.
\end{proof}
In conclusion, given a cover-up-component $C_{k-1}$ in dimension $k-1$ (and a cover-down-component $C_k$ in dimension $k$):
\begin{itemize}
    \item If $\ul{C}_{k-1}$ is a leaf ($\ul{C}_k$ is a root), then the up-walk collapses in one step into the uniform stationary probability distribution on $C_{k-1}$ (the down-walk collapses in the same way on $C_k$, respectively);
    \item Assume that $\ul{C}_{k-1}$ is not a leaf, that $\ul{C}_k$ is not a root, and that they correspond to each other as in Lemma \ref{lemma.components_of_Gamma}. If $\ul{C}_{k-1}$ is a coherent-up-component, or equivalently $\ul{C}_k$ is a coherent-down-component, then the up- and down-walks (respectively) are not aperiodic, and almost every initial probability distribution does not converge to the stationary one;
    \item Suppose instead that $\ul{C}_{k-1}$ is not a coherent-up-component (and that the corresponding $\ul{C}_k$ is not a coherent-down-component). We can then obtain a spectral gap theorem for up- and down-walks. Notice that, as pointed out in Remark \ref{remark.automatic_coherent}, both $\ul{C}_{k-1}$ and $\ul{C}_k$ contain at least two quotient nodes each -- otherwise it would be possible to define a coherent orientation. This guarantees that the ``second maximal eigenvalues'' considered below exist. The result is the following:
\end{itemize}
\begin{theorem}\label{theorem.spectral_gap_theorem_conditional}
    Assume that $\ul{C}_{k-1}$ is a quotient-up-component in dimension $k-1$ that is neither a leaf nor coherent, and let $\ul{C}_k$ be the corresponding quotient-down-component in dimension $k$ prescribed by Lemma \ref{lemma.components_of_Gamma}, which is neither a root nor coherent. Then, $1$ is an eigenvalue of $A_{k-1}^{\ul{C}_{k-1}, \up}$, as well as of $A_k^{\ul{C}_k, \down}$, with multiplicity one, so that the second maximal eigenvalue $\lambda_{\max - 1} (A_{k-1}^{\ul{C}_{k-1}, \up})$, which coincides with $\lambda_{\max - 1} (A_k^{\ul{C}_k, \down})$, is lower than $1$. Moreover, the minimal eigenvalue $\lambda_{\min}(A_{k-1}^{\calO(\ul{C}_{k-1}), \up})$, which coincides with $\lambda_{\min}(A_k^{\calO(\ul{C}_k), \down})$, is greater than $-1$. Finally, the up-walk is irreducible and aperiodic on $C_{k-1}$, as well as the down-walk on $C_k$, and the two random processes share the same convergence rate
    \[
    \begin{array}{cccccc}
        &\max \big(
        & \lambda_{\max - 1} \big(A_{k-1}^{\ul{C}_{k-1}, \up}\big)
        &= &\lambda_{\max - 1} \big(A_k^{\ul{C}_k, \down}\big), & \\
        & & -\lambda_{\min} \big(A_{k-1}^{\calO(\ul{C}_{k-1}), \up}\big)
        &= &-\lambda_{\min} \big(A_k^{\calO(\ul{C}_k), \down}\big) &
        \big).
    \end{array}
    \]
\end{theorem}
\begin{remark}\label{remark.convergence_rate}
    We clarify what we mean by convergence rate. For the random processes of interest, under the hypotheses of Theorem \ref{theorem.spectral_gap_theorem_conditional}, for almost every initial probability distribution, the total variation norm of the difference between the $t$-th evolutionary state of the process and the stationary distribution is, up to a constant, asymptotic to $\lambda^t$ for some $\lambda \in (0, 1)$. Such $\lambda$ is the convergence rate of the process.
\end{remark}

\begin{proof}
    This follows from Theorem \ref{theorem.spectrum_up_down_double_cover_quotient_signed_graph}, Theorem \ref{theorem.-1_eigenvalue}, and the spectral gap theorem for Markov Chains \cite[Chapter 5.3]{beichelt2002stochastic}.
\end{proof}

To conclude this section, we present a stronger result describing the relation between the spectra of the operators of interest in the coherent case. This also provides an alternative proof of the fact that coherent-up- and -down-components induce eigenvalues $-1$ of signed operators:
\begin{proposition}\label{proposition.-1_eigenvalue_same_spectrum}
    Assume that $\ul{C}_k$ is a coherent-up-component in dimension $k$ that is not a leaf. Then, there is an identification of the eigenfunctions of $A_k^{\ul{C}_k, \up}$ and those of $A_k^{\calO(\ul{C}_k), \up}$. In particular, counting eigenvalues with their multiplicities, the operators have opposite spectra. The converse is true as well. The result holds when replacing ``up'' by ``down''.
\end{proposition}
\begin{remark}
    This statement has the same flavour as the classic result for bipartite graph: for any eigenvalue of the adjacency matrix of a bipartite graph, its opposite is an eigenvalue as well.
\end{remark}
\begin{proof}
    It is enough to notice that, by comparing their defining expressions, these operators are one the opposite of the other under the coherent hypothesis, by choosing an orientation $\calO$ according to Definition \ref{definition.k-coherent_up_component}.
\end{proof}

\endgroup

\section{Normalized Laplacians and Cheeger inequalities on simplicial complexes}\label{chapter.norm_lapl_cheeger}
\begingroup
\let\section\subsection
\let\subsection\subsubsection
\let\subsubsection\paragraph
In this section, we demonstrate how the normalized coboundary operator and the normalized Laplacian on simplicial complexes naturally emerge as operators associated with root-to-leaf path random walks on simplicial complexes, and prove Cheeger inequalities for the maximal eigenvalues of such normalized Laplacians. We first state some standard facts and show how simplicial complexes are an example of double covers of graded signed graphs (Section \ref{section.facts_notation_simplicial_complexes}), and recall the properties of the standard Hodge Laplacian (Section \ref{section.hodge_laplacian}). Then, we dedicate part of this section to the normalized Laplacian (Section \ref{section.normalised_laplacian}). We also make a note on $(k+1)$-partitions and their relation to topological properties of interest (Section \ref{section.k+1_partition}). Lastly, in Section \ref{section.cheeger_inequalities}, we introduce Cheeger constants and prove Cheeger inequalities for normalized Laplacians on simplicial complexes. In particular, after recalling more classic Cheeger inequalities on graphs (Section \ref{subsection.auxiliary_graphs}), we study the up- (Section \ref{subsection.cheeger_up_case}) and down- (Section \ref{subsection.cheeger_down_case}) cases, merging them into combined Cheeger inequalities (Section \ref{subsection.cheeger_merging_everything}). We conclude our discussion by providing an explicit example to show the effectiveness of such combined Cheeger inequalities (Section \ref{subsection.example_tetrahedron}). \\

For a solid introduction to simplicial complexes, we refer the reader to \cite{spanier1989algebraic}. The review \cite{lim2020hodge} is well-written and effective; although developed for the clique complex of a graph, much of the machinery applies to simplicial complexes. This section draws largely from \cite{atay2020} for Cheeger inequalities on signed graphs and is inspired by \cite{jost2023} for Cheeger inequalities on simplicial complexes. We also cite \cite{trevisan2009max} for their work on the Cheeger inequality for the maximal eigenvalue of the normalized Laplacian on simplicial complexes, along with \cite{steenbergen2014cheeger} and \cite{parzanchevski2016isoperimetric} for non-normalized versions of Cheeger inequalities on simplicial complexes. The normalized Laplacian on simplicial complexes considered in this section was introduced by \cite{horak2013spectra}, and random walks on low-dimensional faces of simplicial complexes were studied in \cite{schaub2020}.


\section{Overview and the associated double cover}\label{section.facts_notation_simplicial_complexes}

\begin{figure}[H]
    \centering
    \includegraphics[scale=0.8]{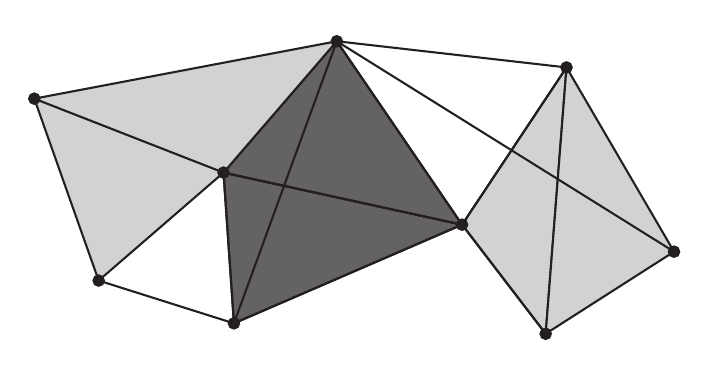}
    \caption{Example of simplicial complex. Notice the use of different scales of gray to differentiate between triangles and tetrahedron.}
    \label{figure.simplicial_complex}
\end{figure}

We introduce simplicial complexes (Figure \ref{figure.simplicial_complex}) and outline how they might be viewed as double covers of graded signed graphs. Typically, a simplicial complex is defined to be its set of faces: given a finite set $\ul{X}_0$ of nodes, or vertices, a set of \textbf{faces} $\ul{X}$ on $\ul{X}_0$ is a set of subsets of $\ul{X}_0$ containing the singletons of $\ul{X}_0$, which we identify with $\ul{X}_0$ itself, and that is closed under sub-inclusion. More explicitly, if $\ul{\varphi}$ is a non-empty subset of $\ul{\varphi'}$ and $\ul{\varphi'} \in \ul{X}$, then $\ul{\varphi} \in \ul{X}$. We choose to assume that $\emptyset \notin \ul{X}$. We will make use of the term simplicial complex (and denote it by $\Gamma$, according to Section \ref{chapter.root_to_leaf}) to refer to $\ul{X}$ together with its accessories, which we describe below. \\

Given $\ul{\varphi} \in \ul{X}$, we define the \textbf{dimension} of $\ul{\varphi}$ to be the cardinality of $\ul{\varphi}$ minus one. Therefore, nodes have dimension $0$, edges have dimension $1$, triangles have dimension $2$, and so on. A \textbf{$k$-face} of $\ul{X}$ is an element of $\ul{X}$ of dimension $k$. We denote by $\ul{X}_k$ the subset of $\ul{X}$ containing the $k$-faces of $\ul{X}$. We also define the \textbf{dimension} of the simplicial complex $\Gamma$ to be the maximal dimension of its faces, and denote it by $\dim(\Gamma)$. Notice that an undirected graph is a simplicial complex of dimension $1$. We denote by $N$ the number of faces of $\ul{X}$, and by $N_k$ the number of $k$-faces in $\ul{X}_k$. \\

A face can be oriented. A $k$-face $\ul{\sigma} = \{\ul{x}_0, \dots, \ul{x}_k\}$ is an unsorted set of $k+1$ nodes. For $k \ge 1$, we say that an \textbf{oriented $k$-face} $\sigma$ is an ordered set $\sigma = [\ul{x}_0, \dots, \ul{x}_k]$, and we refer to $\ul{\sigma}$ as the \textbf{support} of $\sigma$. We consider the orientation to be the same if we rearrange the vertices of $\sigma$ via an even permutation, and we say that the orientation does change if we rearrange them via an odd permutation. We denote by $-\sigma$ the oriented $k$-face supported on the same $\ul{\sigma}$ as $\sigma$, with opposite orientation. An exception is played by nodes $\ul{X}_0$. In this case, for any $\ul{x} \in \ul{X}_0$, we have $x=[\ul{x}]$, and we formally introduce $-x$, another oriented $0$-face supported on $\ul{x}$. We denote the set containing all such $\sigma$, both for $k \ge 1$ and $k = 0$, by $X$, and we see that $X$ contains $2N$ oriented faces. We extend the definition of dimension to $X$, and we denote by $X_k$ the set of oriented $k$-faces, of cardinality $2N_k$. We have a natural quotient map $\ul{ \ . \ } \colon X \to \ul{X}$ that ignores the orientation. On the other hand, choosing an orientation on all faces in $\ul{X}$ corresponds to giving a map $\calO \colon \ul{X} \to X$. \\

Assume that $\sigma, \tau$ are two oriented faces in $X$. We say that $\sigma$ is an \textbf{subface} of $\tau$, and that $\ul{\sigma}$ is a subface of $\ul{\tau}$, if $\ul{\sigma}$ is properly contained in $\ul{\tau}$ as sets. Moreover, we say that $\sigma$ is a \textbf{boundary subface} of $\tau$, and that $\ul{\sigma}$ is a boundary subface of $\ul{\tau}$, if in addition the dimensions of $\tau$ and $\sigma$ differ by $1$. If $\sigma$ is a boundary subface of $\tau$, we write $\sigma \subset \tau$, and $\ul{\sigma} \subset \ul{\tau}$. \\

A fundamental operation on oriented faces of a simplicial complex is taking their boundary. In order to do so, we introduce the modules of \textbf{chains}, on which the boundary operator is defined.
For a given $k$ with $0 \le k \le \dim(\Gamma)$, we denote by $\Z[X_k]$ the $\Z$-module generated by the oriented $k$-faces. Given an orientation, we also denote by $\Z[\calO(\ul{X}_k)]$ the $\Z$-module generated by the oriented $k$-faces in $\calO(\ul{X}_k)$ only. We have a natural map $\Z[X_k] \to \Z[\calO(\ul{X}_k)]$: for every face $\sigma \in \calO(\ul{X}_k)$, the face with opposite orientation $-\sigma$ is identified with $(-1) \cdot \sigma \in \Z[\calO(\ul{X}_k)]$. In other words, this identification makes the multiplication by the scalar $-1$ and taking the opposite orientation equivalent in $\Z[\calO(\ul{X}_k)]$. Even if, technically, the oppositely oriented $-\sigma$ does not belong to $\Z[\calO(\ul{X}_k)]$ when $\sigma$ does, we can identify it with $(-1) \cdot \sigma \in \Z[\calO(\ul{X}_k)]$. Notice that these two elements are not identified in $\Z[X_k]$. \\

With this setup, the \textbf{$k$-th boundary operator} $\partial_k^{\calO(\ul{\Gamma})} \colon \Z[\calO(\ul{X}_k)] \to \Z[\calO(\ul{X}_{k-1})]$ is $\Z$-linear and defined on an oriented $k$-face $\sigma = [\ul{x}_0, \dots, \ul{x}_k]$ as
\begin{equation*}
    \partial_k(\sigma) = \sum_{i=0}^k (-1)^{i+1} \sigma_{-i},
\end{equation*}
where we omit to specify the superscript $\calO(\ul{\Gamma})$. Here, $\sigma_{-i}$ is the oriented $(k-1)$-face $[\ul{x}_0, \dots, \ul{x}_{i-1}, \ul{x}_{i+1}, \dots, \ul{x}_k]$ obtained by removing the $i$-th vertex of $\sigma$. Notice and that the operator $\partial_0$ is the zero operator (nodes have no boundary). To give an example, for an edge $[\ul{x}_0, \ul{x}_1]$, its boundary is given by its ending point minus its starting point, $[\ul{x}_1] - [\ul{x}_0]$. We can then consider a unified boundary operator $\partial \colon \Z[\calO(\ul{X})] \to \Z[\calO(\ul{X})]$ defined as $\partial = \oplus_{k=0}^{\dim(\Gamma)} \partial_k$, with the property that $\partial$ is grading-shifting. A fundamental property of the boundary operator is that a boundary has no boundary, that is, $\partial_{k-1} \partial_k = 0$. Notice that this is true after identifying $-\sigma$ with $(-1) \cdot \sigma$. More concisely, $\partial^2 = 0$. \\

If $\sigma$ is an oriented $k$-face, each of its boundary subfaces $\rho \subset \sigma$, $\rho \in \calO(\ul{X}_{k-1})$, appears, with a positive or negative sign, as a summand in the boundary expression above. The symbol $[\sigma : \rho]$ is defined to be $1$ if the prescribed orientation of $\rho$ is the same as the orientation induced on $\rho$ by taking the boundary of $\sigma$, and $-1$ if they are opposite. More explicitly, if $\rho$ is obtained from $\sigma = [\ul{x}_0, \dots, \ul{x}_k]$ by removing $\ul{x}_i$, then $[\sigma : \rho] = 1$ if the orientation of $(-1)^{i+1} \sigma_{-i}$ and $\rho$ are the same, and $-1$ if they are opposite. We can then rewrite the boundary operator $\partial_k$ as
\[
\partial_k(\sigma) = \sum_{\ul{\rho} \subset \ul{\sigma}} [\sigma : \rho] \cdot \rho.
\]
With all this introduced, we can now see that any simplicial complex $\ul{X}$ produces a double cover $\Gamma = (X, E, [ \ : \ ], -)$ of a strongly graded signed graph, in accordance with Definition \ref{definition.double_covers_graded_signed_graphs} and Definition \ref{definition.strongly_graded}. In particular:
\begin{itemize}
    \item The set of oriented faces $X$ constitutes the set of nodes of the double cover,
    \item The set of edges $E$ is given by the boundary subface condition,
    \item The signature $[ \ : \ ]$ is described through the boundary operator,
    \item The involution $-$ is determined by taking the opposite orientation of an oriented face.
\end{itemize}
We also point out that the notions of involutory quotient (Definition \ref{definition.involutory_quotient}) and orientation (Definition \ref{definition.orientation}) also agree with what discussed so far. \\

It is interesting to describe leaves and roots of a simplicial complex, that is, leaves and roots of the associated double cover, and its leaf- and root-path functions. A leaf has no particular additional characterization: it is a face in $\ul{X}$ that is not properly contained in any other faces. On the other hand, roots are exactly faces of dimension $0$, that is, $\rmR(\ul{\Gamma}) = \ul{X}_0$. The leaf-path function of a simplicial complex, other than Definition \ref{definition.leaf_root_path_function} and Lemma \ref{lemma.number_of_paths_alternative_computation_leaf_root_path_function}, admits a further alternative description:
\begin{lemma}\label{lemma.other_descriprion_leaf_path_function_sc}
The leaf-path function of a simplicial complex can be computed as:
\[
    \LP(\ul{\sigma}) = \sum_{\substack{\ul{\varphi} \in \rmL(\ul{\Gamma}) \\ \ul{\varphi} \text{ contains } \ul{\sigma}}} \big(\dim(\ul{\varphi}) - \dim(\ul{\sigma}) \, \big) !
\]
\end{lemma}
\begin{proof}
    The proof is by reversed induction on $\dim(\ul{\sigma})$. If $\ul{\sigma}$ is a leaf, the result is true. Otherwise, assume that $\ul{\sigma}$ is a $k$-face and not a leaf, and suppose that the result is true for $(k+1)$-faces. Then, using the recursive property defining the leaf-path function, and the inductive hypothesis, we have that
    \begin{align*}
        \LP(\ul{\sigma}) &= \sum_{\ul{\tau} \supset \ul{\sigma}} \LP(\ul{\tau}) \\
        &= \sum_{\ul{\tau} \supset \ul{\sigma}} \sum_{\substack{\ul{\varphi} \in \rmL(\ul{\Gamma}) \\ \ul{\varphi} \text{ contains } \ul{\tau}}} \big(\dim(\ul{\varphi}) - \dim(\ul{\tau}) \, \big) ! \\
        &= \sum_{\ul{\tau} \supset \ul{\sigma}} \sum_{\substack{\ul{\varphi} \in \rmL(\ul{\Gamma}) \\ \ul{\varphi} \text{ contains } \ul{\tau}}} \big(\dim(\ul{\varphi}) - \dim(\ul{\sigma}) - 1 \, \big) ! \\
        &= \sum_{\ul{\tau} \supset \ul{\sigma}} \sum_{\substack{\ul{\varphi} \in \rmL(\ul{\Gamma}) \\ \ul{\varphi} \text{ contains } \ul{\tau}}} \frac{\big(\dim(\ul{\varphi}) - \dim(\ul{\sigma}) \, \big) !}{\dim(\ul{\varphi}) - \dim(\ul{\sigma})} \\
        &= \sum_{\ul{\varphi} \in \rmL(\ul{\Gamma})} \sum_{\substack{\ul{\tau} \supset \ul{\sigma} \\ \ul{\tau} \text{ is contained in } \ul{\varphi}}} \frac{\big(\dim(\ul{\varphi}) - \dim(\ul{\sigma}) \, \big) !}{\dim(\ul{\varphi}) - \dim(\ul{\sigma})} \\
        &= \sum_{\substack{\ul{\varphi} \in \rmL(\ul{\Gamma}) \\ \ul{\varphi} \text{ contains } \ul{\sigma}}} \big(\dim(\ul{\varphi}) - \dim(\ul{\sigma}) \, \big) !
    \end{align*}
    where we used that the number of boundary supfaces $\ul{\tau}$ of $\ul{\sigma}$ contained in $\ul{\varphi}$ is $\dim(\ul{\varphi}) - \dim(\ul{\sigma})$.
\end{proof}

\begin{figure}[H]
    \centering
    \includegraphics[scale=0.8]{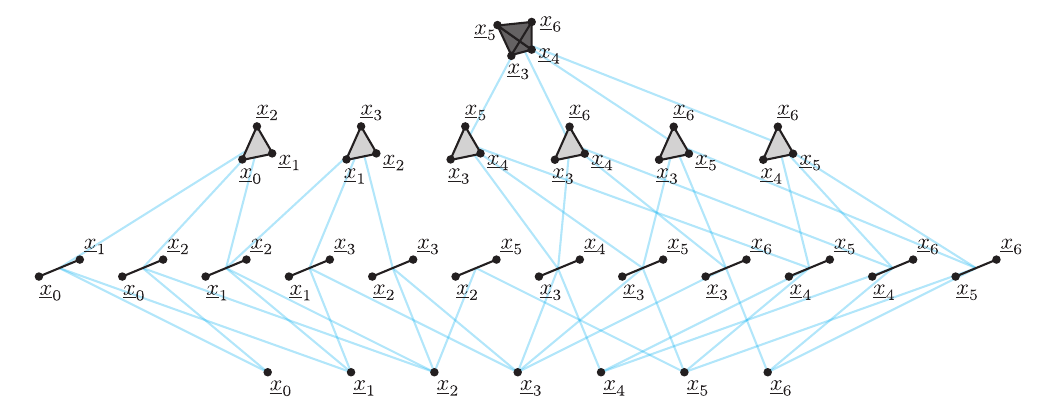}
    \caption{Structure of the quotient $\ul{\Gamma}$ for the simplicial complex whose faces are all non-empty subsets of $\{\ul{x}_0, \ul{x}_1, \ul{x}_2\}, \{\ul{x}_1, \ul{x}_2, \ul{x}_3\}, \{\ul{x}_2, \ul{x}_5\}, \{\ul{x}_3, \ul{x}_4, \ul{x}_5, \ul{x}_6\}$.}
    \label{figure.LP_simplicial_complex}
\end{figure}

\begin{table}[H]
    \centering
    \begin{tabular}{c|c|c||c|c|c}
        {$k$} & {$k$-face $\ul{\sigma}$} & {$\LP(\ul{\sigma})$} & {$k$} & {$k$-face $\ul{\sigma}$} & {$\LP(\ul{\sigma})$} \\
        \hline
        3 & $\{\ul{x}_3, \ul{x}_4, \ul{x}_5, \ul{x}_6\}$ & 1 & 1 & $\{\ul{x}_3, \ul{x}_4\}$ & 2 \\
        2 & $\{\ul{x}_0, \ul{x}_1, \ul{x}_2\}$ & 1 & 1 & $\{\ul{x}_3, \ul{x}_5\}$ & 2 \\
        2 & $\{\ul{x}_1, \ul{x}_2, \ul{x}_3\}$ & 1 & 1 & $\{\ul{x}_3, \ul{x}_6\}$ & 2 \\
        2 & $\{\ul{x}_3, \ul{x}_4, \ul{x}_5\}$ & 1 & 1 & $\{\ul{x}_4, \ul{x}_5\}$ & 2 \\
        2 & $\{\ul{x}_3, \ul{x}_4, \ul{x}_6\}$ & 1 & 1 & $\{\ul{x}_4, \ul{x}_6\}$ & 2 \\
        2 & $\{\ul{x}_3, \ul{x}_5, \ul{x}_6\}$ & 1 & 1 & $\{\ul{x}_5, \ul{x}_6\}$ & 2 \\
        2 & $\{\ul{x}_4, \ul{x}_5, \ul{x}_6\}$ & 1 & 0 & $\ul{x}_0$ & 2 \\
        1 & $\{\ul{x}_0, \ul{x}_1\}$ & 1 & 0 & $\ul{x}_1$ & 4 \\
        1 & $\{\ul{x}_0, \ul{x}_2\}$ & 1 & 0 & $\ul{x}_2$ & 5 \\
        1 & $\{\ul{x}_1, \ul{x}_2\}$ & 2 & 0 & $\ul{x}_3$ & 8 \\
        1 & $\{\ul{x}_1, \ul{x}_3\}$ & 1 & 0 & $\ul{x}_4$ & 6 \\
        1 & $\{\ul{x}_2, \ul{x}_3\}$ & 1 & 0 & $\ul{x}_5$ & 7 \\
        1 & $\{\ul{x}_2, \ul{x}_5\}$ & 1 & 0 & $\ul{x}_6$ & 6 \\
    \end{tabular}
    \caption{Computation of the leaf-path function values $\LP(\ul{\sigma})$ for the simplicial complex in Figure \ref{figure.LP_simplicial_complex}. The values can be obtained via the definition of $\LP$, or Lemma \ref{lemma.number_of_paths_alternative_computation_leaf_root_path_function}, or Lemma \ref{lemma.other_descriprion_leaf_path_function_sc}.}
    \label{table.LP_simplicial_complex}
\end{table}

In Figure \ref{figure.LP_simplicial_complex}, we provide an example of the structure of the quotient $\ul{\Gamma}$ of a simplicial complex, and in Table \ref{table.LP_simplicial_complex} we summarize the values of the leaf path function of said simplicial complex. If we focus the attention on a $k$-face $\ul{\sigma}$, the simplicial structure above $\ul{\sigma}$ might be complicated. However, since every non-empty subset of $\ul{\varphi}$ is a face of the simplicial complex, the simplicial structure is complete below $\ul{\sigma}$. This observation implies that the root-path function takes a simple form in the case of simplicial complexes:

\begin{lemma}\label{lemma.simple_description_root_path_function_sc}
    The root-path function of a simplicial complex takes values $\RP(\ul{\sigma}) = (k+1)!$ on $k$-faces $\ul{\sigma}$.
\end{lemma}
\begin{proof}
    The result follows by induction from Definition \ref{definition.leaf_root_path_function} and the fact that every $k$-face has exactly $k + 1$ boundary subfaces.
\end{proof}
\begin{remark}\label{remark.rp_constant_uniform_probability}
    The fact that the root path function is constant on faces of the same dimension implies that, when walking a down-step of the root-to-leaf path random walk on a simplicial complex, a boundary subface is chosen uniformly.
\end{remark}
We also point out an additional property of coherent-up- and -down-components that is specific to simplicial complexes.
\begin{proposition}\label{proposition.no_k+1_faces}
    Given $k$ with $1 \le k \le \dim(\Gamma)$, assume that $\ul{C}_k$ is a coherent-down-component in dimension $k$ of the simplicial complex $\Gamma$. Then, there are no $(k+1)$-faces over $\ul{C}_k$. More precisely, any $\ul{\sigma} \in \ul{C}_k$ is a leaf.
\end{proposition}
\begin{proof}
    Let $\ul{C}_{k-1}$ be the coherent-up-component in dimension $k-1$ provided by Lemma \ref{lemma.components_of_Gamma}. Assume by contradiction that there exists a $(k+1)$-face $\ul{\tau}$ containing some $k$-face of $\ul{C}_k$. Notice that this implies that all $(k-1)$-faces $\ul{\rho}$ contained in $\ul{\tau}$ are contained in $\ul{C}_{k-1}$. This is because $\ul{C}_{k-1}$ is up-connected, and for any two different $\ul{\rho}, \ul{\rho'}$ both contained in $\ul{\tau}$, $\ul{\rho} \cup \ul{\rho'}$ is a $k$-face (a boundary subface of $\ul{\tau}$). Now, fixing an orientation $\calO$ as in Remark \ref{remark.adjustment_orientation} and only considering oriented faces $\rho, \sigma, \tau$ belonging to $\calO(\ul{X})$, taking twice the boundary of $\tau$ yields
    \begin{align*}
        0 &= \partial_k \partial_{k+1} (\tau) \\
        &= \partial_k \left( \sum_{\ul{\sigma} \subset \ul{\tau}} [\tau : \sigma] \cdot \sigma \right) \\
        &= \sum_{\ul{\rho} \subset \ul{\sigma}} \sum_{\ul{\sigma} \subset \ul{\tau}} [\tau : \sigma] \cdot [\sigma : \rho] \cdot \rho \\
        &= \sum_{\ul{\rho} \text{ contained in } \ul{\tau}} \left( \sum_{\substack{\ul{\sigma} \\ \ul{\rho} \subset \ul{\sigma} \subset \ul{\tau}}} [\tau : \sigma] \right) \cdot \rho.
    \end{align*}
    It follows that, for all $\ul{\rho}$ contained in $\ul{\tau}$,
    \[
    \sum_{\substack{\ul{\sigma} \\ \ul{\rho} \subset \ul{\sigma} \subset \ul{\tau}}} [\tau : \sigma] = 0.
    \]
    Given $\ul{\rho}$ contained in $\ul{\tau}$, there are exactly two $k$-faces $\ul{\sigma}$ such that $\ul{\rho} \subset \ul{\sigma} \subset \ul{\tau}$. Moreover, each pair of $k$-faces $\ul{\sigma} \neq \ul{\sigma'}$, with $\ul{\sigma}, \ul{\sigma'} \subset \ul{\tau}$, identifies a $k$-face $\ul{\rho}$ contained in $\ul{\tau}$ by taking the intersection $\ul{\rho} = \ul{\sigma} \cap \ul{\sigma'}$. This implies that, for any $\ul{\sigma}, \ul{\sigma'} \subset \ul{\tau}$ with $\ul{\sigma} \neq \ul{\sigma'}$, $[\tau : \sigma] = - [\tau : \sigma']$. As the number of $k$-faces $\ul{\sigma} \subset \ul{\tau}$ is $k+2 \ge 3$, this contradicts the choice of orientation $\calO$ specified by Remark \ref{remark.adjustment_orientation}.
\end{proof}

\section{The combinatorial Hodge Laplacian}\label{section.hodge_laplacian}
We present the construction of Hodge Laplacians, whose normalization will naturally arise from the specialization of root-to-leaf path random walks from Section \ref{chapter.root_to_leaf} to simplicial complexes. Having built the bridge between simplicial complexes and double covers, we can import the description of functions on $\Gamma$, whose space splits as
\[
\scrF_k^{\Gamma} = \scrF_k^{\Gamma, \sym} \oplus^\perp \scrF_k^{\Gamma, \alt},
\]
and
\[
\scrF_k^{\Gamma, \sym} \simeq \scrF_k^{\ul{\Gamma}}, \quad \scrF_k^{\Gamma, \alt} \simeq \scrF_k^{\calO(\ul{\Gamma})}.
\]
Functions in $\scrF_k^{\calO(\ul{\Gamma})}$, or equivalently alternating functions, are more natural to consider from a geometric perspective, as they represent the discrete equivalent of differential $k$-forms on differentiable manifolds. There is more to say on this parallelism, and specifically on Laplacians. \\

A function in $\scrF_k^{\calO(\ul{\Gamma})}$ is, equivalently, a $\Z$-linear function $f \colon \Z[\calO(\ul{X}_k)] \to \R$, as discussed above. We can consider the coboundary operator $\partial_k^* \colon \scrF_k^{\calO(\ul{\Gamma})} \to \scrF_{k+1}^{\calO(\ul{\Gamma})}$: given a function $f \colon \Z[\calO(\ul{X}_k)] \to \R$, its \textbf{coboundary} $\partial_k^*(f) \colon \Z[\calO(\ul{X}_{k+1})] \to \R$, or \textbf{differential}, is defined as $\partial_k^*(f) = f \circ \partial_{k+1}$. More explicitly, once an orientation $\calO$ is prescribed, we have
\begin{equation*}
    (\partial_k^* f)(\tau) = \sum_{\ul{\sigma} \subset \ul{\tau}} [\tau : \sigma] \cdot f(\sigma).
\end{equation*}
As usual, $\sigma = \calO(\ul{\sigma})$ and $\tau = \calO(\ul{\tau})$. For instance, $\partial^*_0$ is the gradient of a function, that assigns to a function $f$ on nodes the discrete vector field $\partial^*_0 f$ on edges that takes values $(\partial^*_0 f)([\ul{x}_0, \ul{x}_1]) = f([\ul{x}_1]) - f([\ul{x}_0])$. Notice that, in matrix form, $\partial_k^* = \partial_{k+1}^T$. We can concisely rewrite the operators $\partial_k^* \colon \scrF_k^{\calO(\ul{\Gamma})} \to \scrF_{k+1}^{\calO(\ul{\Gamma})}$ into $\partial^* \colon \scrF^{\calO(\ul{\Gamma})} \to \scrF^{\calO(\ul{\Gamma})}$, where $\partial^* = \oplus_{k=0}^{\dim(\Gamma)} \partial_k^*$ is a grading-shifting operator. Notice that $\partial^* _{\dim(\Gamma)} = 0$. The dual property $\partial_{k+1}^*  \partial_k^* = 0$ can be written as $(\partial^*)^2 = 0$. \\

The \textbf{combinatorial Laplacian} $L$ (we omit the superscript $\calO(\ul{\Gamma})$) of a simplicial complex (as well as its differential analogous on a manifold) is defined as
\begin{equation*}
    L = \left(\partial^* + (\partial^*)^T \right)^2.
\end{equation*}
The transpose matrix form $(\partial^*)^T$ corresponds to the adjoint of $\partial^*$ with respect to standard scalar product on $\scrF^{\calO(\ul{\Gamma})}$. Since $(\partial^*)^2 = 0$, the combinatorial Laplacian is grading-preserving, and can be rewritten as $L = \oplus_{k=0}^{\dim(\Gamma)} L_k$, where $L_k \colon \scrF_k^{\calO(\ul{\Gamma})} \to \scrF_k^{\calO(\ul{\Gamma})}$ is given by
\begin{equation*}
    L_k = \underbrace{(\partial_k^*)^T  \partial_k^*}_{L_k^\up} + \underbrace{\partial_{k-1}^* (\partial_{k-1}^*)^T}_{L_k^\down}.
\end{equation*}
Notice that $L_0^\down = 0$, as well as $L_{\dim(\Gamma)}^\up = 0$. This expression for the combinatorial Laplacian is particularly convenient to identify some of its properties and the meaningful Hodge decomposition of the space of $k$-alternating functions:
\begin{enumerate}
    \item[H1.] $L_k^\up$ and $L_k^\down$ are self-adjoint (or, in matrix form, symmetric) and positive semi-definite.
    \item[H2.] $L_k^\up  L_k^\down = L_k^\down  L_k^\up = 0$. In particular, $L_k^\up$ and $L_k^\down$ are simultaneously diagonalizable, and we have a \textbf{Hodge decomposition}
    \vspace{-0.2cm}
    \begin{table}[H]
        \centering
        \hspace{0.7cm}
        \begin{tabular}{ccccccc}
        $\scrF_k^{\calO(\ul{\Gamma})}$ & $=$ & $\im(L_k^\up)$ & $\oplus^\perp$ & $\im(L_k^\down)$ & $\oplus^\perp$ & $\left( \ker(L_k^\up) \cap \ker(L_k^\down) \right)$ \\[5pt]
        & $=$ & $\im\big((\partial_k^*)^T\big)$ & $\oplus^\perp$ & $\im(\partial_{k-1}^*)$ & $\oplus^\perp$ & $\left( \ker(\partial_k^*) \cap \ker\big((\partial_{k-1}^*)^T\big) \right)$.
        \end{tabular}
    \end{table}
    \vspace{-0.5cm}
    The direct sum is an orthogonal direct sum with respect to the standard inner product on $k$-alternating functions.
    \item[H3.] The non-zero spectra of $L_{k-1}^\up$ and $L_k^\down$, with eigenvalues counted with multiplicities, coincide. Moreover, the coboundary operator $\partial_{k-1}^*$ maps eigenfunctions of $L_{k-1}^\up$ relative to a non-zero eigenvalue to eigenfunctions of $L_k^\down$ of the same eigenvalue (or, equivalently, the adjoint coboundary operator $(\partial_{k-1}^*)^T$ maps eigenfunctions of $L_k^\down$ relative to a non-zero eigenvalue to eigenfunctions of $L_{k-1}^\up$ of the same eigenvalue).
\end{enumerate}
The last term of the Hodge decomposition plays a central role in algebraic topology: it is the space of \textbf{harmonic $k$-functions}
\[
\ker(L_k^\up) \cap \ker(L_k^\down) = \ker(\partial_k^*) \cap \ker\big((\partial_{k-1}^*)^T\big).
\]
Noticing that $\ker((\partial_{k-1}^*)^T) = \im(\partial_{k-1}^*)^\perp$, the natural isomorphism
\begin{equation*}
    \ker(\partial_k^*) \cap \im(\partial_{k-1}^*)^\perp \simeq \frac{\ker(\partial_k^*)}{\im(\partial_{k-1}^*)} = H^k(\Gamma, \R)
\end{equation*}
identifies the space of harmonic $k$-functions with the \textbf{$k$-th cohomology} of the simplicial complex with real coefficients. Again, this fact -- much simpler to verify in this discrete framework -- has its analogous version in the differential context of manifolds. \\

We compute explicitly the action of the up- and down-Laplacians on $k$-alternating functions. Let $f$ be in $\scrF_k^{\calO(\ul{\Gamma})}$. For the up-Laplacian, we define the \textbf{degree} of $\ul{\sigma}$ to be the number of $(k+1)$-supfaces of $\ul{\sigma}$. We denote it by $\deg(\sigma) = \deg(\ul{\sigma})$. Then,
\begin{align*}
    \big(L_k^\up f\big)(\sigma) &= \sum_{\ul{\sigma'}} (L_k^\up)_{\sigma \sigma'} \cdot f(\sigma') \\
    &= \sum_{\ul{\sigma'}} \sum_{\ul{\tau} \supset \ul{\sigma}, \ul{\sigma'}} [\tau : \sigma] \cdot [\tau : \sigma'] \cdot f(\sigma') \\
    &= \deg(\sigma) \cdot f(\sigma) + \sum_{\ul{\sigma'} \sim^\up \ul{\sigma}} s^\up(\sigma, \sigma') \cdot f(\sigma').
\end{align*}
The signature function $s^\up(\sigma, \sigma')$ takes values $\pm 1$. When $\ul{\sigma} \neq \ul{\sigma'}$, if there exists a $\ul{\tau}$ containing both $\ul{\sigma}$ and $\ul{\sigma'}$, then this $\ul{\tau}$ is unique -- and it is $\ul{\tau} = \ul{\sigma} \cup \ul{\sigma'}$. Then, $s^\up(\sigma, \sigma')$ is defined as
\begin{equation*}
    s^\up(\sigma, \sigma') = [\sigma \cup \sigma' : \sigma] \cdot [\sigma \cup \sigma' : \sigma'],
\end{equation*}
where $\sigma \cup \sigma' = \calO(\ul{\sigma} \cup \ul{\sigma'})$. Notice that $s^\up$ only depends on an orientation on $k$-faces, and not on $(k+1)$-faces (or on faces of other dimensions). \\

For the down-Laplacian, we have already noticed that $L_k^\down = 0$ when $k=0$. When $k > 0$, the number of $(k-1)$-subfaces of $\ul{\sigma}$ is $k+1$. Then,
\begin{align*}
    \big(L_k^\down f\big)(\sigma) &= \sum_{\ul{\sigma'}} (L_k^\down)_{\sigma \sigma'} \cdot f(\sigma') \\
    &= \sum_{\ul{\sigma'}} \sum_{\ul{\rho} \subset \ul{\sigma}, \ul{\sigma'}} [\sigma : \rho] \cdot [\sigma' : \rho] \cdot f(\sigma') \\
    &= (k+1) \cdot f(\sigma) + \sum_{\ul{\sigma'} \sim^\down \ul{\sigma}} s^\down(\sigma, \sigma') \cdot f(\sigma').
\end{align*}
Here $s^\down(\sigma, \sigma')$ is defined analogously to the up-case: when $\ul{\sigma} \neq \ul{\sigma'}$, if there exists a $\ul{\rho}$ contained in both $\ul{\sigma}$ and $\ul{\sigma'}$, then this $\ul{\rho}$ is unique -- and it is $\ul{\rho} = \ul{\sigma} \cap \ul{\sigma'}$. Then, $s^\down(\sigma, \sigma')$ is defined as
\begin{equation*}
    s^\down(\sigma, \sigma') = [\sigma : \sigma \cap \sigma'] \cdot [\sigma' : \sigma \cap \sigma'],
\end{equation*}
where $\sigma \cap \sigma' = \calO(\ul{\sigma} \cap \ul{\sigma'})$. Notice that $s^\down$ only depends on an orientation on $k$-faces, and not on $(k-1)$-faces (or on faces of other dimensions).

\section{The normalized Laplacian}\label{section.normalised_laplacian}
A \textbf{normalization} of the Laplacian $L$ of a simplicial complex $\Gamma$ is a choice of diagonal matrices $W_k$, for each $k$ with $0 \le k \le \dim(\Gamma)$, of size $N_k \times N_k$ and positive diagonal entries, that fit into the following desired framework. We think of $W_k$ as indexed by $\sigma \in \calO(\ul{X}_k)$. We define the \textbf{normalized coboundary operator} $\delta_k$ as
\begin{equation}\label{equation.normalised_co-boundary}
    \delta_k = W_{k+1}^{1/2} \partial_k^* W_k^{-1/2}
\end{equation}
and the \textbf{normalized Laplacian}, which splits into its up- and down-components, as
\[
    \Delta_k = \underbrace{\delta_k^T  \delta_k}_{\Delta_k^\up} + \underbrace{\delta_{k-1} \delta_{k-1}^T}_{\Delta_k^\down},
\]
with
\[
    \Delta_k^\up = \delta_k^T  \delta_k = W_k^{-1/2} (\partial_k^*)^T W_{k+1} \partial_k^* W_k^{-1/2}
\]
and
\[
    \Delta_k^\down = \delta_{k-1} \delta_{k-1}^T = W_k^{1/2} \partial_{k-1}^* W_{k-1}^{-1} (\partial_{k-1}^*)^T W_k^{1/2}.
\]
With these constructions, the normalized equivalent of the properties satisfied by the Hodge Laplacian are automatically guaranteed:
\begin{enumerate}
    \item[NH1.] $\Delta_k^\up$ and $\Delta_k^\down$ are self-adjoint (or, in matrix form, symmetric) and positive semi-definite.
    \item[NH2.] $\Delta_k^\up  \Delta_k^\down = \Delta_k^\down  \Delta_k^\up = 0$. In particular, $\Delta_k^\up$ and $\Delta_k^\down$ are simultaneously diagonalizable, and we have a \textbf{normalized Hodge decomposition}
    \vspace{-0.2cm}
    \begin{table}[H]
        \hspace{0.7cm}
        \centering
        \begin{tabular}{ccccccc}
        $\scrF_k^{\calO(\ul{\Gamma})}$ & $=$ & $\im(\Delta_k^\up)$ & $\oplus^\perp$ & $\im(\Delta_k^\down)$ & $\oplus^\perp$ & $\left( \ker(\Delta_k^\up) \cap \ker(\Delta_k^\down) \right)$ \\[5pt]
        & $=$ & $\im\big(\delta_k^T\big)$ & $\oplus^\perp$ & $\im(\delta_{k-1})$ & $\oplus^\perp$ & $\left( \ker(\delta_k) \cap \ker\big(\delta_{k-1}^T\big) \right)$.
        \end{tabular}
    \end{table}
    \vspace{-0.5cm}
    The direct sum is an orthogonal direct sum with respect to the standard inner product on $k$-alternating functions.
    \item[NH3.] The non-zero spectra of $\Delta_{k-1}^\up$ and $\Delta_k^\down$, with eigenvalues counted with multiplicities, coincide. Moreover, the coboundary operator $\delta_{k-1}$ maps eigenfunctions of $\Delta_{k-1}^\up$ relative to a non-zero eigenvalue to eigenfunctions of $\Delta_k^\down$ of the same eigenvalue (or, equivalently, the adjoint coboundary operator $\delta_{k-1}^T$ maps eigenfunctions of $\Delta_k^\down$ relative to a non-zero eigenvalue to eigenfunctions of $\Delta_{k-1}^\up$ of the same eigenvalue).
\end{enumerate}
In addition, we also have that:
\begin{enumerate}
    \item[NH4.] The ranks of $L_k^\up$ and $\Delta_k^\up$ are the same, and the same holds for $L_k^\down$ and $\Delta_k^\down$. From the Hodge decomposition, and its normalized counterpart, it follows that the space of \textbf{normalized harmonic $k$-functions}
    \begin{equation*}
        \ker(\Delta_k^\up) \cap \ker(\Delta_k^\down) = \ker(\delta_k) \cap \ker\big(\delta_{k-1}^T\big)
    \end{equation*}
    has the same dimension of the space of (standard) harmonic $k$-functions, that is the same as the dimension of the $k$-th cohomology of $\Gamma$.
\end{enumerate}
A normalization procedure on simplicial complexes is relevant if it yields properties that, in some sense, generalise those achieved through the normalization procedure for the graph Laplacian. In this regard, we seek a normalization satisfying these further properties:
\begin{enumerate}
    \item[NH5$^0$.] The spectra of $\Delta_k^\up$ and $\Delta_k^\down$ are uniformly bounded from above. Here ``uniformly'' means ``by the same value across all simplicial complexes''.
    \item[NH6$^0$.] The spectral gaps between the maximal eigenvalues of $\Delta_k^\up$ and $\Delta_k^\down$ and such upper bound values are saturated if and only if the simplicial complex admits meaningful topological/combinatorial substructures.
\end{enumerate}
It is natural to seek a process across different dimensions to obtain a normalization of the coboundary operator that is compatible across dimensions, as in Equation \ref{equation.normalised_co-boundary}. From this perspective, root-to-leaf path random walks reveal a normalization procedure that satisfies the above:
\begin{theorem}\label{theorem.normalization}
    Given $k$ with $0 \le k \le \dim(\Gamma)$, we define $W_k$ to be the $N_k \times N_k$ diagonal matrix whose $(\sigma, \sigma)$-entry is the value $H(\ul{\sigma}) = \LP(\ul{\sigma}) / \RP(\ul{\sigma}) = \LP(\ul{\sigma}) / (k+1)!$. This yields a normalization of the Hodge Laplacian of a simplicial complex $\Gamma$. In particular:
    \begin{itemize}
        \item The normalized coboundary operator $\delta_k = W_{k+1}^{1/2} \partial_k^* W_k^{-1/2}$ coincides with the operator $\delta_k^{\calO(\ul{\Gamma})}$ associated with the root-to-leaf path random walk on $\Gamma$,
        \item The normalized up- and down-Laplacians can be expressed in terms of operators associated with the conditional random up- and down-walks, respectively, as
        \[
        \Delta_k^\up = - A_k^{\calO(\ul{\Gamma}), \up}, \qquad \text{and} \qquad \Delta_k^\down = - A_k^{\calO(\ul{\Gamma}), \down}.
        \]
        \item The following properties are satisfied:
        \begin{enumerate}
            \item[NH5.] The spectra of $\Delta_k^\up$ and $\Delta_k^\down$ are uniformly bounded from above by the value $1$.
            \item[NH6.] Suppose that $1 \le k \le \dim(\Gamma)$. Then, the spectral gaps $1 - \lambda_{\max} (\Delta_{k-1}^\up)$ and $1 - \lambda_{\max} (\Delta_k^\down)$ coincide, and they are both saturated if and only if the simplicial complex admits a coherent-up-component in dimension $k-1$, or equivalently a coherent-down-component in dimension $k$. Moreover, the multiplicities of the value $1$ as an eigenvalue of $\Delta_{k-1}^\up$ and of $\Delta_k^\down$ are the same, and they are equal to the number of such components.
        \end{enumerate}
    \end{itemize}
\end{theorem}
\begin{proof}
    This follows from the discussion carried out in Section \ref{chapter.root_to_leaf}, and specifically Theorem \ref{theorem.-1_eigenvalue}.
\end{proof}
Finally, we point out that, similarly to the classic theory developed on graphs, there exist Cheeger inequalities that relate the corresponding spectral gaps to combinatorial measures that quantify how far the simplicial complex is from admitting relevant coherent components; these will be discussed in Section \ref{section.cheeger_inequalities}. \\

To conclude this section, we write down the explicit matrix form of $\delta_k$, $\Delta_k^\up$, and $\Delta_k^\down$:
\[
(\delta_k)_{\tau \sigma} = \begin{cases}
    (k + 2)^{-1/2} \cdot [\tau : \sigma] \cdot \LP(\ul{\tau})^{1/2} \cdot \LP(\ul{\sigma})^{-1/2} & \text{if $\ul{\tau} \supset \ul{\sigma}$,} \\
    0 & \text{otherwise,}
\end{cases}
\]
\[
(\Delta_k^\up)_{\sigma \sigma'} = \begin{cases}
    1 / (k + 2) & \text{if $\sigma = \sigma'$,} \\
    1 / (k + 2) \cdot s^\up(\sigma, \sigma') \\
    \qquad \qquad \cdot \LP(\ul{\sigma})^{-1/2} \cdot \LP(\ul{\sigma'})^{-1/2} \cdot \LP(\ul{\sigma} \cup \ul{\sigma'}) & \text{if $\sigma \sim^\up \sigma'$,} \\
    0 & \text{otherwise,}
\end{cases}
\]
and
\[
(\Delta_k^\down)_{\sigma \sigma'} = \begin{cases}
    1 / (k + 1) \cdot \LP(\ul{\sigma}) \cdot \sum_{\ul{\rho} \subset \ul{\sigma}} \LP(\ul{\rho})^{-1} & \text{if $\sigma = \sigma'$,} \\
    1 / (k + 1) \cdot s^\down(\sigma, \sigma') \\
    \qquad \qquad \cdot \LP(\ul{\sigma})^{1/2} \cdot \LP(\ul{\sigma'})^{1/2} \cdot \LP(\ul{\sigma} \cap \ul{\sigma'})^{-1} & \text{if $\sigma \sim^\down \sigma'$,} \\
    0 & \text{otherwise.}
\end{cases}
\]

\section{A note on partitions}\label{section.k+1_partition}

\begin{figure}[H]
    \centering
    \includegraphics[scale=0.8]{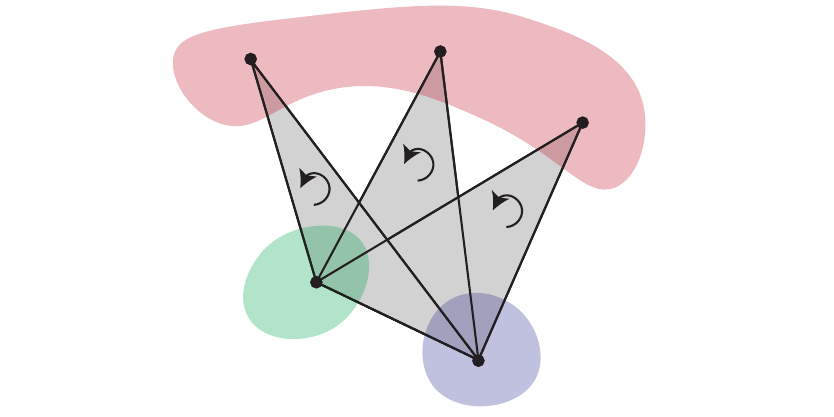}
    \caption{A $(k+1)$-partition induces a coherent-down-component in dimension $k$. In this example, $k=2$. The vertices are partitioned accordingly to the three colors green, blue, and red. The triangles are oriented as described in the proof of Proposition \ref{proposition.k+1_partition}, and form a coherent-down-component in dimension $2$.}
    \label{figure.bipartite_orientation}
\end{figure}

First, we point out that Definition \ref{definition.k-coherent_up_component} and Definition \ref{definition.k-coherent_down_component} of coherent-up- and -down-components extend the notion of bipartiteness on graphs. Indeed, if a graph (for simplicity, connected) is bipartite, then by orienting one class of vertices as $x=[\ul{x}]$ and the second class of vertices as $-x=-[\ul{x}]$, and all the edges from a vertex in the former class to a vertex in the latter, we see that the vertices constitute a coherent-up-component in dimension $0$, while the edges are a coherent-down-component in dimension $1$. The converse is also true on graphs. Assume that the edges of a connected graph constitute a coherent-down-component in dimension $1$. Choose an orientation as specified by Remark \ref{remark.adjustment_orientation}. Then, if a vertex $\ul{x}$ appears as the starting vertex in an oriented edge as $[\ul{x}, \ul{x'}]$ for some other vertex $\ul{x'}$, then $\ul{x}$ appears as starting vertex in all oriented edges containing it. The same is true for vertices appearing as terminal vertices. Therefore, the classes of starting and terminal vertices determine a bipartition of the graph. \\

The situation is more sophisticated on simplicial complexes. While one might expect that a notion of $(k+1)$-partition would be the natural concept to investigate, it turns out to be too strong of a condition. It implies the existence of a coherent-up-component in dimension $k-1$ (or, equivalently, a coherent-down-component in dimension $k$). However, the spectrum of the normalized Laplacian is only able to detect these topological structures, rather than $(k+1)$-partitions. Here, we define $(k+1)$-partitions on simplicial complexes, and show that one can in fact construct relevant structures from such a partition (Figure \ref{figure.bipartite_orientation}).
\begin{definition}\label{definition.k+1_partition}
    Assume that $\ul{C}_k$ is a quotient-down-component in dimension $k$ of a simplicial complex. Denote by $\ul{X}_0(\ul{C}_k)$ the set of vertices that belong to at least one $k$-face in $\ul{C}_k$. A \textbf{$(k+1)$-partition} of $\ul{X}_0(\ul{C}_k)$ is a collection of $k+1$ sets $V_0, \dots , V_k$ such that:
    \begin{enumerate}
        \item $\ul{X}_0(\ul{C}_k)$ is the disjoint union of $V_0, \dots, V_k$,
        \item For each $k$-face $\ul{\sigma} \in \ul{C}_k$, and for each $i = 0, \dots, k$, $\ul{\sigma} \cap V_i$ consists of exactly one vertex.
    \end{enumerate}
\end{definition}
\begin{proposition}\label{proposition.k+1_partition}
    Assume that a quotient-down-component $\ul{C}_k$ in dimension $k$ of a simplicial complex admits a $(k+1)$-partition. Then, $\ul{C}_k$ is a coherent-down-component.
\end{proposition}
\begin{proof}
    Given $\ul{\sigma} \in \ul{C}_k$, define by $\ul{x}_{\ul{\sigma}}(i)$ the only element of $\ul{\sigma} \cap V_i$. Consider the orientation of $\ul{\sigma}$ given by $\sigma = [\ul{x}_{\ul{\sigma}}(0), \dots, \ul{x}_{\ul{\sigma}}(k)]$. Now, given a different $\ul{\sigma'} \in \ul{C}_k$ oriented as $\sigma' = [\ul{x}_{\ul{\sigma'}}(0), \dots, \ul{x}_{\ul{\sigma'}}(k)]$, assume that $\ul{\sigma}$ and $\ul{\sigma'}$ intersect in a $(k-1)$-face. This means that, for all but exactly one index $i$, $\ul{x}_{\ul{\sigma}}(i) = \ul{x}_{\ul{\sigma'}}(i)$. Then, both $\sigma$ and $\sigma'$ induce on the intersection $\ul{\sigma} \cap \ul{\sigma'}$, when taking their boundaries, the same orientation
    \begin{align*}
        (-1)^{i+1} [\ul{x}_{\ul{\sigma}}(0), \dots,  \ul{x}_{\ul{\sigma}}(i - 1), &\ul{x}_{\ul{\sigma}}(i + 1), \ul{x}_{\ul{\sigma}}(k)] \\
        &= (-1)^{i+1} [\ul{x}_{\ul{\sigma'}}(0), \dots,  \ul{x}_{\ul{\sigma'}}(i - 1), \ul{x}_{\ul{\sigma'}}(i + 1), \ul{x}_{\ul{\sigma'}}(k)].
    \end{align*}
    It follows that $\ul{C}_k$ is a coherent-down-component.
\end{proof}

\begin{remark}\label{remark.counterexample_k+1_partition}
    While for graphs the converse of Proposition \ref{proposition.k+1_partition} is also true, on a simplicial complex it is generally not true that the existence of a coherent-down-component in dimension $k$ guarantees the presence of a $(k+1)$-partition of its nodes. We provide a counterexample, drawn in Figure \ref{figure.counterexample_k+1_partition}. Consider the simplicial complex on nodes $\ul{X}_0=\{\ul{x}_0, \dots, \ul{x}_7\}$ obtained by considering all non-empty subsets of the supports of the oriented triangles
    \begin{align*}
        &[\ul{x}_0, \ul{x}_4, \ul{x}_5], [\ul{x}_0, \ul{x}_4, \ul{x}_3], [\ul{x}_7, \ul{x}_4, \ul{x}_3], [\ul{x}_7, \ul{x}_2, \ul{x}_3], \\
        &[\ul{x}_7, \ul{x}_2, \ul{x}_6], [\ul{x}_1, \ul{x}_2, \ul{x}_6], [\ul{x}_1, \ul{x}_5, \ul{x}_6], [\ul{x}_1, \ul{x}_5, \ul{x}_0].
    \end{align*}
	    The specified orientation on these triangles shows that $\ul{X}_2$ is a coherent-down-component in dimension $2$. However, it is not possible to $3$-partition the set of vertices $\ul{X}_0$ of this simplicial complex.
\end{remark}
\begin{figure}[t]
    \centering
    \includegraphics[scale=0.8]{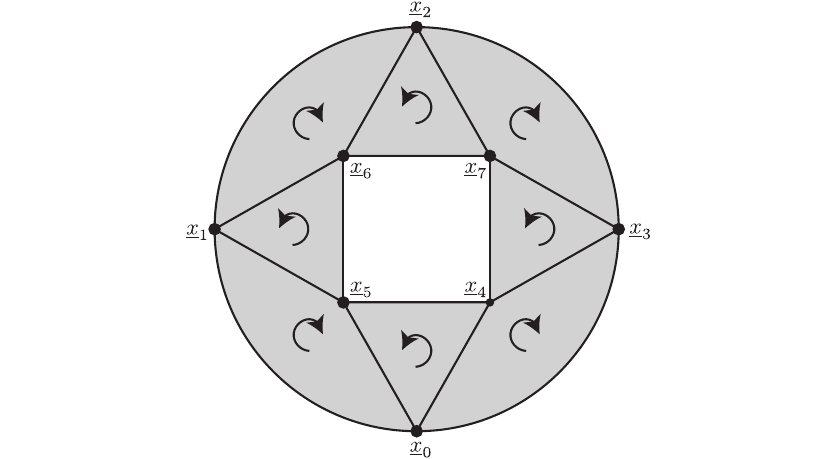}
    \caption{A visualization of the counterexample described in Remark \ref{remark.counterexample_k+1_partition}. Triangles constitute a coherent-down-component in dimension $2$. However, the nodes of the simplicial complex cannot by $3$-partitioned.}
    \label{figure.counterexample_k+1_partition}
\end{figure}
\begin{remark}
	    It is straightforward to deduce Proposition \ref{proposition.no_k+1_faces} for a quotient-down-component $\ul{C}_k$ in dimension $k$ that admits a $(k+1)$-partition: if there existed a $(k+1)$-face $\ul{\tau}$ over $\ul{C}_k$, then two of its $k+2$ nodes would belong to the same class in the partition, and a $k$-boundary subface $\ul{\sigma} \subset \ul{\tau}$ containing them both would belong to $\ul{C}_k$, contradicting the fact that $\ul{C}_k$ is $(k+1)$-partitioned.
\end{remark}

\section{Cheeger inequalities}\label{section.cheeger_inequalities}
After recalling more classic Cheeger inequalities on graphs (Section \ref{subsection.auxiliary_graphs}), we study the up- (Section \ref{subsection.cheeger_up_case}) and down- (Section \ref{subsection.cheeger_down_case}) cases, merging them into combined Cheeger inequalities (Section \ref{subsection.cheeger_merging_everything}). We conclude our discussion by providing an explicit example to show the effectiveness of such combined Cheeger inequalities (Section \ref{subsection.example_tetrahedron}).

\subsection{Auxiliary graphs}\label{subsection.auxiliary_graphs}
The main strategy proposed by \cite{jost2023} for proving Cheeger inequalities on simplicial complexes involves constructing auxiliary graphs based on the up- or down-adjacency of $k$-faces and establishing a relationship between the Laplacians of these objects. We refer the reader to \cite{atay2020} for proofs of Cheeger inequalities on (possibly signed) graphs. In their paper, the authors also develop a theory for multi-way Cheeger constants and inequalities, which provide double bounds on eigenvalues further from the maximal and minimal ends of the spectrum. We note that, although their theory would also apply to our framework, we prefer not to delve into this topic to avoid excessive notational complexity. \\

We will make use of double covers of signed graphs (not graded)
\[
\Gamma^\aux = (X^\aux, E^\aux, s^\aux, -^\aux, \omega^\aux, \mu^\aux),
\]
defined similarly to Section \ref{chapter.root_to_leaf}. The signature function $s^\aux$, the edge structure $E^\aux$, and the involution $-^\aux$ are compatible with each other, as described in Definition \ref{definition.double_covers_graded_signed_graphs}. Two additional ingredients are:
\begin{itemize}
    \item $\omega^\aux \colon \ul{E}^\aux \to \R_{>0}$ is a \textbf{weight function} on (undirected) edges of the quotient, and
    \item $\mu^\aux \colon \ul{X}^\aux \to \R_{>0}$ is a \textbf{measure function} on quotient nodes.
\end{itemize}
As usual, we can also consider an orientation $\calO^\aux \colon \ul{\Gamma}^\aux \to \Gamma^\aux$, and the signed subgraph $\calO^\aux(\ul{\Gamma}^\aux)$ in $\Gamma^\aux$. We will introduce and recall Cheeger inequalities for the quotient Laplacian $\Delta^{\ul{\Gamma}^\aux}$ and the signed Laplacian $\Delta^{\calO^\aux(\ul{\Gamma}^\aux)}$. \\

The symmetric positive semi-definite operator $\Delta^{\ul{\Gamma}^\aux}$ is defined, on functions $f \in \scrF^{\ul{\Gamma}^\aux}$, as
\[
\big(\Delta^{\ul{\Gamma}^\aux} f\big) (\ul{u}) = \frac{1}{\mu^\aux(\ul{u})^{1/2}} \cdot \sum_{\ul{u'} \sim^\aux \ul{u}} \omega^\aux_{\ul{u u'}} \cdot \left( \frac{f(\ul{u})}{\mu^\aux(\ul{u})^{1/2}} - \frac{f(\ul{u'})}{\mu^\aux(\ul{u'})^{1/2}} \right).
\]
More explicitly, the matrix form of this operator is given, for $\ul{u}, \ul{u'} \in \ul{X}^\aux$, by
\[
\Delta^{\ul{\Gamma}^\aux}_{\ul{u u'}} = \begin{cases}
    \mu^\aux(\ul{u})^{-1} \cdot \sum_{\ul{u''} \sim^\aux \ul{u}} \omega^\aux_{\ul{u u''}} & \text{if $\ul{u} = \ul{u'}$}, \\
    - \mu^\aux(\ul{u})^{-1/2} \cdot \mu^\aux(\ul{u'})^{-1/2} \cdot \omega^\aux_{\ul{u u'}} & \text{if $\ul{u} \sim^\aux \ul{u'}$}, \\
    0 & \text{otherwise}.
\end{cases}
\]
We denote by $(\mu^\aux)^{1/2}$ the function in $\scrF^{\ul{\Gamma}^\aux}$ defined as $(\mu^\aux)^{1/2}(\ul{u}) = \mu^\aux(\ul{u})^{1/2}$. Then, $(\mu^\aux)^{1/2}$ is in the kernel of $\Delta^{\ul{\Gamma}^\aux}$, so that $\lambda_{\min} (\Delta^{\ul{\Gamma}^\aux}) = 0$. We assume that the graph $\ul{\Gamma}^\aux$ is connected. In this case, the value $0$ has multiplicity one as an eigenvalue. We recall here the Cheeger inequality for the smallest positive eigenvalue $\lambda_{\min + 1} (\Delta^{\ul{\Gamma}^\aux})$ of $\Delta^{\ul{\Gamma}^\aux}$. Given a subset $\ul{Y}^\aux \subseteq \ul{X}^\aux$ such that $\ul{Y}^\aux \neq \emptyset, \ul{X}^\aux$, we define $\ul{\beta}^\aux(\ul{Y}^\aux)$ as
\[
\ul{\beta}^\aux(\ul{Y}^\aux) = \frac{\omega^\aux \big(\ul{Y}^\aux, \ul{X}^\aux \setminus \ul{Y}^\aux \big)}{\min \big( \mu^\aux(\ul{Y}^\aux), \mu^\aux(\ul{X}^\aux \setminus \ul{Y}^\aux) \big)},
\]
where we define
\[
\omega^\aux \big(\ul{Y}^\aux, \ul{X}^\aux \setminus \ul{Y}^\aux \big) = \sum_{\substack{\ul{u}, \ul{u'} \\ \ul{u} \sim^\aux \ul{u'} \\ \ul{u} \in \ul{Y}^\aux, \ul{u'} \in \ul{X}^\aux \setminus \ul{Y}^\aux}} \omega_{\ul{uu'}},
\]
and
\[
\mu^\aux(\ul{Y}^\aux) = \sum_{\ul{u} \in \ul{Y}^\aux} \mu^\aux(\ul{u}).
\]
Then, the Cheeger constant $\ul{h}^\aux$ is given by
\[
\ul{h}^\aux = \min_{\substack{\ul{Y}^\aux \subseteq \ul{X}^\aux \\ \ul{Y}^\aux \neq \emptyset, \ul{X}^\aux}} \ul{\beta}^\aux (\ul{Y}^\aux).
\]
Notice that $\ul{\beta}^\aux$ and $\ul{h}^\aux$ are underlined to remark that these quantities are associated with the quotient $\ul{\Gamma}^\aux$.
We also define $d_{\mu^\aux}^{\omega^\aux}$ as
\[
d_{\mu^\aux}^{\omega^\aux} = \max_{\ul{u} \in \ul{X}^\aux} \frac{\sum_{\ul{u'} \sim^\aux \ul{u}} \omega_{\ul{u u'}}}{\mu^\aux(\ul{u})}.
\]
\begin{remark}\label{remark.d_mu_omega_zero}
    Notice that $d_{\mu^\aux}^{\omega^\aux}$ might be zero. This happens if and only if there are no edges. Since we are assuming that $\ul{\Gamma}^\aux$ is connected, $d_{\mu^\aux}^{\omega^\aux} = 0$ if and only if $\ul{\Gamma}^\aux$ has only one node (that is, $\Gamma^\aux$ is simply constituted by an involutory pair of isolated nodes). In this case, $\Delta^{\ul{\Gamma}^\aux}$ is the $1 \times 1$ zero matrix.
\end{remark}
Assuming that $\ul{\Gamma}^\aux$ is connected and not reduced to one single node, the Cheeger inequality for the smallest positive eigenvalue of $\Delta^{\ul{\Gamma}^\aux}$ is
\begin{equation}\label{equation.cheeger_auxiliary_quotient}
    \frac{\big(\ul{h}^\aux \big)^2}{2 \cdot d_{\mu^\aux}^{\omega^\aux}} \le \lambda_{\min + 1} (\Delta^{\ul{\Gamma}^\aux}) \le 2 \cdot \ul{h}^\aux.
\end{equation}

Moving to the signed Laplacian, the symmetric positive semi-definite operator $\Delta^{\calO^\aux(\ul{\Gamma}^\aux)}$ is defined, on functions $f \in \scrF^{\calO^\aux(\ul{\Gamma}^\aux)}$, as
\[
\big(\Delta^{\calO^\aux(\ul{\Gamma}^\aux)} f\big) (u) = \frac{1}{\mu^\aux(\ul{u})^{1/2}} \cdot \sum_{\ul{u'} \sim^\aux \ul{u}} \hspace{-0.25cm} \omega^\aux_{\ul{u u'}} \cdot \left( \frac{f(u)}{\mu^\aux(\ul{u})^{1/2}} - s^\aux(u, u') \cdot \frac{f(u')}{\mu^\aux(\ul{u'})^{1/2}} \right).
\]
More explicitly, the matrix form of this operator is given, for $u, u' \in \calO^\aux(\ul{X}^\aux)$, by
\[
\Delta^{\calO^\aux(\ul{\Gamma}^\aux)}_{u u'} = \begin{cases}
    \mu^\aux(\ul{u})^{-1} \cdot \sum_{\ul{u''} \sim^\aux \ul{u}} \omega^\aux_{\ul{u u''}} & \text{if $u = u'$}, \\
    - s^\aux(u, u') \cdot \mu^\aux(\ul{u})^{-1/2} \cdot \mu^\aux(\ul{u'})^{-1/2} \cdot \omega^\aux_{\ul{u u'}} & \text{if $\ul{u} \sim^\aux \ul{u'}$}, \\
    0 & \text{otherwise}.
\end{cases}
\]
In contrast with the quotient scenario, the value $0$ is not necessarily an eigenvalue of the operator $\Delta^{\calO^\aux(\ul{\Gamma}^\aux)}$. Assuming that $\ul{\Gamma}^\aux$ is connected and not reduced to a single node (or, equivalently, that $\Gamma^\aux$ is connected), then $0$ is an eigenvalue of $\Delta^{\calO^\aux(\ul{\Gamma}^\aux)}$ if and only if $\Gamma^\aux$ is balanced, that is: there exists an orientation $\calO^\aux \colon \ul{\Gamma}^\aux \to \Gamma^\aux$ such that $s^\aux (u, u') = 1$ for any $u, u' \in \calO^\aux(\ul{X}^\aux)$. More generally, a Cheeger inequality quantifies the spectral gap $\lambda_{\min} (\Delta^{\calO^\aux(\ul{\Gamma}^\aux)})$. Given a subset $\ul{Y}^\aux \subseteq \ul{X}^\aux$ such that $\ul{Y}^\aux \neq \emptyset$ (possibly $\ul{Y}^\aux = \ul{X}^\aux$), and an orientation $\calO^\aux \colon \ul{\Gamma}^\aux \to \Gamma^\aux$, we define $\beta^\aux (\ul{Y}^\aux, \calO^\aux)$ as
\[
\beta^\aux (\ul{Y}^\aux, \calO^\aux) = \frac{\omega^\aux \big(\ul{Y}^\aux, \ul{X}^\aux \setminus \ul{Y}^\aux \big) + \omega^\aux_{-} (\ul{Y}^\aux, \calO^\aux)}{\mu^\aux(\ul{Y}^\aux)},
\]
where we additionaly define
\[
\omega^\aux_{-} (\ul{Y}^\aux, \calO^\aux) = \sum_{\substack{\ul{u}, \ul{u'} \\ \ul{u} \sim^\aux \ul{u'} \\ \ul{u}, \ul{u'} \in \ul{Y}^\aux \\ s^\aux(u, u') = -1}} \omega_{\ul{uu'}},
\]
where $u, u'$ are the images of $\ul{u}, \ul{u'}$ via $\calO^\aux$ (and therefore $\omega^\aux_{-}$ also depend on $\calO^\aux$). Notice that, given two adjacent nodes $\ul{u}$ and $\ul{u'}$, the pairs of nodes ``first $\ul{u}$, then $\ul{u'}$'' and ``first $\ul{u'}$, then $\ul{u}$'' are considered distinct, and they both contribute as summands. Then, the Cheeger constant $h^\aux$ is given by
\[
h^\aux = \min_{\substack{\ul{Y}^\aux \subseteq \ul{X}^\aux, \calO^\aux \\ \ul{Y}^\aux \neq \emptyset}} \beta^\aux (\ul{Y}^\aux, \calO^\aux).
\]
Notice that $\beta^\aux$ and $h^\aux$ are not underlined to remark that these quantities are associated with possible orientations $\calO^\aux \colon \ul{\Gamma}^\aux \to \Gamma^\aux$ and the signature $s^\aux$, rather than merely the quotient $\ul{\Gamma}^\aux$. The Cheeger inequality for the smallest eigenvalue of $\Delta^{\calO^\aux(\ul{\Gamma}^\aux)}$ is
\begin{equation}\label{equation.cheeger_auxiliary_orientation}
    \frac{\big(h^\aux \big)^2}{2 \cdot d_{\mu^\aux}^{\omega^\aux}} \le \lambda_{\min} \big(\Delta^{\calO^\aux(\ul{\Gamma}^\aux)}\big) \le 2 \cdot h^\aux.
\end{equation}

\subsection{The up-case}\label{subsection.cheeger_up_case}
Given $k \ge 0$, we consider a quotient-up-component $\ul{C}_k$ in dimension $k$ of a simplicial complex that is not a leaf. Notice that such quotient component contains at least two $k$-faces, thanks to the simplicial structure. We construct the following auxiliary graph $\Gamma^{\aux, \up} = (X^{\aux, \up}, E^{\aux, \up}, s^{\aux, \up}, -^{\aux, \up}, \omega^{\aux, \up}, \mu^{\aux, \up})$:
\begin{itemize}
    \item The nodes $X^{\aux, \up}$ are the oriented $k$-faces in $C_k$,
    \item The edges $E^{\aux, \up}$ correspond to the relation of up-adjacency $\sim^\up$ on oriented $k$-faces,
    \item The signature function is $s^{\aux, \up} = s^\up$, that is, $s^\up(\sigma, \sigma') = [\sigma \cup \sigma' : \sigma] \cdot [\sigma \cup \sigma' : \sigma']$ on up-adjacent oriented $k$-faces $\sigma, \sigma'$,
    \item The involution $-^{\aux, \up}$ simply reverses the orientation on oriented $k$-faces,
    \item The weight function $\omega^{\aux, \up} = \omega^\up$ is given, on up-adjacent $k$-faces $\ul{\sigma}, \ul{\sigma'}$, by $\omega^\up_{\ul{\sigma} \ul{\sigma'}} = \LP(\ul{\sigma} \cup \ul{\sigma'})$,
    \item Finally, the measure $\mu^{\aux, \up} = \mu$ is given by $\mu(\ul{\sigma}) = \LP(\ul{\sigma})$.
\end{itemize}
In Figure \ref{figure.auxiliary_graphs}, we provide an example of this construction, together with its down-counterpart. For any $\ul{\sigma} \in \ul{C}_k$, we have
\begin{align*}
    \frac{\sum_{\ul{\sigma'} \sim^\up \ul{\sigma}} \omega^\up_{\ul{\sigma\sigma'}}}{\mu(\ul{\sigma})} &= \frac{\sum_{\ul{\sigma'} \sim^\up \ul{\sigma}} \LP(\ul{\sigma} \cup \ul{\sigma'})}{\LP(\ul{\sigma})} \\
    &= \frac{\sum_{\ul{\tau} \supset \ul{\sigma}} \sum_{\substack{\ul{\sigma'} \subset \ul{\tau} \\ \ul{\sigma'} \neq \ul{\sigma}}} \LP(\ul{\tau})}{\LP(\ul{\sigma})} \\
    &= \frac{(k + 1) \cdot \LP(\ul{\sigma})}{\LP(\ul{\sigma})} \\
    &= k+1.
\end{align*}
It follows that $d_k^\up = d_\mu^{\omega^\up} = k+1$. Following the computation above, we can also simplify the matrix form of the quotient Laplacian of the constructed auxiliary graph:
\[
\Delta^{\ul{\Gamma}^{\aux, \up}}_{\ul{\sigma\sigma'}} = \begin{cases}
    k+1 & \text{if $\ul{\sigma} = \ul{\sigma'}$,} \\
    - \LP(\ul{\sigma})^{-1/2} \cdot \LP(\ul{\sigma'})^{-1/2} \cdot \LP(\ul{\sigma} \cup \ul{\sigma'}) & \text{if $\ul{\sigma} \sim^\up \ul{\sigma'}$}, \\
    0 & \text{otherwise}.
\end{cases}
\]
Denote by $N_{\ul{C}_k}$ the number of $k$-faces in $\ul{C}_k$. Since $\RP(\ul{\sigma}) = (k+1)!$ on $k$-faces $\ul{\sigma}$, and $\RP(\ul{\tau}) = (k+2)!$ on $(k+1)$-faces $\ul{\tau}$, we crucially observe that
\[
\Delta^{\ul{\Gamma}^{\aux, \up}} = (k+2) \cdot I_{N_{\ul{C}_k}} - (k+2) \cdot A^{\ul{C}_k, \up}_k,
\]
or equivalently
\[
I_{N_{\ul{C}_k}} - A^{\ul{C}_k, \up}_k = \frac{\Delta^{\ul{\Gamma}^{\aux, \up}}}{k + 2}.
\]
In particular,
\[
1 - \lambda_{\max - 1} \big( A^{\ul{C}_k, \up}_k \big) = \frac{\lambda_{\min + 1} \big( \Delta^{\ul{\Gamma}^{\aux, \up}} \big)}{k + 2}.
\]
We define the \textbf{$k$-th quotient up-Cheeger constant} $\ul{h}_{C_k}^\up$ as the Cheeger constant $\ul{h}^{\aux, \up}$ associated with the considered auxiliary graph. More explicitly, 
\[
\ul{h}_{C_k}^\up = \min_{\substack{\ul{Y} \subseteq \ul{C}_k \\ \ul{Y} \neq \emptyset, \ul{C}_k}} \ul{\beta}_k^\up (\ul{Y}),
\]
where
\[
\ul{\beta}_k^\up (\ul{Y}) = \frac{\omega^\up \big(\ul{Y}, \ul{C}_k \setminus \ul{Y} \big)}{\min \big( \mu(\ul{Y}), \mu(\ul{C}_k \setminus \ul{Y}) \big)}.
\]
\begin{remark}\label{remark.beta_up_quotient_positive}
    Notice that, since $\ul{C}_k$ is connected, $\ul{\beta}_k^\up (\ul{Y})$ is positive, and so is $\ul{h}_{C_k}^\up$.
\end{remark}
From the above discussion, together with Equation \ref{equation.cheeger_auxiliary_quotient}, we obtain:
\begin{lemma}\label{lemma.cheeger_inequality_quotient_up}
    Given $k \ge 0$ and a quotient-up-component $\ul{C}_k$ in dimension $k$ of a simplicial complex that is not a leaf, we have that
    \[
    \frac{\big(\ul{h}_{C_k}^\up \big)^2}{2 \cdot (k+1) \cdot (k+2)} \le 1 - \lambda_{\max - 1} \big( A^{\ul{C}_k, \up}_k \big) \le \frac{2 \cdot \ul{h}_{C_k}^\up}{k + 2}.
    \]
\end{lemma}
With the same approach, we can analyze the signed Laplacian $\Delta^{\calO^{\aux, \up}(\ul{\Gamma}^{\aux, \up})}$:
\[
\Delta^{\calO^{\aux, \up}(\ul{\Gamma}^{\aux, \up})}_{\ul{\sigma\sigma'}} = \begin{cases}
    k+1 & \text{if $\ul{\sigma} = \ul{\sigma'}$,} \\
    - s^\up(\sigma, \sigma') \\
    \qquad \cdot \LP(\ul{\sigma})^{-1/2} \cdot \LP(\ul{\sigma'})^{-1/2} \cdot \LP(\ul{\sigma} \cup \ul{\sigma'}) & \text{if $\ul{\sigma} \sim^\up \ul{\sigma'}$}, \\
    0 & \text{otherwise}.
\end{cases}
\]
From this, it follows that
\[
\Delta^{\calO^{\aux, \up}(\ul{\Gamma}^{\aux, \up})} = (k + 2) \cdot I_{N_{\ul{C}_k}} + (k + 2) \cdot A_k^{\calO(\ul{C}_k), \up},
\]
or equivalently
\[
I_{N_{\ul{C}_k}} - \big( - A_k^{\calO(\ul{C}_k), \up} \big) = \frac{\Delta^{\calO^{\aux, \up}(\ul{\Gamma}^{\aux, \up})}}{k + 2}.
\]
In particular,
\[
1 - \left( - \lambda_{\min} \big( A_k^{\calO(\ul{C}_k), \up} \big) \right) =
1 - \lambda_{\max} \big( - A_k^{\calO(\ul{C}_k), \up} \big) = \frac{\lambda_{\min} \big( \Delta^{\calO^{\aux, \up}(\ul{\Gamma}^{\aux, \up})} \big)}{k + 2}.
\]
We define the \textbf{$k$-th signed up-Cheeger constant} $h_{C_k}^\up$ as the Cheeger constant $h^{\aux, \up}$ associated with the considered auxiliary graph, that is:
\[
h_{C_k}^\up = \min_{\substack{\ul{Y} \subseteq \ul{C}_k, \calO \\ \ul{Y} \neq \emptyset}} \beta_k^\up (\ul{Y}, \calO),
\]
where
\[
\beta_k^\up (\ul{Y}, \calO) = \frac{\omega^\up \big(\ul{Y}, \ul{C}_k \setminus \ul{Y} \big) + \omega^\up_{-} (\ul{Y}, \calO)}{\mu(\ul{Y})}.
\]
\begin{remark}\label{remark.beta_up_signed_zero}
    Notice that, since $\ul{C}_k$ is connected, $\ul{\beta}_k^\up (\ul{Y}) = 0$ if and only if $\ul{Y} = \ul{C}_k$, and
    \[
    \omega^\up_{-} (\ul{Y}, \calO) = \omega^\up_{-} (\ul{C}_k, \calO) = 0
    \]
    for some orientation $\calO$, that is, $\ul{C}_k$ is a coherent-up-component in dimension $k$. Therefore, $h_{C_k}^\up = 0$ if and only if $\ul{C}_k = 0$ is a coherent-up-component.
\end{remark}
From Equation \ref{equation.cheeger_auxiliary_orientation}, we obtain:
\begin{lemma}\label{lemma.cheeger_inequality_orientation_up}
    Given $k \ge 0$ and a quotient-up-component $\ul{C}_k$ in dimension $k$ of a simplicial complex that is not a leaf, we have that
    \[
    \frac{\big(h_{C_k}^\up \big)^2}{2 \cdot (k+1) \cdot (k+2)} \le 1 - \left( - \lambda_{\min} \big( A_k^{\calO(\ul{C}_k), \up} \big) \right) \le \frac{2 \cdot h_{C_k}^\up}{k + 2}.
    \]
\end{lemma}

\subsection{The down-case}\label{subsection.cheeger_down_case}
We adjust to the down-case the argument carried out in Section \ref{subsection.cheeger_up_case}. Consider $k \ge 0$ and a quotient-down-component $\ul{C}_k$ in dimension $k$ of a simplicial complex that is not a root. We enforce on $\ul{C}_k$ a stronger condition: assume that $\ul{C}_k$ contains at least two $k$-faces. Then, consider the following auxiliary graph $\Gamma^{\aux, \down} = (X^{\aux, \down}, E^{\aux, \down}, s^{\aux, \down}, -^{\aux, \down}, \omega^{\aux, \down}, \mu^{\aux, \down})$:
\begin{itemize}
    \item The nodes $X^{\aux, \down}$ are the oriented $k$-faces in $C_k$,
    \item The edges $E^{\aux, \down}$ correspond to the relation of down-adjacency $\sim^\down$ on oriented $k$-faces,
    \item The signature function is $s^{\aux, \down} = s^\down$, that is, $s^\down(\sigma, \sigma') = [\sigma : \sigma \cap \sigma'] \cdot [\sigma' : \sigma \cap \sigma']$ on down-adjacent oriented $k$-faces $\sigma, \sigma'$,
    \item The involution $-^{\aux, \down}$ simply reverses the orientation on oriented $k$-faces,
    \item The weight function $\omega^{\aux, \down} = \omega^\down$ is given, on down-adjacent $k$-faces $\ul{\sigma}, \ul{\sigma'}$, by
    \[
    \omega^\down_{\ul{\sigma} \ul{\sigma'}} = \frac{\LP(\ul{\sigma}) \cdot \LP(\ul{\sigma'})}{\LP(\ul{\sigma} \cap \ul{\sigma'})},
    \]
    \item Finally, the measure $\mu^{\aux, \down} = \mu$ is given by $\mu(\ul{\sigma}) = \LP(\ul{\sigma})$.
\end{itemize}
In Figure \ref{figure.auxiliary_graphs}, we provide an example of this construction, together with its up-counterpart. For any $\sigma \in \ul{C}_k$,
\begin{align*}
    \frac{\sum_{\ul{\sigma'} \sim^\down \ul{\sigma}} \omega^\down_{\ul{\sigma\sigma'}}}{\mu(\ul{\sigma})} &= \frac{\sum_{\ul{\sigma'} \sim^\down \ul{\sigma}} \frac{\LP(\ul{\sigma}) \cdot \LP(\ul{\sigma'})}{\LP(\ul{\sigma} \cap \ul{\sigma'})}}{\LP(\ul{\sigma})} \\
    &= \sum_{\ul{\rho} \subset \ul{\sigma}} \LP(\ul{\rho})^{-1}
    \cdot \sum_{\substack{\ul{\sigma'} \supset \ul{\rho} \\ \ul{\sigma'} \neq \ul{\sigma}}} \LP(\ul{\sigma'}) \\
    &= \sum_{\ul{\rho} \subset \ul{\sigma}} \LP(\ul{\rho})^{-1} \cdot \left( \LP(\ul{\rho}) - \LP(\ul{\sigma}) \right) \\
    &= k + 1 - \LP(\ul{\sigma}) \cdot \sum_{\ul{\rho} \subset \ul{\sigma}} \LP(\ul{\rho})^{-1} \\
    &\le k + 1.
\end{align*}

\begin{figure}[t]
    \centering
    \includegraphics[scale=0.62]{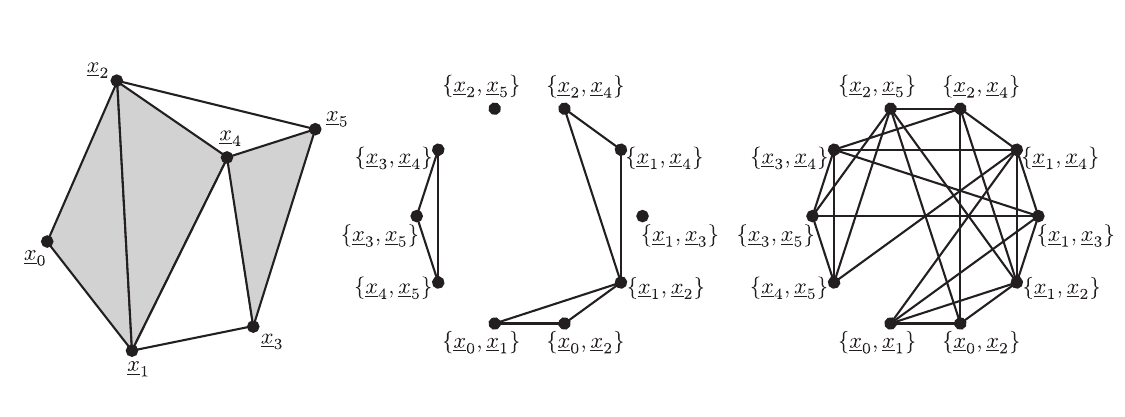}
    \caption{Construction of the auxiliary quotient graphs $\ul{\Gamma}^{\aux, \up}$ (in the middle) and $\ul{\Gamma}^{\aux, \down}$ (on the right) associated with a simplicial complex (on the left), when $k=1$. About $\ul{\Gamma}^{\aux, \up}$: $\{\ul{x}_1, \ul{x}_3\}$ and $\{\ul{x}_2, \ul{x}_5\}$ are isolated nodes in $\ul{\Gamma}^{\aux, \up}$, as they are leaves of $\ul{\Gamma}$; the graph $\ul{\Gamma}^{\aux, \up}$ has two further connected components as, for instance, there is no path of triangles joining the edge $\{\ul{x}_0, \ul{x}_1\}$ to the edge $\{\ul{x}_3, \ul{x}_5\}$. Signature functions, edge weights and node measures are not specified in this figure.}
    \label{figure.auxiliary_graphs}
\end{figure}

It follows that
\[
d_{C_k}^\down = d_\mu^{\omega^\down} = k + 1 - \min_{\ul{\sigma} \in \ul{C}_k} \left( \LP(\ul{\sigma}) \cdot \sum_{\ul{\rho} \subset \ul{\sigma}} \LP(\ul{\rho})^{-1} \right) \le k + 1.
\]
The same computations allow us to simplify the matrix form of the quotient Laplacian of this auxiliary graph:
\[
\Delta^{\ul{\Gamma}^{\aux, \down}}_{\ul{\sigma\sigma'}} = \begin{cases}
    k + 1 - \LP(\ul{\sigma}) \cdot \sum_{\ul{\rho} \subset \ul{\sigma}} \LP(\ul{\rho})^{-1} & \text{if $\ul{\sigma} = \ul{\sigma'}$,} \\
    - \LP(\ul{\sigma})^{1/2} \cdot \LP(\ul{\sigma'})^{1/2} \cdot \LP(\ul{\sigma} \cap \ul{\sigma'})^{-1} & \text{if $\ul{\sigma} \sim^\down \ul{\sigma'}$}, \\
    0 & \text{otherwise}.
\end{cases}
\]
Again, a critical equality is
\[
    \Delta^{\ul{\Gamma}^{\aux, \down}} = (k+1) \cdot I_{N_{\ul{C}_k}} - (k+1) \cdot A^{\ul{C}_k, \up}_k,
\]
or equivalently
\[
   I_{N_{\ul{C}_k}} - A^{\ul{C}_k, \down}_k = \frac{\Delta^{\ul{\Gamma}^{\aux, \down}}}{k + 1}.
\]
In particular,
\[
1 - \lambda_{\max - 1} \big( A^{\ul{C}_k, \down}_k \big) = \frac{\lambda_{\min + 1} \big( \Delta^{\ul{\Gamma}^{\aux, \down}} \big)}{k + 1}.
\]
We define the \textbf{$k$-th quotient down-Cheeger constant} $\ul{h}_{C_k}^\down$ as the Cheeger constant $\ul{h}^{\aux, \down}$ associated with this auxiliary graph, that is:
\[
\ul{h}_{C_k}^\down = \min_{\substack{\ul{Y} \subseteq \ul{C}_k \\ \ul{Y} \neq \emptyset, \ul{C}_k}} \ul{\beta}_k^\down (\ul{Y}),
\]
where
\[
\ul{\beta}_k^\down (\ul{Y}) = \frac{\omega^\down \big(\ul{Y}, \ul{C}_k \setminus \ul{Y} \big)}{\min \big( \mu(\ul{Y}), \mu(\ul{C}_k \setminus \ul{Y}) \big)}.
\]
\begin{remark}\label{remark.beta_down_quotient_positive}
    Notice that, since $\ul{C}_k$ is connected, $\ul{\beta}_k^\down (\ul{Y})$ is positive, and so is $\ul{h}_{C_k}^\down$.
\end{remark}
Equation \ref{equation.cheeger_auxiliary_quotient} then implies:
\begin{lemma}\label{lemma.cheeger_inequality_quotient_down}
    Given $k \ge 0$ and a quotient-down-component $\ul{C}_k$ in dimension $k$ of a simplicial complex containing at least two $k$-faces, we have that
    \[
    \frac{\big(\ul{h}_{C_k}^\down \big)^2}{2 \cdot d_{C_k}^\down \cdot (k+1)} \le 1 - \lambda_{\max - 1} \big( A^{\ul{C}_k, \down}_k \big) \le \frac{2 \cdot \ul{h}_{C_k}^\down}{k + 1}.
    \]
\end{lemma}
For the signed counterpart,
\[
\Delta^{\calO^{\aux, \down}(\ul{\Gamma}^{\aux, \down})}_{\ul{\sigma\sigma'}} = \begin{cases}
    k + 1 - \LP(\ul{\sigma}) \cdot \sum_{\ul{\rho} \subset \ul{\sigma}} \LP(\ul{\rho})^{-1} & \text{if $\ul{\sigma} = \ul{\sigma'}$,} \\
    - s^\down(\sigma, \sigma') \\
    \qquad \cdot \LP(\ul{\sigma})^{1/2} \cdot \LP(\ul{\sigma'})^{1/2} \cdot \LP(\ul{\sigma} \cap \ul{\sigma'})^{-1} & \text{if $\ul{\sigma} \sim^\down \ul{\sigma'}$}, \\
    0 & \text{otherwise}.
\end{cases}
\]
We obtain the equality
\[
\Delta^{\calO^{\aux, \down}(\ul{\Gamma}^{\aux, \down})} = (k + 1) \cdot I_{N_{\ul{C}_k}} + (k + 1) \cdot A_k^{\calO(\ul{C}_k), \down},
\]
or equivalently
\[
I_{N_{\ul{C}_k}} - \big( - A_k^{\calO(\ul{C}_k), \down} \big) = \frac{\Delta^{\calO^{\aux, \down}(\ul{\Gamma}^{\aux, \down})}}{k + 1}.
\]
In particular,
\[
1 - \left( - \lambda_{\min} \big( A_k^{\calO(\ul{C}_k), \down} \big) \right) \hspace{-0.06cm} = \hspace{-0.06cm}
1 - \lambda_{\max} \big( - A_k^{\calO(\ul{C}_k), \down} \big) \hspace{-0.06cm} = \hspace{-0.06cm} \frac{\lambda_{\min} \big( \Delta^{\calO^{\aux, \down}(\ul{\Gamma}^{\aux, \down})} \big)}{k + 1}.
\]
We define the \textbf{$k$-th signed down-Cheeger constant} $h_{C_k}^\down$ as the Cheeger constant $h^{\aux, \down}$ associated with the considered auxiliary graph:
\[
h_{C_k}^\down = \min_{\substack{\ul{Y} \subseteq \ul{C}_k, \calO \\ \ul{Y} \neq \emptyset}} \beta_k^\down (\ul{Y}, \calO),
\]
where
\[
\beta_k^\down (\ul{Y}, \calO) = \frac{\omega^\down \big(\ul{Y}, \ul{C}_k \setminus \ul{Y} \big) + \omega^\down_{-} (\ul{Y}, \calO)}{\mu(\ul{Y})}.
\]
\begin{remark}\label{remark.beta_down_signed_zero}
    Notice that, since $\ul{C}_k$ is connected, $\ul{\beta}_k^\down (\ul{Y}) = 0$ if and only if $\ul{Y} = \ul{C}_k$, and
    \[
    \omega^\down_{-} (\ul{Y}, \calO) = \omega^\down_{-} (\ul{C}_k, \calO) = 0
    \]
    for some orientation $\calO$, that is, $\ul{C}_k$ is a coherent-down-component in dimension $k$. Therefore, $h_{C_k}^\down = 0$ if and only if $\ul{C}_k$ is a coherent-down-component.
\end{remark}
From Equation \ref{equation.cheeger_auxiliary_orientation}, we obtain:
\begin{lemma}\label{lemma.cheeger_inequality_orientation_down}
    Given $k \ge 0$ and a quotient-down-component $\ul{C}_k$ in dimension $k$ of a simplicial complex containing at least two $k$-faces, we have that
    \[
    \frac{\big(h_{C_k}^\down \big)^2}{2 \cdot d_{C_k}^\down \cdot (k+1)} \le 1 - \left( - \lambda_{\min} \big( A_k^{\calO(\ul{C}_k), \down} \big) \right) \le \frac{2 \cdot h_{C_k}^\down}{k + 1}.
    \]
\end{lemma}

\subsection{Merging up- and down-cases}\label{subsection.cheeger_merging_everything}
The Cheeger inequalities obtained in the previous sections can be combined by means of the relation between the spectra of $(k-1)$-up and $k$-down operators described in Theorem \ref{theorem.spectrum_up_down_double_cover_quotient_signed_graph}, producing tighter simultaneous bounds on such operators. We present it here from two perspectives: first, that of root-to-leaf path random walks (Theorem \ref{theorem.spectral_gap_cheeger}); then, that of normalized Laplacians on simplicial complexes (Theorem \ref{theorem.spectral_gap_cheeger_laplacians}). \\

Assume that $\ul{C}_{k-1}$ is a quotient-up-component in dimension $k-1$ that is not a leaf, and let $\ul{C}_k$ its corresponding quotient-down-component in dimension $k$ that is not a root, as in Lemma \ref{lemma.components_of_Gamma}. If one of these two components only contains one quotient node, then, as pointed out in Remark \ref{remark.automatic_coherent}, these components are coherent. In this case, Theorem \ref{theorem.-1_eigenvalue}, Remark \ref{remark.beta_up_signed_zero} and Remark \ref{remark.beta_down_signed_zero} replace the signed Cheeger inequality in the result below: all three terms of the double-inequality are zero. If, instead, $\ul{C}_{k-1}$ and $\ul{C}_k$ contain at least two $(k-1)$- and $k$-faces, respectively, then:
\begin{theorem}\label{theorem.spectral_gap_cheeger}
    Given a simplicial complex $\Gamma$ and $k$ with $1 \le k \le \dim (\Gamma)$, assume that $\ul{C}_{k-1}$ is a quotient-up-component in dimension $k-1$ that is not a leaf, and let $\ul{C}_k$ its corresponding quotient-down-component in dimension $k$ that is not a root, as in Lemma \ref{lemma.components_of_Gamma}. Assume that $\ul{C}_{k-1}$ and $\ul{C}_k$ contain at least two $(k-1)$- and $k$-faces, respectively. Then, the following combined Cheeger inequalities hold:
    \begin{align*}
    \frac{\max \left(
    \frac{\big(\ul{h}_{C_{k-1}}^\up \big)^2}{k},
    \frac{\big(\ul{h}_{C_k}^\down \big)^2}{d_{C_k}^\down}
    \right)}{2 \cdot (k + 1)}
    &\le 1 - \lambda_{\max - 1} \big(A_{k-1}^{\ul{C}_{k-1}, \up}\big) \\
    &= 1 - \lambda_{\max - 1} \big(A_k^{\ul{C}_k, \down}\big) \le
    \frac{ 2 \cdot
    \min \left( \ul{h}_{C_{k-1}}^\up, \ul{h}_{C_k}^\down \right)
    }{k + 1},
    \end{align*}
    as well as
    \begin{align*}
    \frac{\max \left(
    \frac{\big(h_{C_{k-1}}^\up \big)^2}{k},
    \frac{\big(h_{C_k}^\down \big)^2}{d_{C_k}^\down}
    \right)}{2 \cdot (k + 1)}
    &\le 1 - \left( - \lambda_{\min} \big( A_{k-1}^{\calO(\ul{C}_{k-1}), \up} \big) \right) \\
    &= 1 - \left( - \lambda_{\min} \big( A_k^{\calO(\ul{C}_k), \down} \big) \right) \le
    \frac{ 2 \cdot
    \min \left( h_{C_{k-1}}^\up, h_{C_k}^\down \right)
    }{k + 1}.
    \end{align*}
    Finally, assume additionally that these components are not coherent. Then, the convergence rate shared by the up-walk on $C_{k-1}$ and the down-walk on $C_k$ is lower-bounded by
    \[
    1 - \frac{ 2 \cdot
    \max \left( \min \left( \ul{h}_{C_{k-1}}^\up, \ul{h}_{C_k}^\down \right) ,
    \min \left( h_{C_{k-1}}^\up, h_{C_k}^\down
    \right) \right)
    }{k + 1}
    \]
    and upper-bounded by
    \[
    1 -
    \frac{\min \left( \max \left(
    \frac{\big(\ul{h}_{C_{k-1}}^\up \big)^2}{k},
    \frac{\big(\ul{h}_{C_k}^\down \big)^2}{d_{C_k}^\down} \right) ,
    \max \left( \frac{\big(h_{C_{k-1}}^\up \big)^2}{k},
    \frac{\big(h_{C_k}^\down \big)^2}{d_{C_k}^\down}
    \right) \right)}{2 \cdot (k + 1)}.
    \]
\end{theorem}
\begin{proof}
    This follows by specializing Theorem \ref{theorem.spectral_gap_theorem_conditional} and combining the Cheeger inequalities in Lemma \ref{lemma.cheeger_inequality_quotient_up}, Lemma \ref{lemma.cheeger_inequality_quotient_down}, Lemma \ref{lemma.cheeger_inequality_orientation_up}, and Lemma \ref{lemma.cheeger_inequality_orientation_down}.
\end{proof}
Equivalently, we rewrite this result in terms of normalized Laplacians on simplicial complexes:
\begin{theorem}\label{theorem.spectral_gap_cheeger_laplacians}
    Given a simplicial complex $\Gamma$ and $k$ with $1 \le k \le \dim (\Gamma)$, assume that $\ul{C}_{k-1}$ is a quotient-up-component in dimension $k-1$ that is not a leaf, and let $\ul{C}_k$ its corresponding quotient-down-component in dimension $k$ that is not a root, as in Lemma \ref{lemma.components_of_Gamma}. Then, the following combined Cheeger inequality holds for its normalized Laplacians:
    \begin{align*}
    \frac{\max \left(
    \frac{\big(h_{C_{k-1}}^\up \big)^2}{k},
    \frac{\big(h_{C_k}^\down \big)^2}{d_{C_k}^\down}
    \right)}{2 \cdot (k + 1)}
    &\le 1 - \lambda_{\max} \big( \Delta_{k-1}^\up|_{\calO(\ul{C}_{k-1})} \big) \\
    &= 1 - \lambda_{\max} \big( \Delta_k^\down|_{\calO(\ul{C}_k)} \big) \le
    \frac{ 2 \cdot
    \min \left( h_{C_{k-1}}^\up, h_{C_k}^\down \right)
    }{k + 1}.
    \end{align*}
\end{theorem}

\subsection{An explicit example: the tetrahedron}\label{subsection.example_tetrahedron}
We provide a tractable example to show that the combined bounds of Theorem \ref{theorem.spectral_gap_cheeger} are effective, that is, the optimal bounds are at times provided by an up-Cheeger constant, at other times by a down-Cheeger constant. \\

The simplicial complex $\Gamma$ that we consider is fairly simple: given four nodes $\ul{x}_0, \ul{x}_1, \ul{x}_2, \ul{x}_3$, consider the tetrahedron $\{\ul{x}_0, \ul{x}_1, \ul{x}_2, \ul{x}_3\}$ (and all its non-empty subsets). All the values of interest for this limited geometric object can be computed by hand or with the help of a computer. In Table \ref{table.bounds_quotient} and Table \ref{table.bounds_signed}, we summarize the relevant values for quotient and signed Cheeger constants and their associated combined inequalities, respectively. In both cases, we might observe that more competitive bounds can come from both up- or down-Cheeger constants.

\begin{table}[H]
    \centering
    \begin{tabular}{c|c|c|c|c|c|c|c|c}
        & & & & low. & low. & $k$-th & up. & up. \\
        & & & & bound & bound & quotient & bound & bound \\
        $k$ & $d_k^\down$ & $\ul{h}_{X_{k-1}}^\up$ & $\ul{h}_{X_k}^\down$ & $\ul{h}_{X_{k-1}}^\up$ & $\ul{h}_{X_k}^\down$ & sp. gap & $\ul{h}_{X_{k-1}}^\up$ & $\ul{h}_{X_k}^\down$ \\
        \hline
        $1$ & $4/3$ & $2/3$ & $2/3$ & $\mathbf{1/9}$ & $1/12$ & $2/3$ & $\mathbf{2/3}$ & $\mathbf{2/3}$ \\
        $2$ & $3/2$ & $1$ & $1$ & $1/12$ & $\mathbf{1/9}$ & $2/3$ & $\mathbf{2/3}$ & $\mathbf{2/3}$  \\
    \end{tabular}
    \caption{Numerical values relative to the quotient spectral gap in Theorem \ref{theorem.spectral_gap_cheeger} for the tetrahedron. Values that provide the tighter lower and upper bounds are in bold font. Lower and upper bounds provided by $\ul{h}_{X_{k-1}}^\up$ follow from Lemma \ref{lemma.cheeger_inequality_quotient_up}, while those provided by $\ul{h}_{X_k}^\down$ follow from Lemma \ref{lemma.cheeger_inequality_quotient_down}. The $k$-th quotient spectral gap is the spectral gap $1 - \lambda_{\max - 1} (A_{k-1}^{\ul{\Gamma}, \up}) = 1 - \lambda_{\max - 1} (A_k^{\ul{\Gamma}, \down})$.}
    \label{table.bounds_quotient}
\end{table}

\begin{table}[H]
    \centering
    \begin{tabular}{c|c|c|c|c|c|c|c|c}
        & & & & low. & low. & $k$-th & up. & up. \\
        & & & & bound & bound & signed & bound & bound \\
        $k$ & $d_k^\down$ & $h_{X_{k-1}}^\up$ & $h_{X_k}^\down$ & $h_{X_{k-1}}^\up$ & $h_{X_k}^\down$ & sp. gap & $h_{X_{k-1}}^\up$ & $h_{X_k}^\down$ \\
        \hline
        $1$ & $4/3$ & $1/3$ & $4/9$ & $1/36$ & $\mathbf{1/27}$ & $1/3$ & $\mathbf{1/3}$ & $4/9$ \\
        $2$ & $3/2$ & $2/3$ & $1/2$ & $\mathbf{1/27}$ & $1/36$ & $1/3$ & $4/9$ & $\mathbf{1/3}$  \\
    \end{tabular}
    \caption{Numerical values relative to the signed spectral gap in Theorem \ref{theorem.spectral_gap_cheeger} for the tetrahedron. Values that provide the tighter lower and upper bounds are in bold font. Lower and upper bounds provided by $h_{X_{k-1}}^\up$ follow from Lemma \ref{lemma.cheeger_inequality_orientation_up}, while those provided by $h_{X_k}^\down$ follow from Lemma \ref{lemma.cheeger_inequality_orientation_down}. The $k$-th quotient spectral gap is the spectral gap $1 - \lambda_{\max} (A_{k-1}^{\calO(\ul{\Gamma}), \up}) = 1 - \lambda_{\max} (A_k^{\calO(\ul{\Gamma}), \down})$.}
    \label{table.bounds_signed}
\end{table}

\endgroup

\section*{Acknowledgments}
FV was supported by the EPSRC (EP/S021590/1), the EPSRC Centre for Doctoral Training in Geometry and Number Theory (LSGNT), University College London, Imperial College London. TB was supported by a UKRI Future Leaders Fellowship [MR/Y018818/1]. MTS received funding by the European Union (ERC, HIGH-HOPeS, 101039827). Views and opinions expressed are however those of the author(s) only and do not necessarily reflect those of the European Union or the European Research Council Executive Agency. Neither the European Union nor the granting authority can be held responsible for them.

\clearpage
\addcontentsline{toc}{section}{References}
\renewcommand{\refname}{References}
{\small\sloppy\bibliographystyle{unsrtnat}\bibliography{ref}}

\end{document}